\documentclass[10pt,a4paper]{article}
\usepackage[utf8]{inputenc}
\usepackage[american]{babel}
\usepackage[all]{xy}
\usepackage{enumitem}
\usepackage[dvipsnames]{xcolor}
\usepackage{graphicx}
\usepackage{stmaryrd}
\usepackage{amsfonts,amstext,amssymb,amsmath,amsthm,pspicture,bigints}
\usepackage{moreverb}
\usepackage{pifont}
\usepackage{pst-all}
\usepackage[left=2cm,right=2cm,top=2cm,bottom=2cm]{geometry}
\usepackage{esint}
\usepackage{fourier-orns}
\usepackage{mathrsfs}
\usepackage{array}
\usepackage{hyperref}
\usepackage{caption}
\usepackage{subcaption}
\usepackage{authblk}
\usepackage{makecell}
\newcolumntype{C}[1]{>{\centering\arraybackslash}b{#1}}
\newcolumntype{R}[1]{>{\raggedleft\arraybackslash}b{#1}}
\newcolumntype{L}[1]{>{\raggedright\arraybackslash}b{#1}}
\newcolumntype{M}[1]{>{\centering}m{#1}}
\newtheorem{theo}{Theorem}[section]
\newtheorem{defin}{Definition}[section]
\newtheorem{lem}{Lemma}[section]
\newtheorem{prop}{Proposition}[section]

\newtheorem{cor}{Corollary}[section]
\newtheorem{rem}{Remark}[section]
\usepackage{bbm}
\numberwithin{equation}{section}
\newcommand{\ii }{{\rm i} }

\DeclareMathOperator{\T}{\mathbb{T}}
\DeclareMathOperator{\C}{\mathbb{C}}

\date{}
\title{On stationary Quasi-Geostrophic Shallow-Water flows}
\author{Vittorio Baroncini\qquad Claudia Garc\'ia\qquad Emeric Roulley}

\begin{document}
\maketitle

\begin{abstract}
In this paper, we prove the existence of $\mathbf{m}$-fold doubly-connected stationary vortex patches for the quasi-geostrophic shallow-water equations. The solutions are obtained through a bifurcation analysis based on the Crandall--Rabinowitz theorem, with either the inner radius of an annulus or the Rossby deformation length serving as the bifurcation parameter. A central feature of the work is the highly nontrivial analysis of modified Bessel functions arising in the spectral study of the linearized operator. The proof requires delicate and extensive manipulations of these special functions, including precise asymptotic expansions, differentiation formulas, recurrence identities, monotonicity properties and the analysis of singular quantities governing the bifurcation mechanism. These ingredients are essential for characterizing the bifurcation points and establishing the transversality conditions. Finally, we investigate the radial symmetry of stationary and uniformly rotating simply-connected vortex patch solutions, therefore motivating the previous bifurcation results.
\end{abstract}
\tableofcontents
\section{Introduction}
In this introductive section, we present the atmospheric model of interest in this study, some related literature, our main contributions and the key steps of their proofs.

\subsection{Context and main results}
The Quasi-Geostrophic Shallow-Water equations (QGSW) is a well-known model used to describe the dynamics of the atmospheric and oceanic circulation at large scales, see \cite[p.220]{V17}. It is derived from the rotating Shallow-Water equations in the limit of fast rotation compared to the variation of the free surface, namely for small Rossby number. The model takes the form of an active scalar equation satisfied by the potential vorticity $q$,
\begin{equation}\label{QGSW:eq}
    \begin{cases}
        \partial_tq+u\cdot\nabla q=0, & \textnormal{in }\mathbb{R}_+\times\mathbb{R},\\
        u=\nabla^{\perp}(\Delta-\lambda^2)^{-1}q & \textnormal{in }\mathbb{R}_+\times\mathbb{R}, \\
        q(0, \cdot) = q_0 & \textnormal{in } \mathbb{R},
    \end{cases}\qquad\nabla^{\perp}\triangleq\begin{pmatrix}
        -\partial_{x_2}\\
        \partial_{x_1}
    \end{pmatrix}.
\end{equation}
The vector field $u$ is the velocity field of the fluid that is solenoidal. The external parameter $\lambda>0$ is the inverse Rossby radius linked to the Coriolis frequency $\omega_c$, the gravity constant $g$ and the mean active layer depth $H$ through the relation
$$\lambda\triangleq\frac{\omega_c}{\sqrt{gH}}\cdot$$
Note that one recovers the classical 2D Euler equation by setting $\lambda = 0$ in \eqref{QGSW:eq}, which physically corresponds to neglecting the Coriolis force and rotational effect. The corresponding stream function $\psi$ to the divergence-free velocity field $u$ satisfies
$$u=\nabla^{\perp}\psi,\qquad(\Delta-\lambda^2)\psi=q.$$
Inverting the above Helmholtz equation gives the following integral representation of convolution-type for the stream function
\begin{equation*}
    \psi(t,x)=-\frac{1}{2\pi}\int_{\mathbb{R}^2}K_0\big(\lambda|x-y|\big)q(t,y)dy,
\end{equation*}
where $K_0$ is the modified Bessel function of the second kind. Such special functions will be of constant use throughout this work and we refer to Appendix \ref{appendix Bessel} which gathers some of their properties used in our analysis. The QGSW dynamics has been extensively studied in the physics literature \cite{JD20,PD11,PD13,PD14,P91,PZF89}, while its mathematical analysis has been developed more recently; see, for example, \cite{DHR19,HH23,HR21,HM25,MMO26,R21,YZ24}. The present paper contributes to this line of research by providing new mathematical results on the QGSW dynamics. The Yudovich theory for the system \eqref{QGSW:eq} has been recently rigorously justified in \cite{MMO26}. This allows to consider vortex patches, namely vorticity distributions initially uniformly concentrated on a bounded region $D_0 \subset \mathbb{R}^2$, namely
$$q_0=\mathbf{1}_{D_0}.$$ 
Since the vorticity is transported along particle trajectories, the patch structure is preserved throughout the time evolution. More precisely, $$q(t,\cdot)=\mathbf{1}_{D_t},\qquad D_t\triangleq X_{t}(D_0),$$
where $t\mapsto X_t$ is the flow map associated with the velocity field $u$ defined by the ordinary differential equation
$$\partial_tX_t(x)=u\big(t,X_t(x)\big),\qquad X_0(x)=x.$$
In the case of smooth boundary, the problem of finding vortex patch solutions for the system \eqref{QGSW:eq} reduces to solving the so-called contour dynamics equations. This means that the dynamics of the patch is completely understood by knowing the evolutions of its boundary. More precisely, if we assume that each connected component of the boundary of the domain $D_t$ is parametrized by a smooth mapping $z(t, \cdot) : [0, 2\pi) \to \partial D_t$, then the function $z$ must satisfy the following kinematic condition 
\begin{equation} \label{CDE}
\forall\,(t,\vartheta)\in\mathbb{R}_+\times[0,2\pi),\quad\Big[\partial_t z(t,\vartheta)-u\big(t,z(t,\vartheta)\big)\Big]\cdot\vec{n}\big(t,z(t,\vartheta)\big)=0,
\end{equation}
where $\vec{n}\big(t,z(t,\vartheta)\big)$ denotes the outward unit normal vector to $\partial D_t$ at the point $z(t,\vartheta)$. This equation simply states that the interface moves with the fluid, meaning that no fluid transport occurs across the boundary.\\

The case $\lambda=0$ corresponds to the classical 2D Euler equations and therefore has been much more studied. We begin by recalling the main results on vortex patch dynamics in this classical setting and relevant for our analysis. Due to the invariance of the equations under rotations, any radial profile generates a stationary solution. In the context of vortex patches, this translates into a radial shape of the domain $D_0$ and $D_t=D_0$ for all $t\geqslant0$. This includes discs, annuli or nested annuli. In particular, when the domain is the unit disc $\mathbb{D}$, the corresponding solution is called Rankine vortex. In 1874, Kirchhoff \cite{K74} discovered the first non-trivial example of vortex patch solutions corresponding to uniformly rotating ellipses. More precisely, if $D_0$ is an ellipse with semi-axes $a$ and $b$, then $D_t$ is given by a rigid rotation of the initial ellipse $D_0$ around its barycenter with explicit uniform angular velocity $\Omega=\frac{ab}{(a+b)^2}$.  More generally, vortex patches that preserve their shape during a uniform rotation are called V-states. By construction, they satisfy for some real number $\Omega\in\mathbb{R}$, called angular velocity, the property
$$\forall\,t\geqslant0,\quad D_t=e^{\ii\Omega t}D_0.$$
Note that stationary vortex patch solutions are just V-states with $\Omega=0$. In 1978, Deem \& Zabusky \cite{DZ78} numerically observed the existence of simply-connected V-states with $\mathbf{m}$-fold symmetry (i.e. invariance under a rotation of angle $\frac{2\pi}{\mathbf{m}})$. The rigorous mathematical justification of this fact was given in 1982 by Burbea \cite{B82}, by means of a bifurcation argument from the Rankine vortex. Bifurcations occur at the angular velocities $\Omega_\mathbf{m}\triangleq \frac{\mathbf{m}-1}{2 \mathbf{m}}$, $\mathbf{m} \geqslant 2$, each of which gives rise to a branch of $\mathbf{m}$-fold symmetric vortex patches. In 2013, Hmidi, Mateu \& Verdera, not only revisited and partially corrected Burbea's proof, but also established the $C^\infty$-regularity of boundary of the bifurcated V-states near the disc. These results show that the Rankine vortex is not an isolated V-state for angular velocities $\Omega\in\left(0,\frac{1}{2}\right)$. In contrast the situation changes completely outside that interval, where the implicit function theorem implies that no non-trivial V-states bifurcate from the disc in a neighborhood of it. This a priori does not exclude the possibility of other rotating solutions, which are not obtained as perturbations of the disc. However, when additional geometric assumptions are imposed on the initial domain, stronger rigidity phenomena emerge. In particular, Fraenkel \cite{F00} proved that any bounded simply-connected stationary patch with boundary of class $C^1$ must be a disc. This rigidity was subsequently extended by Hmidi \cite{H15}, who showed - using the moving plane method - that a simply-connected bounded $C^1$ domain is necessarily a disc if its angular velocity takes the critical value $\Omega = \frac{1}{2}$ or if $\Omega < 0$, under the extra hypothesis that the domain is convex. Finally, we mention the work by Gómez-Serrano, Park, Shi \& Yao \cite{GPSY20}, where they proved that any uniformly rotating bounded domain with $C^1$ boundary must be radially symmetric whenever $\Omega \in (-\infty, 0) \cup \left[\frac{1}{2}, +\infty\right)$, or radially symmetric up to translation if $\Omega = 0$. We mention that further extends to more general vorticity distributions have been claimed in the recent preprint \cite{FWZ25}. We also refer to \cite{HMW20,H25,P22} for more description on the local and global bifurcation diagram of the V-states. In the stationary case, the assumption that the vorticity is non-negative is essential for the rigidity result. Indeed, it was shown by Gómez-Serrano, Park \& Shi \cite{GPS21} that by allowing an arbitrarily small
portion of negative vorticity, one can find a non-radial stationary vortex patch. We refer the interested reader to \cite{EH26,EHSX24,EFR24,EFR25,GXX26} for further recent important analysis of stationary/steady solutions for 2D Euler equations.

A natural question is whether similar bifurcation structures persist in more general configurations, in particular for doubly-connected domains such as annuli. For the 2D Euler equations, Hmidi, de la Hoz, Mateu \& Verdera \cite{HHMV16} proved that, for every positive integer $\mathbf{m}$ satisfying the condition 
\begin{equation}\label{cond:DC-intro}
    1+b^{\mathbf{m}}-\frac{1-b^2}{2}\mathbf{m}<0,
\end{equation}
one can find a curve of non-annular $\mathbf{m}$-fold symmetric doubly-connected V-states bifurcating from the renormalized annulus
\begin{equation}\label{def:annulus}
    \mathbb{A}_b\triangleq\big\{z\in\mathbb{C}\quad\textnormal{s.t.}\quad b<|z|<1\big\},\qquad b\in(0,1),
\end{equation}
at each of the angular velocities
$$\Omega_\mathbf{m}^\pm(b) = \frac{1-b^2}{4} \pm \frac{1}{2\mathbf{m}} \sqrt{\left(\frac{\mathbf{m}\left(1-b^2\right)}{2} - 1\right)^2 - b^{2\mathbf{m}}}.$$
The degenerate case where the relation \eqref{cond:DC-intro} exactly vanishes has been partially treated in \cite{HM16-2,WXZ22}. We mention that the previous results have been extended to the case where the Euler equations are set in the unit disc: V-states both simply and doubly-connected patch domains \cite{HHHM15} and rigidity results \cite{FWZ24}.\\

Later, the construction of V-states has been extended to the generalized Surface Quasi-Geostrophic equations (gSQG) which is an active scalar equation where the Green kernel is proportional to $|x|^{-\alpha}$ for $\alpha\in(0,2).$ In the gSQG context, the landscape closely mirrors that of the classical Euler setting. In the simply-connected framework, it was demonstrated by Hmidi \& Hassainia \cite{HH15} (for $\alpha\in(0,1)$) and by Castro, Córdoba \& Gómez-Serrano \cite{CCG16} (for $\alpha\in[1,2)$), that, for every $\mathbf{m}\geqslant2$, there exists a branch of $\mathbf{m}$-fold symmetric V-states bifurcating from the unit disc $\mathbb{D}$ at some angular velocity related to the Gamma function. 
The regularity of the boundary of the bifurcating patches was further improved by Castro, Córdoba \& G\'omez-Serrano in \cite{CCG16}, where it was proved that these boundary are in fact analytic. For doubly-connected domains, de la Hoz, Hassainia \& Hmidi \cite{HHH16} (for $\alpha\in(0,1)$) and Renault \cite{R17} (for $\alpha=1$) established, for every $\mathbf{m}$ sufficiently large depending on $b$, the existence of $\mathbf{m}$-fold symmetric doubly-connected V-states bifurcating from the annulus $\mathbb{A}_b$ at the angular velocities 
$$\Omega_{\mathbf{m}}^{\alpha, \pm}(b) = \frac{1-b^2}{2} \Lambda_1(b) + \frac{1}{2} \left(1-b^{-\alpha}\right)\Theta_\mathbf{m} \pm \frac{1}{2} \sqrt{\left[\left(b^{-\alpha}+1\right) \Theta_\mathbf{m} - \left(1+b^2\right)\Lambda_1(b)\right]^2-4b^2\Lambda_{\mathbf{m}}^2(b)},$$
with 
$$\Lambda_n(b)\triangleq\frac{1}{b} \int_0^\infty J_n(bt)J_n(t) \, \frac{dt}{t^{1-\alpha}}, \qquad \Theta_n\triangleq\Lambda_1(1)-\Lambda_n(1),$$
and where $J_n$ denotes the Bessel function of the first kind. In their work, de la Hoz, Hassainia and Hmidi carried out numerical simulations indicating the existence of non-trivial V-states with exactly zero angular velocity, in a neighborhood of the annular configuration. This open question was subsequently fully resolved by Gómez-Serrano. Indeed, in \cite{G19} he presented the first non-trivial construction of analytic stationary patch solutions for the gSQG equations provided $\alpha\in(0,1)$, by making use of the inner radius $b$ as a bifurcation parameter. This work and its approach have provided the principal motivation for our main result Theorem \ref{thm stationary Vstates QGSW} for the QGSW model.\\

We move now to the QGSW equations. In 2019, Dritschel, Hmidi \& Renault \cite{DHR19} implemented the perturbative approach near the circular Rankine vortex, revealing that bifurcation curves branch out from the trivial family at the discrete angular velocities
\begin{equation} \label{bifpt:SC}
    \Omega_{\mathbf{m}}(\lambda) \triangleq I_1(\lambda) K_1(\lambda) - I_\mathbf{m}(\lambda) K_\mathbf{m}(\lambda),
\end{equation}
where $I_n$ and $K_n$ are the modified Bessel functions. We refer the reader to Appendix \ref{appendix Bessel} for the definition of such special functions and some main properties that will be used along the manuscript. Second, this framework was extended by Roulley \cite{R21} to the doubly-connected setting near the renormalized annulus $\mathbb{A}_b$. In this scenario, the bifurcation angular velocities are given by the two expressions
\begin{align*}
    \Omega_{\mathbf{m}}^{\pm}(\lambda,b)&\triangleq\frac{1-b^2}{2b}\Lambda_1(\lambda,b)+\frac{1}{2}\big(\Omega_{\mathbf{m}}(\lambda)-\Omega_{\mathbf{m}}(\lambda b)\big)\\
    &\qquad\qquad\pm\frac{1}{2b}\sqrt{\Big(b\big(\Omega_{\mathbf{m}}(\lambda)+\Omega_{\mathbf{m}}(\lambda b)\big)-(1+b^2)\Lambda_1(\lambda,b)\Big)^2-4b^2\Lambda_{\mathbf{m}}^2(\lambda,b)},
\end{align*}
where $\Omega_{\mathbf{m}}$ is defined in \eqref{bifpt:SC} and
\begin{equation} \label{def_lambda_m_intro}
    \Lambda_{\mathbf{m}}(\lambda,b)\triangleq I_{\mathbf{m}}(\lambda b)K_{\mathbf{m}}(\lambda).
\end{equation}
This latter result provides the core impetus for the present work. As a matter of fact, numerical simulations suggest that for certain values of the parameters $\lambda$ and $b$ these angular velocities $\Omega_\mathbf{m}^\pm(\lambda, b)$ may vanish (see Figure \ref{CV:eigen}), a phenomenon that hints at the existence of non-trivial stationary solutions. Crucially, however, such vanishing of the bifurcation frequencies doesn't occur in the Eulerian case ($\lambda = 0$): for every $\mathbf{m} \geqslant 2$ and $b \in (0,1)$, it holds $\Omega_{\mathbf{m}}^{\pm}(0, b) >0$ (see Figure \ref{CV:eigen}). Hence, bifurcation theory cannot produce non-trivial stationary doubly-connected V-states in the Euler setting. This is fully consistent with the strong rigidity properties of the two-dimensional Euler equations discussed above. Let us mention the extension of V-states to the two-layer QGSW model \cite{HH23}. See \cite{GHR23,GHM22,GHM23,GHM24,GHS20,HHRZ25,HXX26,HXX26-1,R23} for the existence of V-states or more general uniformly rotating solutions for other fluid models. 

\begin{figure}[!ht]
\begin{subfigure}[l]{0.3\textwidth}
		\includegraphics[width=\textwidth]{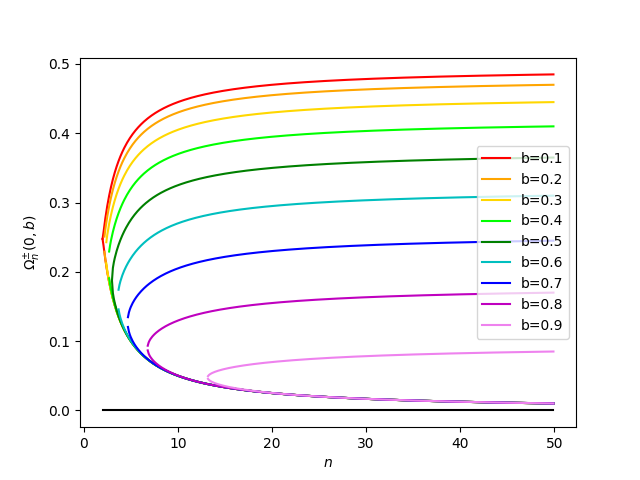}
	\end{subfigure}
	\centering
	\begin{subfigure}[c]{0.3\textwidth}
		\includegraphics[width=\textwidth]{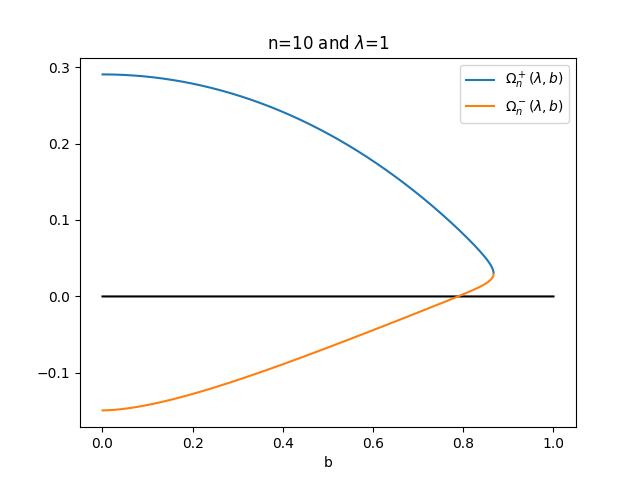}
	\end{subfigure}
	\begin{subfigure}[r]{0.3\textwidth}
		\includegraphics[width=\textwidth]{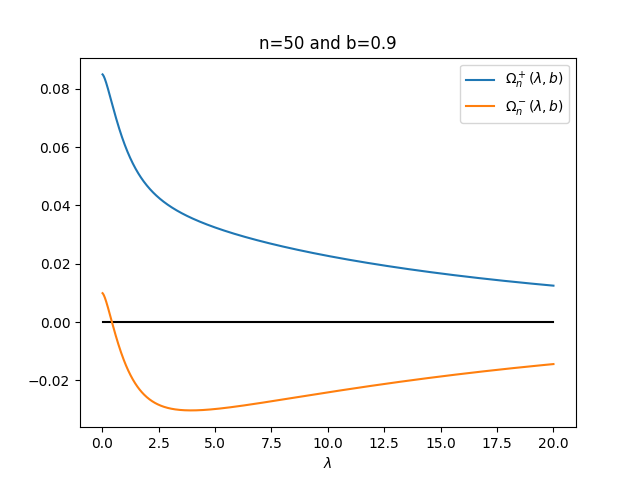}
	\end{subfigure}
	\caption{Non-vanishing Eulerian and vanishing QGSW annular spectrum.}\label{CV:eigen}
\end{figure}

The primary objective of this study is to rigorously establish the existence of non-trivial, doubly-connected, $\mathbf{m}$-fold symmetric stationary patch-type solutions to QGSW  equations \eqref{QGSW:eq}. To this end, we investigate two distinct bifurcation regimes originating from the trivial annulus configuration $\mathbb{A}_b$: one where the inverse Rossby radius $\lambda$ is fixed and the inner radius $b$ of the annulus varies, and another where $b$ is held constant and $\lambda$ changes. The exact statement of our main result reads as follows.
\begin{theo}\label{thm stationary Vstates QGSW}
	\textbf{(Stationary vortex patches for Quasi-Geostrophic Shallow-Water equations)}\\
    The following assertions hold true.
	\begin{enumerate}[label=(\roman*)]
		\item Let $\lambda\in(0,\infty).$ There exist $N(\lambda)\in\mathbb{N}^*$ and a strictly increasing sequence $\big(b_{\mathbf{m},\lambda}\big)_{\mathbf{m}\geqslant N(\lambda)}$ in $(0,1)$ converging to $1$ such that for any integer $\mathbf{m}\geqslant N(\lambda)$ there exist $\mathbf{m}$-fold doubly-connected stationary vortex patches bifurcating from the annulus
		$\mathbb{A}_{b_{\mathbf{m},\lambda}}$ for $(QGSW)_{\lambda}$ equations. In addition, the number $b_{\mathbf{m},\lambda}$ satisfies the asymptotic law
        $$1-b_{\mathbf{m},\lambda}\underset{\mathbf{m}\to\infty}{\sim}\frac{\beta(\lambda)}{\mathbf{m}},\qquad\beta(\lambda)>0.$$
		\item Let $b\in(0,1)$ and $\lambda_{\textnormal{max}}>0.$ There exist $N(b)\triangleq N(b,\lambda_{\textnormal{max}})\in\mathbb{N}^*$ and a strictly decreasing sequence $\big(\lambda_{\mathbf{m},b}\big)_{\mathbf{m}\geqslant N(b)}$ in $(0,\lambda_{\textnormal{max}})$ converging to $0$ such that for any integer $\mathbf{m}\geqslant N(b)$ there exist $\mathbf{m}$-fold doubly-connected stationary vortex patches bifurcating from the annulus  $\mathbb{A}_{b}$ for $(QGSW)_{\lambda_{\mathbf{m},b}}$ equations. In addition, the number $\lambda_{\mathbf{m},b}$ satisfies the asymptotic law
        $$\lambda_{\mathbf{m},b}\underset{\mathbf{m}\to\infty}{\sim}\frac{2}{\sqrt{(1-b^2)\mathbf{m}\log(\mathbf{m})}}\cdot$$
\end{enumerate}
\end{theo}
\begin{rem}We shall make the following remarks.
\begin{enumerate}
    \item The result of Theorem \ref{thm stationary Vstates QGSW}-$(i)$ states that fixing the model (i.e. fixing $\lambda$), we can construct stationary patches near specific very thin annuli. This is the analogous to G\'omez-Serrano's result \cite{G19} -- published in {\it Advances in Mathematics} -- for the generalized surface quasi-geostrophic (gSQG) equations where the Green kernel is proportional to $|x|^{-\alpha}$, fore some $\alpha\in(0,1)$. Even if less singular, the QGSW Green kernel $K_0(\lambda|x|)$ is non longer homogeneous and is related to modified Bessel functions. This difference makes the analysis substantially more delicate than in the aforementioned work.
    \item The result of Theorem \ref{thm stationary Vstates QGSW}-$(ii)$ allows to construct stationary patches near annuli of arbitrary size up to selecting the model. Each annulus admits highly symmetric families of non-trivial stationary solutions for QGSW models getting closer to the two-dimensional Euler equations as the symmetry increases. This is in contrast with the Euler equations themselves, which do not support stationary solutions of this type \cite{GPSY20}. This idea of playing with the model parameter to perform the bifurcation of V-states is new and inspired for instance from \cite{HR21}.
    \item The bifurcation analysis is run within the Hölderian functional framework. However, stationary solutions being in particular uniformly rotating with angular speed $\Omega=0,$ one can apply \cite[Lem. 4.1]{R21} to obtain that the boundaries of the patch solutions are actually analytic.
\end{enumerate}  
\end{rem}

We now state our second result providing some rigidity in the simply-connected class, therefore motivating the previous result.
\begin{theo}\label{thm sym Vstates QGSW}\textbf{(Rigidity results on simply-connected V-states for QGSW)}\\
Let $\lambda\in(0,\infty)$ and $\Omega\in\mathbb{R}.$ Let $D_{0}$ be a bounded, simply-connected domain in $\mathbb{R}^{2}$ generating a V-state at angular velocity $\Omega.$ 
\begin{enumerate}[label=(\roman*)]
\item We assume that $D_0$  has rectifiable boundary.
\begin{enumerate}
	\item If $\Omega\in(-\infty,0),$ then $D_{0}$ is radially symmetric.
\item If $\Omega=0$, then $D_{0}$ is radially symmetric up to a translation.
\end{enumerate}
\item We assume that $D_0$ is of class $C^1$. We denote 
$$R\triangleq \max_{x\in D_0}|x|.$$
If $\Omega\geqslant I_{1}(\lambda R)K_{1}(\lambda R),$ then $D_0=R\cdot\mathbb{D}.$
\end{enumerate}
\end{theo}
\begin{rem}Let us make the following remarks.
	\begin{enumerate} 
    \item Observe that the stationary simply-connected patches must be discs. This justifies the need of working within the doubly-connected framework for obtaining the non-trivial stationary states of Theorem \ref{thm stationary Vstates QGSW}. 
		\item Note that all sets $D_0\subset\mathbb{R}^2$ with area $\pi$ must have $R\geqslant 1.$ In this case we have $I_1(\lambda R)K_1(\lambda R)\leqslant I_1(\lambda) K_1(\lambda).$ Thus Theorem \ref{thm sym Vstates QGSW} immediately implies that all simply-connected rotating patches with area $\pi$ and $\Omega\geqslant I_1(\lambda)K_1(\lambda)$ must be a disc. The constant $I_1(\lambda)K_1(\lambda)$ is sharp, since there exist $\mathbf{m}$-fold patches bifurcating from a disc of radius $1$ at velocities $\Omega_{\mathbf{m}}(\lambda)=I_1(\lambda)K_1(\lambda)-I_{\mathbf{m}}(\lambda)K_{\mathbf{m}}(\lambda)$ which can get arbitrarily close to $I_1(\lambda)K_1(\lambda)$ as $\mathbf{m}\to\infty$ \cite[Thm. 5.1]{DHR19}.
         \item In Corollary \ref{cor mea diffsym}, we provide an upper bound for the symmetric difference between the unit disc $\mathbb{D}$ and the domain $D_0$ of a non-trivial V-state. This upper bound tends to zero as the angular velocity tends to $I_1(\lambda)K_1(\lambda)$ meaning that the non-trivial patch solutions becomes almost circular.
	\end{enumerate}
\end{rem}

Let us summarize our new contributions to the completeness of the following two tables. 

\begin{table}[!ht]
\centering
\renewcommand{\arraystretch}{1.5}
\begin{tabular}{cccc}
\hline
& Euler & $(\mathrm{gSQG})_{\alpha}$ & $(\mathrm{QGSW})_{\lambda}$ \\
\hline

$\Omega=0$
& Rigidity \cite{GPSY20}
& \cite{G19}
& \textbf{Theorem \ref{thm stationary Vstates QGSW}}
\\
\hline

$\Omega\neq0$
& \cite{HHMV16}
& \cite{HHH16,R17}
& \cite{R21}
\\
\hline
\end{tabular}
\caption{Existence of $\mathbf{m}$-fold V-states bifurcating from the annulus $\mathbb{A}_b$.}
\end{table}

\begin{table}[!ht]
\centering
\renewcommand{\arraystretch}{1.6}
\begin{tabular}{cccc}
\hline
& Euler & $(\mathrm{gSQG})_{\alpha}$ & $(\mathrm{QGSW})_{\lambda}$ \\
\hline
Rigidity for
$\Omega < 0$
& \cite{H15}
& \cite{GPSY20}
& \textbf{Theorem \ref{thm sym Vstates QGSW}-$(i)$-$(a)$}
\\
\hline
Rigidity for
$\Omega=0$
& \cite{F00}
& \cite{GPSY20}
& \textbf{Theorem \ref{thm sym Vstates QGSW}-$(i)$-$(b)$} 
\\
\hline
Rigidity for
$\Omega \geqslant \overline{\Omega}$
&  \makecell{\cite{H15,GPSY20}\vspace{0.1cm}\\
$\overline{\Omega}=\frac{1}{2}$}
&  \makecell{\vspace{0.1cm}\cite{GPSY20}\\
$\overline{\Omega}=2^{\alpha-1} \frac{\Gamma(1-\alpha)}{\Gamma\left(1-\frac{\alpha}{2}\right)^2} \frac{\Gamma\left(1+\frac{\alpha}{2}\right)}{\Gamma\left(2-\frac{\alpha}{2}\right)}$}
& \makecell{\textbf{Theorem \ref{thm sym Vstates QGSW}-$(ii)$}\vspace{0.1cm}\\
$\overline{\Omega}=I_1(\lambda)K_1(\lambda)$} 
\\
\hline
\makecell{Non-trivial V-states for\\
$0<\Omega<\overline{\Omega}$}
& \cite{B82,HMV13}
& \cite{CCG16,HH15}
& \cite{DHR19}
\\
\hline

\end{tabular}
\caption{Rigidity and flexibility results for simply-connected V-states.}
\end{table}

\newpage\subsection{Ideas of the proofs}
To improve readability of the manuscript, let us expose the main steps in the proofs of both Theorem \ref{thm stationary Vstates QGSW} and Theorem \ref{thm sym Vstates QGSW}.\\

\noindent\ding{172} \textbf{Bifurcation analysis:} This is the content of Section \ref{sec:bif}. Let us fix the Rossby radius $\lambda\in(0,\infty)$ and an inner radius $b\in(0,1)$ of the annulus $\mathbb{A}_b$ introduced in \eqref{def:annulus}. Consider a doubly-connected vortex patch 
$$q_0=\mathbf{1}_{D_0},\qquad D_0\triangleq D_1\setminus\overline{D_2},$$
where $D_1$ and $D_2$ are simply-connected bounded domains with $\overline{D}_2\subset D_1.$
We assume that the domain $D_0$ is ``close enough'' to $\mathbb{A}_b$, that is $D_1$ is a small perturbation of the unit disc $\mathbb{D}$, while $D_2$ is a small perturbation of the rescaled disc $b\cdot\mathbb{D}$. Since $D_1$ and $D_2$ are simply-connected and bounded, by the Riemman mapping theorem there exist two unique conformal applications 
$$\Phi_k:\mathbb{C}\setminus\overline{\mathbb{D}}\to\mathbb{C}\setminus\overline{D_k},\qquad\Phi_k(z)=b_k\,z+\sum_{n = 0}^{\infty}\frac{f_{k,n}}{z^{n-1}}\triangleq b_k\,z+f_k(z),\qquad k\in\{1,2\},\qquad b_1\triangleq1,\qquad b_2\triangleq b.$$
Assuming enough smoothness on the boundary of $D_0$, the conformal mappings $\Phi_1,\Phi_2$ extend smoothly to the boundary $\mathbb{T}\triangleq\partial\mathbb{D}.$ In this setting, following \cite{R21}, the contour dynamics equation \eqref{CDE} is reduced to a set of two equations
$$G\triangleq(G_1,G_2)=0,$$ 
where
$$\forall\,w\in\mathbb{T},\quad G_j\left(\lambda,b,f_1,f_2\right)(w)\triangleq\operatorname{Im}\left\{\left[S\left(\lambda,\Phi_2, \Phi_j\right)(w)-S\left(\lambda,\Phi_1,\Phi_j\right)(w)\right]\overline{w\Phi_j'(w)}\right\}.$$
Here, for $(j, k) \in \{1, 2\}^2$, 
$$S\left(\lambda, \Phi_j, \Phi_k\right)(w) \triangleq \fint_{\mathbb{T}} \Phi_j'(\tau) K_0\left(\lambda \left|\Phi_k(w) - \Phi_j(\tau)\right|\right) \, d \tau.$$
Hence, finding a non trivial, doubly-connected, stationary patch sufficiently ``close'' to the annulus for the system \eqref{QGSW:eq} is equivalent to finding a non trivial solution of the nonlinear functional $G.$
Note that, no matters the model (i.e. the value of $\lambda \in (0,\infty)$) and the thickness (i.e. the choice of $b\in(0,1)$), the annulus $\mathbb{A}_b$ constitutes a trivial solution to \eqref{QGSW:eq} that is
$$G(\lambda,b,0,0)=0.$$
This family of trivial stationary solutions is the starting point for our bifurcation analysis. The construction of nontrivial branches of solutions is carried out by applying Crandall-Rabinowitz theorem, whose statement is recalled in Theorem \ref{Crandall-Rabinowitz theorem}. The first step is to provide the functional framework where to run the analysis. We choose to work within the classical Hölderian regularity in order the exploit the computations already carried out in \cite{R21}. We look for stationary solutions that are $\mathbf{m}$-fold symmetric, namely invariant under the action of the dihedral group of order $\mathbf{m}$. We encode this symmetry in the functional spaces to which $f_1$ and $f_2$ belong, by requiring that their Fourier expansions have real coefficients and involve only modes of the form $n \mathbf{m}-1$. The precise definition of the working function spaces is given in Subsection \ref{sub_fun_sp}. Then, in Proposition \ref{reg G and lin op}, we remind that the nonlinear functional $G$ is well-defined and of class $C^1$ with respect to these spaces. We also give the expression of the linearized operator with respect to $f_1$ and $f_2$ at the trivial solutions
$$\forall\, w\in\mathbb{T},\quad DG(\lambda,b,0,0)\left[h_1, h_2\right](w)=\sum_{n=0}^{\infty}n\mathbf{m} M_{n\mathbf{m}}(\lambda,b) \begin{pmatrix}
		h_{1,n}\\
		h_{2,n}
\end{pmatrix}\textnormal{Im}\big(w^{n\mathbf{m}}\big),$$
where
\begin{equation*}
			M_{n}(\lambda,b)\triangleq \begin{pmatrix}
				\Omega_{n}(\lambda)-b\Lambda_{1}(\lambda,b) & b\Lambda_{n}(\lambda,b)\\
				-\Lambda_{n}(\lambda,b) & \Lambda_{1}(\lambda,b)-b\Omega_{n}(\lambda b)
			\end{pmatrix},
\end{equation*}
with $\Omega_n(\lambda)$ as in \eqref{bifpt:SC} and $\Lambda_n(\lambda,b)$ as in \eqref{def_lambda_m_intro}. A sufficient condition to get a bifurcation point is the non-invertibility of the linearized operator $DG(\lambda,b,0,0)$. To this aim, we introduce
$$\mathcal{D}_n(\lambda, b) \triangleq \det\big(M_n(\lambda, b)\big)=\Big[\Lambda_{1}(\lambda,b)-b\Omega_{n}(\lambda b)\Big]\Big[\Omega_{n}(\lambda)-b\Lambda_{1}(\lambda,b)\Big]+b\Lambda_{n}^2(\lambda,b)$$ 
and study its roots, which correspond to the values of the parameters for which the $n$-th Fourier mode contributes to the kernel of the linearized operator. These zeros constitute the natural candidates for bifurcation points. We analyze this problem from two complementary perspectives. In Section \ref{sec:bif_b}, we fix the Rossby radius $\lambda\in(0,\infty)$ and regard the inner radius $b\in(0,1)$ as bifurcation parameter. In Section \ref{sec:bif_lambda}, we reverse the roles of the parameters and fix $b$, using $\lambda$ as bifurcation parameter.\\

$\diamond$ \textit{The $b$-bifurcation:} It is performed in Section \ref{sec:bif_b}. We fix $\lambda\in(0,\infty)$. By exploiting the asymptotic behavior in $n$ at the end points of the interval $(0, 1)$, we establish in Lemma \ref{lem:cvDnb} a sign-changing property of the continuous function $b\mapsto \mathcal{D}_n(\lambda,b)$ that guarantees the existence, for $n$ large enough, of at least one root $b_{n,\lambda}\in(0,1)$ to the equation $\mathcal{D}_{n}(\lambda,b)=0$. Then, under a suitable renormalization analysis, we show in Lemma \ref{lem:accu1} that all such roots must accumulate to the limiting configuration $b=1$ as $n\to\infty$. A crucial step is then to derive a refined asymptotic description of their convergence. We demonstrate in Lemma \ref{lem:asybn} that any such root satisfies an expansion of the form
\begin{equation} \label{bnasy}
    b_{n, \lambda} \underset{n\to \infty}{=} 1 - \frac{\beta(\lambda)}{n} + O\left(\frac{1}{n^2}\right),
\end{equation}
for an explicitly characterized coefficient $\beta(\lambda) > 0$ (see Remark \ref{rem_beta_lambda}). The derivation of this asymptotic formula, that is fundamental for the following analysis, constitutes one of the most challenging part of the paper.
Indeed, the expression $0=\mathcal{D}_{n}\left(\lambda, b_{n, \lambda} \right)$ involves modified Bessel functions depending on the large parameter $n$ simultaneously through their order and their argument. As a consequence, a standard Taylor expansion around $b=1$, that is the limiting value of $b_{n,\lambda}$, breaks down. The reason is that the $k$-th order derivative in $b$ of the term $\Lambda_n(\lambda,b)$, appearing in the expression of $\mathcal{D}_n(\lambda,b)$, grows in a polynomial way with $n$ (see equation \eqref{svi:diffkLbd}), compensating the small factor $\left(1-b_{n,\lambda}\right)^k$. In other words, all the terms of the Taylor expansion of $\Lambda_n\left(\lambda,b_{n, \lambda}\right)$ (the ``bad'' part of $\mathcal{D}_{n})$ around $1$ contribute at the same scale, implying that every order must be summed and any truncation of the Taylor expansion is not valid. To overcome this problem, we isolate the growing part in the Taylor integral formula and study the corresponding convergence.

Once the location of the roots is understood, we turn to their qualitative properties. Exploiting the crucial asymtotic \eqref{bnasy}, we prove in Lemma \ref{lem:asyan} that $\partial_b \mathcal{D}_n\left(\lambda,b_{n, \lambda}\right)>0$ for all sufficiently large $n$ depending on $\lambda$. This positivity property being valid for any root, combined with some suitable global estimates, yields the uniqueness and simplicity of the root of $b_{n,\lambda}$ in the interval $(0,1)$. The formula \eqref{bnasy} also provides the asymtotic monotonicity of the sequence of the zeros.

We are now in the position perform the bifurcation analysis in Proposition \ref{prop:CRb}. We fix a symmetry parameter $\mathbf{m}$ large enough and choose as potential bifurcation point $b=b_{\mathbf{m},\lambda}$. The uniqueness and simplicity of the zero $b_{\mathbf{m},\lambda}$ together with its strict asymptotic monotonicity as $\mathbf{m}\to\infty$ imply that $b_{\mathbf{m},\lambda}$ cannot be a zero of $\mathcal{D}_n(\lambda,\cdot)$ for any $n\not=\mathbf{m}$, meaning that no resonance between different harmonics can occur at the candidate bifurcation point. Therefore, the kernel of the linearized operator $DG\left(\lambda, b_{\mathbf{m}, \lambda}, 0, 0\right)$ is exactly one-dimensional. Since the linearized operator is Fredholm of index $0$ (see Proposition \ref{prop-fredholmness}), it follows that its range has codimension one. All the hypotheses required by the Crandall-Rabinowitz theorem are verified, except for the transversality condition. The remainder of the analysis is devoted to proving this last non-degeneracy condition.

We first describe the range by duality following the argument developed in \cite{R23} in the scalar case and then extended to matricial situations in \cite{GHR23,HHRZ25,R25,R26}. This allows to translate the transversality condition into checking that a scalar product does not vanish. At this stage, once more, the higher-order asymptotic expansion \eqref{bnasy} of the root $b_{\mathbf{m},\lambda}$ becomes the key ingredient of the argument. In fact, if one only retains the leading-order asymptotic behavior corresponding to the limit $\mathbf{m} \to \infty$, the principal contribution to the transversality expression vanishes. The underlying reason is that the generator of the kernel tends to $(0,0)$ as $\mathbf{m}\to\infty$. Thus, the limiting problem is unable to determine whether the transversality condition is satisfied. To resolve this issue, one must go beyond the leading-order approximation and incorporate the higher-order correction terms arising in the expansion \eqref{bnasy} of $b_{\mathbf{m},\lambda}$. These contributions, together with a lower bound on $\beta(\lambda)$, produce the first non-vanishing term in the transversality expression yielding the desired result.\\

$\diamond$ \textit{The $\lambda$-bifurcation:} It is performed in Section \ref{sec:bif_lambda}. We fix $b\in(0,1).$ The bifurcation with respect to $\lambda$ follows a similar strategy to the one presented above. In Proposition \ref{prop seq lbd}, we first obtain the existence, for $n$ large enough, of a zero by a sign-changing property of the function $\lambda\mapsto\mathcal{D}_{n}(\lambda,b)$. Then, we show that these roots accumulate to $0$ as $n\to\infty$ following the asymptotic expansion
\begin{equation}\label{lbdnasy}
    \lambda_{n,b}\underset{n\to\infty}{=}\frac{2}{\sqrt{\left(1-b^2\right)n\log(n)}}\left(1-\frac{\log\left(\log(n)\right)}{2\log(n)}+O\left(\frac{1}{\log(n)}\right)\right).
\end{equation}
Such asymptotic law comes from the expansion of the determinant $\mathcal{D}_{n}(\lambda,b)$ as $\lambda\to0$, and is in particular driven by the contribution of $\Lambda_1(\lambda,b)$. Indeed, the leading-order balance is governed by
$$\lambda_{n,b}^2 \log\left(\lambda_{n,b}^2\right) \underset{n\to\infty}{\sim}- \frac{4}{\left(1-b^2\right)n}\cdot$$
Once we have this asymptotic relation, we compute $\partial_\lambda \mathcal{D}_n\left(\lambda_{n, b}, b\right)$ and prove that it is strictly negative for $n$ large enough. This implies the uniqueness and simplicity of the zero $\lambda_{n,b}$. Coming back to the expansion \eqref{lbdnasy}, we also get the strict asymptotic monotonicity. In Proposition \ref{prop:CRlbd}, we show that the spectral assumptions appearing in the statement of the Crandall-Rabinowitz Theorem \ref{Crandall-Rabinowitz theorem} hold true, due to the Fredholmness of the linearized operator and the properties of the sequence $\left(\lambda_{n, b}\right)_{n}$. On the other hand, in order to check the transersality condition, we make a new use of the refined asymptotic \eqref{lbdnasy}.\\

\noindent\ding{173} \textbf{Rigidity Theorem for simply-connected patches:} In Section \ref{section_rigidity} we prove Theorem \ref{thm sym Vstates QGSW} establishing rigidity results for simply-connected uniformly rotating vortex patches for QGSW equations \eqref{QGSW:eq}. The symmetry result for non-positive angular velocities (Theorem \ref{thm sym Vstates QGSW}-$(i)$) follows by verifying that the quasi-geostrophic shallow-water Green kernel satisfies the structural assumption of \cite[Thm. 4.2]{GPSY20}. To prove the disc rigidity for large enough angular velocities (Theorem \ref{thm sym Vstates QGSW}-$(ii)$), we introduce the defect set $U\triangleq (R\cdot\mathbb{D})\setminus D_0$ and rewrite the boundary equation for V-states in the form
$$\mathbf{1}_U \ast\mathcal{K}_{\lambda}=\mathbf{1}_{R\cdot\mathbb{D}}\ast\mathcal{K}_{\lambda}-\frac{\Omega}{2}|\cdot|^2+C, \quad \textnormal{in } \partial D_0,\qquad\mathcal{K}_{\lambda}\triangleq-\frac{1}{2\pi}K_0(\lambda|\cdot|).$$
A crucial step consists in deriving the explicit expression of the potential generated by a disc $\mathbf{1}_{R\cdot\mathbb{D}}\ast\mathcal{K}_{\lambda}$ via Beltrami's addition formula \eqref{Beltrami's summation formula}, the integral identities \eqref{Bessel and anti-derivatives} and the Wronskian relation \eqref{wronskian}. This representation, together with the monotonicity of $x \mapsto \frac{I_1(x)}{x}$, yields that the function $\mathbf{1}_{R\cdot\mathbb{D}}\ast\mathcal{K}_{\lambda}-\frac{\Omega}{2}|\cdot|^2$ is non-decreasing with respect to $|x|$ on $(0,R]$ provided that $\Omega\geqslant I_1(\lambda R)K_1(\lambda R)$. On the other hand, differentiating the defect potential $\mathbf{1}_U\ast\mathcal{K}_{\lambda}$ and using the identity \eqref{Bessel derivatives}, we prove its strict superharmonicity within $D_0$. By the maximum principle, the defect potential attains its minimum on $\overline{D_0}$ at a boundary point $x_0 \in \partial D_0$. Consequently, the radial derivative at this point satisfies
$$\nabla \left[\mathbf{1}_{U}\ast\mathcal{K}_{\lambda}\right]\left(x_0\right) \cdot x_0 \leqslant 0.$$
On the other hand, since $U \subset R\cdot\mathbb{D}$ and $\left|x_0\right| = R$, we get, for every $y \in U$, $\left(x_0 - y\right) \cdot x_0 > 0$ and so, by \eqref{symmetry Bessel}, 
$$\nabla\left[\mathbf{1}_{U}\ast\mathcal{K}_{\lambda}\right]\left(x_0\right)\cdot x_0>0,$$ 
which yields a contradiction. It follows that $U=\varnothing$ and hence $D_0=R\cdot\mathbb{D}$.\\

\noindent\textbf{Acknowledgements:} 

\smallskip
V.B. is  supported by the FPI grant from the Spanish Government PREP2022-000038, and partially  by the AEI project PID2022-140494NA-I00 (Spain).

C.G. is supported by RYC2022-035967-I (MCIU/AEI/10.13039/501100011033 and
FSE+), and partially by Grants PID2022-140494NA-I00 and PID2022-137228OB-I00 funded
by MCIN/AEI/10.13039/5011000 11033/FE\\
DER, UE,  by the RED2024-153842-T, funded by MCIN/AEI /10.13039/501100011033,  by Grant C-EXP-265-UGR23 funded by
Consejeria de Universidad, Investigacion e Innovacion \& ERDF/EU Andalusia Program, and
by Modeling Nature Research Unit, project QUAL21-011. Proyecto realizado con la Beca Leonardo a
Investigadores y Creadores Culturales 2024 de la Fundaci\'on BBVA.

E.R. is supported by the ERC STARTING GRANT 2021 ``Hamiltonian Dynamics, Normal Forms and Water Waves'' (HamDyWWa), Project Number: 101039762. E.R. is also supported by INdAM - GNAMPA Project ``(In)-stability in Fluid Mechanics'', CUP E53C25002010001.

\section{Bifurcation analysis}\label{sec:bif}
This section is devoted to the proof of Theorem \ref{thm stationary Vstates QGSW}, namely the construction of doubly-connected stationary $\mathbf{m}$-fold vortex patches, where $\mathbf{m}$ is a large integer, for the quasi-geostrophic shallow-water equations. To this end, we employ local bifurcation argument. We begin by introducing the nonlinear functional describing the problem and the associated function spaces on which the bifurcation analysis is carried out. The functional depends on two external real parameters: the inner radius $b$ of the reference annulus and the Rossby deformation length $\lambda$. The bifurcation analysis is performed in two stages, according to whether the bifurcation parameter is chosen to be $b$ or $\lambda$.
\subsection{Functional analysis}
We start with the functional framework in which we apply the bifurcation theory. We begin with the definition of Hölder regularity spaces, then we continue with the presentation of the functional of interest, its regularity with respect to the function spaces and its linearization at the trivial radial state. Finally, we expose the Fredholm nature of the linearized operator at the equilibrium state.
\subsubsection{Functions spaces} \label{sub_fun_sp}
Throughout this work, it is convenient to identify a $2\pi$-periodic function $g:\mathbb{R}\rightarrow\mathbb{C}$ with a function defined on the unit circle through the change of variables $w=e^{\ii\vartheta}.$ More precisely, given a continuous function on the unit circle $f:\mathbb{T}\rightarrow\mathbb{R}^{2}$, we associate to it a $2\pi$-
periodic function $g:\mathbb{R}\rightarrow\mathbb{R}^{2}$ by setting
$$f(w)=g(\vartheta),\qquad w=e^{\ii\vartheta}.$$
Whenever $f$ is sufficiently smooth, differentiation with respect to $w$ is related to differentiation with respect to $\vartheta$ through
$$f'(w)\triangleq \frac{df}{dw}=-\ii e^{-\ii\vartheta}g'(\vartheta).$$
Since the operators $\frac{d}{dw}$ and $\frac{d}{d\vartheta}$ differ only by a smooth multiplicative factor of unit modulus, we shall henceforth use $\frac{d}{dw}$ instead of $\frac{d}{d\vartheta}$, which turns out to be more convenient for our computations.
Moreover, if $f$ has real Fourier coefficients and is of class $C^{1}$
then we can easily check that
\begin{equation*}
\left(\overline{f}\right)'(w)=-\frac{\overline{f'(w)}}{w^{2}}\cdot
\end{equation*}
Now, we shall recall the definition of Hölder spaces on the unit circle.
\begin{defin}
Let $\alpha\in(0,1).$
\begin{enumerate}[label=(\roman*)]
\item We denote by $C^{\alpha}(\mathbb{T})$ the space of continuous functions $f$ such that
$$\| f\|_{C^{\alpha}(\mathbb{T})}\triangleq \| f\|_{L^{\infty}(\mathbb{T})}+\sup_{\underset{\tau\neq w}{(\tau,w)\in\mathbb{T}^{2}}}\frac{|f(\tau)-f(w)|}{|\tau-w|^{\alpha}}<+\infty.$$
\item We denote by $C^{1+\alpha}(\mathbb{T})$ the space of $C^{1}$ functions with $\alpha$-Hölder continuous derivative :
$$\| f\|_{C^{1+\alpha}(\mathbb{T})}\triangleq \| f\|_{L^{\infty}(\mathbb{T})}+\Big\| \frac{df}{dw}\Big\|_{C^{\alpha}(\mathbb{T})}<+\infty.$$
\end{enumerate}
\end{defin}
\noindent For $\alpha\in(0,1)$ and $\mathbf{m}\in\mathbb{N}^*,$ we set
$$X_{\mathbf{m}}^{1+\alpha}\triangleq X_{1,\mathbf{m}}^{1+\alpha}\times X_{1,\mathbf{m}}^{1+\alpha}\quad\mbox{ with }\quad X_{1,\mathbf{m}}^{1+\alpha}\triangleq \left\lbrace f\in C^{1+\alpha}(\mathbb{T})\quad\textnormal{s.t.}\quad\forall w\in\mathbb{T},\,f(w)=\sum_{n=1}^{+\infty}f_{n}\overline{w}^{\mathbf{m}n-1},\quad f_n\in\mathbb{R}\right\rbrace,$$
and
$$Y_{\mathbf{m}}^{\alpha}\triangleq Y_{1,\mathbf{m}}^{\alpha}\times Y_{1,\mathbf{m}}^{\alpha}\quad\mbox{ with }\quad Y_{1,\mathbf{m}}^{\alpha}\triangleq \left\lbrace g\in C^{\alpha}(\mathbb{T})\quad\textnormal{s.t.}\quad\forall w\in\mathbb{T},\,g(w)=\sum_{n=1}^{+\infty}g_{n}e_{\mathbf{m}n}(w),\quad g_n\in\mathbb{R}\right\rbrace,$$
with
$$e_{n}(w)\triangleq \mbox{Im}(w^{n}).$$
The parameter ${\bf m}$ will describe the $\mathbf{m}$-fold symmetry of the bifurcated domains, and the condition that $f_n\in\mathbb{R}$ will imply the symmetry with respect with the real axis. The space $X_{\mathbf{m}}^{1+\alpha}$ (resp. $Y_{\mathbf{m}}^{\alpha}$) is equipped with the strong product topology of $C^{1+\alpha}(\mathbb{T})\times C^{1+\alpha}(\mathbb{T})$ (resp. $C^{\alpha}(\mathbb{T})\times C^{\alpha}(\mathbb{T})$). Given $r>0,$ we denote the ball of radius $r$ in $X_{1,\mathbf{m}}^{1+\alpha}$ as
$$B_{r,\mathbf{m}}^{1+\alpha}\triangleq \big\{ f\in X_{1,\mathbf{m}}^{1+\alpha}\quad\textnormal{s.t.}\quad\| f\|_{C^{1+\alpha}(\mathbb{T})}<r\big\}.$$
We shall consider the following pre-Hilbertian structure on $Y_{\mathbf{m}}^{\alpha}.$ For $f=(f_1,f_2)$ and $g=(g_1,g_2)$ in $Y_{\mathbf{m}}^{\alpha}$ in the form
$$\forall i\in\{1,2\},\qquad f_{i}=\sum_{n=1}^{\infty}f_{i,n}e_{n\mathbf{m}},\quad g_{i}=\sum_{n=1}^{\infty}g_{i, n}e_{n\mathbf{m}},\qquad f_{i,n},g_{i, n}\in\mathbb{R},$$
we define
\begin{equation}\label{scalar product}
	\big\langle f,g\big\rangle\triangleq\sum_{n=1}^{\infty}\Big(f_{1,n}\,g_{1,n}+f_{2,n}\,g_{2,n}\Big).
\end{equation}
Finally, for a given continuous function $\varphi:\mathbb{T}\to\mathbb{R},$ we define the complex integral as follows
$$\fint_{\mathbb{T}} \varphi(w) \, dw \triangleq \frac{1}{2\ii\pi} \int_{\mathbb{T}} \varphi(w) \, d w \triangleq \frac{1}{2\pi} \int_{0}^{2\pi} \varphi\left(e^{\ii\vartheta}\right) e^{\ii \vartheta} \, d \vartheta.$$
\subsubsection{Functional of interest and linearzation at annuli}
Fix $(\lambda,b)\in(0,\infty)\times(0,1)$. We take an initial doubly-connected bounded domain $D_0$ near the annulus $\mathbb{A}_{b}$ introduced in \eqref{def:annulus}. We can write
$$D_0=D_1\setminus\overline{D_2},$$
where $D_1$ and $D_2$ are simply-connected bounded domains with $\overline{D_2}\subset D_1$ and
$$D_1\approx\mathbb{D},\qquad D_2\approx b\cdot\mathbb{D}.$$
Here we denoted by $\mathbb{D}$ the unit disc. Next consider the external Riemann mappings 
$$\Phi_1:\mathbb{C}\setminus\overline{\mathbb{D}}\to\mathbb{C}\setminus\overline{D_1}\quad\textnormal{and}\quad\Phi_2:\mathbb{C}\setminus\overline{\mathbb{D}}\to\mathbb{C}\setminus\overline{D_2},$$
given by
$$\Phi_1(z)=z+f_1(z)=z+\sum_{n=1}^{\infty}\frac{f_{1, n}}{z^{n-1}},\qquad\Phi_2(z)=bz+f_2(z)=bz+\sum_{n=1}^{\infty}\frac{f_{2, n}}{z^{n-1}}\cdot$$
The coefficients $f_{i, n}$ ($i\in\{1,2\}$) are a priori complex but will be chosen real. This amounts to restricting our attention to simply-connected domains that are symmetric with respect to the real axis. Since our goal is to construct $\mathbf{m}$-fold symmetric solutions, such domains necessarily possess at least one axis of symmetry. The rotational invariance of the problem allows any such axis to be mapped onto the real axis through a suitable rotation, thereby preserving the full generality of the argumentation. The applications $\Phi_1$ and $\Phi_2$ are holomorphic. If the domains $D_1$ and $D_2$ are smooth enough, by virtue of Kellogg-Warschawski theorem (see \cite{W35} and \cite[Thm. 3.6]{P92}), these applications admit a continuous extension of class $C^{1+\alpha}$ on $\mathbb{C}\setminus\mathbb{D}$ for $\alpha\in(0,1)$ and the Jordan curve $\Phi_i(\mathbb{T})$ is of class $C^{1+\alpha}.$ Recall that $\mathbb{T}=\partial\mathbb{D}$ stands for the unit circle, agreeing with the notations of the previous subsection. In our analysis, we consider the restriction of $f_1,f_2$ to $\mathbb{T}$ (keeping the same notations). The domain $D_0$ is taken $\mathbf{m}$-fold and smooth enough so that $(f_1,f_2)\in X_{\bf m}^{1+\alpha}$ in accordance with the notations of the previous subsection. Note that the $\mathbf{m}$-fold symmetry imposes the coefficients $f_{k,n}$ to be non-zero only for $n$ in the lattice $\mathbf{m}\mathbb{Z}.$\\

With these notations in hand, we can now introduce the functional of main interest in this study. A stationary vortex patch can be seen as a V-state rotating with angular velocity $\Omega=0.$ Following \cite{R21}, the boundary equations of interest write
$$G(\lambda,b,f_{1},f_{2})=0,$$
where $G=(G_{1},G_{2})$ is defined by
\begin{equation}\label{def:G}\forall j\in\{1,2\},\quad\forall w\in\mathbb{T},\quad G_{j}(\lambda,b,f_{1},f_{2})(w)\triangleq \mbox{Im}\left\lbrace\big[S(\lambda,\Phi_{2},\Phi_{j})(w)-S(\lambda,\Phi_{1},\Phi_{j})(w)\big]\overline{w}\overline{\Phi_{j}'(w)}\right\rbrace,\end{equation}
with
$$\forall(i,j)\in\{1,2\}^{2},\quad \forall w\in\mathbb{T},\quad S(\lambda,\Phi_{i},\Phi_{j})(w)\triangleq \fint_{\mathbb{T}}\Phi_{i}'(\tau)K_{0}\big(\lambda\left|\Phi_{j}(w)-\Phi_{i}(\tau)\right|\big)d\tau.$$
We shall now briefly discuss the regularity of the functional $G$ with respect to the above functional framework. By \cite[Prop. 2.1]{R21}, there exists $r>0$ such that the functional $$G(\lambda,b,\cdot,\cdot):B_{r,\mathbf{m}}^{1+\alpha}\times B_{r,\mathbf{m}}^{1+\alpha}\rightarrow Y_{\mathbf{m}}^{\alpha}$$ 
is well-defined and of class $C^{1}$. The regularity with respect to the parameter $\lambda\in(0,\infty)$ was partially investigated in \cite[Prop. 5.7]{DHR19} for the self-interaction terms associated with the boundaries. The interaction terms coupling the two boundaries involve only non-singular kernels and can therefore be handled in a straightforward manner. The regularity with respect to the parameter $b\in(0,1)$ can be established similarly. Indeed, $b$ either appears in the interaction terms between the two boundaries, where it is associated with smooth kernels, or in the self-interaction term of the inner interface, where it factors out and behaves as the parameter $\lambda$. The first point of the next proposition summarizes the previous discussion and reminds that the annuli are trivial solutions, see \cite[Lem. 2.1]{R21} for more details on this last property. The second point recalls the Fourier multiplier structure of the linearized operator at the annulus state and follows immediately from \cite[Prop. 3.1]{R21}.
\begin{prop}\label{reg G and lin op}
	The following assertions hold true.
	\begin{enumerate}
		\item There exists $r>0$ such that, for any $\alpha\in(0,1),$ the application 
        $$G:(0,\infty)\times(0,1)\times B_{r,\mathbf{m}}^{1+\alpha}\times B_{r,\mathbf{m}}^{1+\alpha}\rightarrow Y_{\mathbf{m}}^\alpha$$ 
        is well-defined and of class $C^1.$ In addition,
        $$\forall\lambda>0,\quad\forall b\in(0,1),\quad G(\lambda,b,0,0)=0.$$
		\item For any $(\lambda,b)\in(0,\infty)\times (0,1)$, for any $\left(h_1,h_2\right)\in X_{\mathbf{m}}^{1+\alpha}$, the linearized operator at the equilibrium state writes
		$$DG(\lambda,b,0,0)[h_{1},h_{2}]=\sum_{n=0}^{\infty}n\mathbf{m}M_{n\mathbf{m}}(\lambda,b)\begin{pmatrix}
			h_{1,n}\\
			h_{2,n}
		\end{pmatrix}e_{n\mathbf{m}},$$
		where $D\triangleq\partial_{(f_1,f_2)}$ and
		\begin{equation}\label{matrix Mn}
			M_{n}(\lambda,b)\triangleq \begin{pmatrix}
				\Omega_{n}(\lambda)-b\Lambda_{1}(\lambda,b) & b\Lambda_{n}(\lambda,b)\\
				-\Lambda_{n}(\lambda,b) & \Lambda_{1}(\lambda,b)-b\Omega_{n}(\lambda b)
			\end{pmatrix},
		\end{equation}
		with
		\begin{equation}\label{def:Lbd and Omg}
		    \Lambda_{n}(\lambda,b)\triangleq I_{n}(\lambda b)K_{n}(\lambda)\quad\mbox{ and }\quad\Omega_{n}(\lambda)\triangleq I_{1}(\lambda)K_{1}(\lambda)-I_{n}(\lambda)K_{n}(\lambda).
		\end{equation}
	\end{enumerate}
\end{prop}
\noindent In order to get non-trivial stationary solutions, we need a nontrivial kernel for the linearized operator. Then, we want to find
\begin{enumerate}[label=\textbullet]
\item for fixed $n\in\mathbb{N}^{*}$ and $\lambda\in(0,\infty)$, the number(s) $b\in(0,1)$ such that the matrix $M_{n}(\lambda,b)$ is singular;
\item for fixed $n\in\mathbb{N}^{*}$ and $b\in(0,1)$, the number(s) $\lambda\in(0,\infty)$ such that the matrix $M_{n}(\lambda,b)$ is singular.
\end{enumerate}
Then we hope to apply Crandall-Rabinowitz's Theorem to bifurcate from these points. Therefore, we are led to study the determinant function
\begin{equation}\label{def:Dn}
    \mathcal{D}_{n}(\lambda,b)\triangleq  \det\big(M_{n}(\lambda,b)\big)=\big[\Omega_{n}(\lambda)-b\Lambda_{1}(\lambda,b)\big]\big[\Lambda_{1}(\lambda,b)-b\Omega_{n}(\lambda b)\big]+b\Lambda_{n}^{2}(\lambda,b).
\end{equation}
We start by giving some elementary asymptotic behaviors of the determinant function.
\begin{lem}\label{lem:elemD}
The following properties hold true.
    \begin{enumerate}[label=(\roman*)]
        \item For every $n \in \mathbb{N}^*$, we have the following evaluations
        \begin{equation}\label{lim det}
	\mathcal{D}_{n}(\lambda,0)=\mathcal{D}_{n}(\lambda,1)=\lim_{\lambda\to\infty}\mathcal{D}_{n}(\lambda,b)=0,\qquad \mathcal{D}_{n}(0,b)=\frac{b}{4n^2}\Big[n(1-b^2)-1+b^{2n}\Big].
\end{equation}
\item We have the following uniform convergence
\begin{equation}\label{unifCV:D-2var}
    \lim_{n\rightarrow\infty}\sup_{(\lambda,b)\in[0,\infty)\times[0,1]}\Big|\mathcal{D}_{n}(\lambda,b)-\mathcal{D}_{\infty}(\lambda,b)\Big|=0,
\end{equation}
where we denote
\begin{equation}\label{def:Dinfty}
    \mathcal{D}_{\infty}(\lambda,b)\triangleq  \Lambda_{1}(\lambda,b)\big[I_{1}(\lambda)-bI_{1}(\lambda b)\big]\big[K_{1}(\lambda)-bK_{1}(\lambda b)\big].
\end{equation}
\item The asymptotic profile $\mathcal{D}_{\infty}$ satisfies
\begin{equation}\label{sign Dinfty}
	\forall\lambda>0,\quad\forall b\in(0,1),\quad\mathcal{D}_{\infty}(\lambda,b)<0
\end{equation}
and 
\begin{equation}\label{diff2Dinfty}
    \forall\lambda>0,\quad\partial_b^2\mathcal{D}_{\infty}(\lambda,1)=-2\lambda^2I_0(\lambda)K_0(\lambda)I_1(\lambda)K_1(\lambda).
\end{equation}
    \end{enumerate}
\end{lem}
\begin{rem}In the following, we give some comments about the function $\mathcal{D}_{\infty}(\lambda,b).$
\begin{enumerate}
    \item Note that $\mathcal{D}_{\infty}(\lambda,b)$ is the product of the isomorphism coefficients in \eqref{isoL12} below.
    \item We refer to Figures \ref{detlimb} and \ref{detlimlbd} for numerical representations of the limiting profiles $b\mapsto\mathcal{D}_{\infty}(\lambda,b)$ and $\lambda\mapsto\mathcal{D}_{\infty}(\lambda,b).$
\end{enumerate}
\end{rem}
\begin{proof}
 $(i)$ From the regularity of the modified Bessel functions, we deduce that the application $(\lambda,b)\mapsto \mathcal{D}_n(\lambda,b)=\det\big(M_n(\lambda,b)\big)$ is smooth (actually analytic) on $(0,\infty)\times(0,\infty)$ so in particular on $(0,\infty)\times(0,1].$ Using \eqref{asymptotic expansion of small argument}, we see that the above expression can also be prolongated by continuity for $b=0$ and $\lambda=0.$ In addition, the evaluations \eqref{lim det} follow from the asymptotics \eqref{asymptotic expansion of small argument}-\eqref{asymp large z}.\\
 $(ii)$ We can write
    \begin{equation}\label{decomp:Dn}
        \mathcal{D}_{n}(\lambda,b)=\mathcal{D}_{\infty}(\lambda,b)-I_n(\lambda)K_n(\lambda)\big[\Lambda_1(\lambda,b)-b\Omega_n(\lambda b)\big]+bI_n(\lambda b)K_n(\lambda b)\big[\Omega_n(\lambda)-b\Lambda_1(\lambda,b)\big]+b\Lambda_n^2(\lambda,b).
    \end{equation}
    First remark that, from \eqref{def:Lbd and Omg},
\begin{align*}
	\|\Omega_{n}-I_1K_1\|_{L^{\infty}([0,\infty))}&=\|I_nK_n\|_{L^{\infty}([0,\infty))}=\tfrac{1}{2n}\underset{n\to\infty}{\longrightarrow}0.
\end{align*}
Notice that the last equality is obtained from the decay property of $x\mapsto I_n(x)K_n(x)$ on $(0,\infty)$ together with the asymptotic \eqref{asymptotic expansion of small argument}. Similarly, from \eqref{def:Lbd and Omg} and the fact that $I_n$ is increasing, we deduce that
\begin{align*}
	\sup_{(\lambda,b)\in[0,\infty)\times[0,1]}|\Lambda_{n}(\lambda,b)|&\leqslant\|I_nK_n\|_{L^{\infty}([0,\infty))}=\tfrac{1}{2n}\underset{n\to\infty}{\longrightarrow}0.
\end{align*}
From what precedes, we have that $(\Omega_n)_{n\in\mathbb{N}^*}$ and $(\Lambda_n)_{n\in\mathbb{N}^*}$ are two sequences of uniformly bounded functions converging uniformly on $[0,\infty)$ and $[0,\infty)\times [0,1]$ toward $I_1K_1$ and $0$, respectively. Therefore, one obtains the uniform convergence \eqref{unifCV:D-2var}.\\
$(iii)$ Since $I_{1}$ is strictly increasing and strictly positive on $(0,\infty)$, we have 
\begin{equation}\label{comp I1}
	\forall (\lambda,b)\in(0,\infty)\times (0,1),\quad bI_{1}(\lambda b)<I_{1}(\lambda b)<I_{1}(\lambda).
\end{equation}
Now for a fixed $\lambda>0,$ we consider the function $\psi_{\lambda}$ defined by
$$\forall b\in(0,1],\quad \psi_{\lambda}(b)\triangleq bK_{1}(\lambda b)-K_{1}(\lambda).$$
By using \eqref{Bessel derivatives} and the fact that $K_{0}$ is strictly positive on $(0,\infty),$ we get 
$$\psi_{\lambda}'(b)=K_{1}(\lambda b)+\lambda bK_{1}'(\lambda b)=-\lambda bK_0(\lambda b)<0.$$
Hence $\psi_{\lambda}$ is strictly decreasing on $(0,1]$ with $\psi_{\lambda}(1)=0.$ Thus, $\psi_{\lambda}$ is strictly positive on $(0,1)$ which implies
\begin{equation}\label{comp K1}
	\forall (\lambda,b)\in(0,\infty)\times (0,1),\quad bK_1(\lambda b)>K_1(\lambda).
\end{equation}
Therefore, \eqref{sign Dinfty} holds. We now turn to the second order derivative. It is clear from \eqref{def:Dinfty} that differentiating twice with respect to $b$, the surviving term when evaluating at $b=1$ is
\begin{equation}\label{d2D8}
    \begin{aligned}
    \partial_b^2\mathcal{D}_{\infty}(\lambda,1)&=\Lambda_1(\lambda,1)\Big(\partial_b\big(bI_1(\lambda b)\big)|_{b=1}\Big)\Big(\partial_b\big(bK_1(\lambda b)\big)|_{b=1}\Big)\\
    &=I_1(\lambda)K_1(\lambda)\Big(I_1(\lambda)+\lambda I_1'(\lambda)\Big)\Big(K_1(\lambda)+\lambda K_1'(\lambda)\Big).
\end{aligned}
\end{equation}
Making appeal to \eqref{Bessel derivatives}, we have
\begin{equation}\label{useful:diff}
    I_1'(\lambda)=I_0(\lambda)-\frac{I_1(\lambda)}{\lambda}\qquad\textnormal{and}\qquad K_1'(\lambda)=-K_0(\lambda)-\frac{K_1(\lambda)}{\lambda}\cdot
\end{equation}
Inserting \eqref{useful:diff} into \eqref{d2D8}, we find \eqref{diff2Dinfty}. This achieves the proof of Lemma \ref{lem:elemD}. 
\end{proof}

\subsubsection{Fredholmness of the linearized operator at the equilibrium}
In this section, we show that the linearized operator at the equilibrium is Fredholm of zero index. This is done by writing the linear operator as a compact perturbation of an isomorphism. \begin{prop}\label{prop-fredholmness}
For any $\mathbf{m}\in\mathbb{N}\setminus\{0,1\}$ and $\alpha\in(0,1)$, the linear operator $D G(\lambda,b,0,0):X_{\mathbf{m}}^{1+\alpha}\rightarrow Y_{\mathbf{m}}^\alpha$ is Fredholm of zero index. 
\end{prop}

\begin{proof}
Recall the definition of $G$ in \eqref{def:G}. Hence
\begin{align*}
\partial_{f_1}G_1(\lambda,b,0,0)[h_1](w)&=\textnormal{Im}\left\{\left[S(\lambda,bw,w)(w)-S(\lambda,w,w)\right]\overline{w}\overline{h_1'(w)}\right\}\\
&\quad+\textnormal{Im}\left\{\left[\lambda b \fint_{\mathbb{T}}K_0'(\lambda|w-b\tau|)\frac{(w-b\tau)\cdot h_1(w)}{|w-b\tau|}d\tau\right.\right.\\
&\qquad\qquad\quad-\lambda  \fint_{\mathbb{T}}K_0'(\lambda|w-\tau|)\frac{(w-\tau)\cdot (h_1(w)-h_1(\tau))}{|w-\tau|}d\tau\\
&\qquad\qquad\quad\left.\left.-\fint_{\T} h_1'(\tau)K_0(\lambda|w-\tau|)d\tau \right]\overline{w}\right\}.
\end{align*}
Using the change of variables $\tau\mapsto w\tau$, the fact that $|w|=1$ and Beltrami's summation formula \eqref{Beltrami's summation formula}, proceeding as in \cite{R21}, we find
\begin{align*}
S(\lambda,bw,w)(w)=&bw\fint_{\T}K_0(\lambda|1-b\tau|)d\tau={bw}I_1(\lambda b)K_1(\lambda),\\
S(\lambda,w,w)(w)=&w\fint_{\T}K_0(\lambda|1-\tau|)d\tau={w}I_1(\lambda)K_1(\lambda).
\end{align*}
For the details of the last two identities, we refer the reader to \cite{R21}. That amounts to
\begin{align*}
\partial_{f_1}G_1(\lambda,b,0,0)[h_1](w)&=\left[I_1(\lambda)K_1(\lambda)-bI_1(\lambda b)K_1(\lambda)\right]\textnormal{Im}\left\{h_1'(w)\right\}\\
&\quad+\textnormal{Im}\left\{\left[\lambda b \fint_{\mathbb{T}}K_0'(\lambda|w-b\tau|)\frac{(w-b\tau)\cdot h_1(w)}{|w-b\tau|}d\tau\right.\right.\\
&\qquad\qquad\quad-\lambda  \fint_{\mathbb{T}}K_0'(\lambda|w-\tau|)\frac{(w-\tau)\cdot (h_1(w)-h_1(\tau))}{|w-\tau|}d\tau\\
&\qquad\qquad\quad\left.\left.-\fint_{\T} h_1'(\tau)K_0(\lambda|w-\tau|)d\tau \right]\overline{w}\right\}.
\end{align*}
We can move to $\partial_{f_2}G_1(\lambda,b,0,0)$ to find
\begin{align*}
\partial_{f_2}G_1(\lambda,b,0,0)[h_2](w)=&\textnormal{Im}\left\{\left[\fint_{\T}h_2'(\tau)K_0(\lambda|w-b\tau|)d\tau-b\fint_{\T}K_0'(\lambda|w-b\tau|)\frac{(w-b\tau)\cdot h_2(\tau)}{|w-b\tau|}d\tau\right]\overline{w}\right\}.
\end{align*}
Similarly we get
\begin{align*}
\partial_{f_1}G_2(\lambda,b,0,0)[h_1](w)=&-\textnormal{Im}\left\{\left[\fint_{\T}h_1'(\tau)K_0(\lambda|bw-\tau|)d\tau-\fint_{\T}K_0'(\lambda|bw-\tau|)\frac{(bw-\tau)\cdot h_1(\tau)}{|bw-\tau|}d\tau\right]\overline{w}\right\}.
\end{align*}
Moreover, since
\begin{align*}
S(\lambda, bw, bw)(w)=&b\fint_{\mathbb{T}} K_0(\lambda b|w-\tau|)d\tau=bw\fint_{\mathbb{T}} K_0(\lambda b|1-\tau|)d\tau=bwI_1(\lambda b)K_1(\lambda b),\\
S(\lambda, w, bw)(w)=&\fint_{\mathbb{T}} K_0(\lambda|bw-\tau|)d\tau=w\fint_{\mathbb{T}} K_0(\lambda |b-\tau|)d\tau=wI_1(\lambda b)K_1(\lambda),
\end{align*}
we get
\begin{align*}
\partial_{f_2}G_2(\lambda,b,0,0)[h_2](w)=&\left[I_1(\lambda b)K_1(\lambda)-bI_1(\lambda b)K_1(\lambda b)\right]\textnormal{Im}\left\{h_2'(w)\right\}\\
&+b\textnormal{Im}\left\{\left[b\lambda  \fint_{\mathbb{T}}K_0'(\lambda b|w-\tau|)\frac{(w-\tau)\cdot (h_2(w)-h_2(\tau))}{|w-\tau|}d\tau\right.\right.\\
&\qquad\qquad\quad+\fint_{\T}h_2'(\tau)K_0(\lambda b|w-\tau|)d\tau\\
&\qquad\qquad\quad\left.\left.-\lambda \fint_{\T}K_0'(\lambda |bw-\tau|)\frac{(bw-\tau)\cdot h_2(\tau)}{|bw-\tau|}d\tau\right]\overline{w}\right\}.
\end{align*}
Putting everything together we find 
\begin{align*}
DG_1(\lambda,b,0,0)[h_1,h_2](w)&=\left[I_1(\lambda)K_1(\lambda)-bI_1(\lambda b)K_1(\lambda)\right]\textnormal{Im}\left\{h_1'(w)\right\}\\
&\quad+\textnormal{Im}\left\{\left[\lambda b \fint_{\mathbb{T}}K_0'(\lambda|w-b\tau|)\frac{(w-b\tau)\cdot h_1(w)}{|w-b\tau|}d\tau\right.\right.\\
&\qquad\qquad\quad-\lambda  \fint_{\mathbb{T}}K_0'(\lambda|w-\tau|)\frac{(w-\tau)\cdot (h_1(w)-h_1(\tau))}{|w-\tau|}d\tau\\
&\qquad\qquad\quad-\fint_{\T} h_1'(\tau)K_0(\lambda|w-\tau|)d\tau+ \fint_{\T}h_2'(\tau)K_0(\lambda|w-b\tau|)d\tau\\
&\qquad\qquad\quad-\left.\left.b\fint_{\T}K_0'(\lambda|w-b\tau|)\frac{(w-b\tau)\cdot h_2(\tau)}{|w-b\tau|}d\tau\right]\overline{w}\right\}
\end{align*}
and 
\begin{align*}
  DG_2(\lambda,b,0,0)[h_1,h_2](w)&=\left[I_1(\lambda b)K_1(\lambda)-bI_1(\lambda b)K_1(\lambda b)\right]\textnormal{Im}\left\{h_2'(w)\right\}\\
&\quad+b\textnormal{Im}\left\{\left[b\lambda  \fint_{\mathbb{T}}K_0'(\lambda b|w-\tau|)\frac{(w-\tau)\cdot (h_2(w)-h_2(\tau))}{|w-\tau|}d\tau\right.\right.\\
&\qquad\qquad\quad+\fint_{\T}h_2'(\tau)K_0(\lambda b|w-\tau|)d\tau-\fint_{\T}h_1'(\tau)K_0(\lambda|bw-\tau|)d\tau\\
&\qquad\qquad\quad-\lambda \fint_{\T}K_0'(\lambda |bw-\tau|)\frac{(bw-\tau)\cdot h_2(\tau)}{|bw-\tau|}d\tau\\
&\qquad\qquad\quad\left.\left.-\fint_{\T}K_0'(\lambda|bw-\tau|)\frac{(bw-\tau)\cdot h_1(\tau)}{|bw-\tau|}d\tau\right]\overline{w}\right\}.  
\end{align*}
 Thus, we can decompose the linear operator as 
$$\partial_{(f_1,f_2)}G(\lambda,b,0,0)=\mathcal{L}+\mathcal{K},$$ 
with $\mathcal{L}\triangleq(\mathcal{L}_1,\mathcal{L}_2)$ and $\mathcal{K}\triangleq(\mathcal{K}_1,\mathcal{K}_2)$, where
\begin{equation}\label{isoL12}
    \begin{aligned}
   \mathcal{L}_1[h_1](w)&\triangleq\left[I_1(\lambda)K_1(\lambda)-bI_1(\lambda b)K_1(\lambda)\right]\textnormal{Im}\left\{h_1'(w)\right\},\\
   \mathcal{L}_2[h_2](w)&\triangleq\left[I_1(\lambda b)K_1(\lambda)-bI_1(\lambda b)K_1(\lambda b)\right]\textnormal{Im}\left\{h_2'(w)\right\}
\end{aligned}
\end{equation}
and
\begin{align*}
    \mathcal{K}_1[h_1,h_2](w)&\triangleq\textnormal{Im}\left\{\left[\lambda b \fint_{\mathbb{T}}K_0'(\lambda|w-b\tau|)\frac{(w-b\tau)\cdot h_1(w)}{|w-b\tau|}d\tau\right.\right.\\
&\qquad\qquad\quad-\fint_{\T} h_1'(\tau)K_0(\lambda|w-\tau|)d\tau+\fint_{\T}h_2'(\tau)K_0(\lambda|w-b\tau|)d\tau\\
&\qquad\qquad\quad-\lambda  \fint_{\mathbb{T}}K_0'(\lambda|w-\tau|)\frac{(w-\tau)\cdot (h_1(w)-h_1(\tau))}{|w-\tau|}d\tau\\
&\qquad\qquad\quad\left.\left.-b\fint_{\T}K_0'(\lambda|w-b\tau|)\frac{(w-b\tau)\cdot h_2(\tau)}{|w-b\tau|}d\tau\right]\overline{w}\right\},\\
    \mathcal{K}_2[h_1,h_2](w)&\triangleq b\textnormal{Im}\left\{\left[b\lambda  \fint_{\mathbb{T}}K_0'(\lambda b|w-\tau|)\frac{(w-\tau)\cdot (h_2(w)-h_2(\tau))}{|w-\tau|}d\tau\right.\right.\\
&\qquad\qquad\quad+\fint_{\T}h_2'(\tau)K_0(\lambda b|w-\tau|)d\tau-\fint_{\T}h_1'(\tau)K_0(\lambda|bw-\tau|)d\tau\\
&\qquad\qquad\quad-\lambda \fint_{\T}K_0'(\lambda |bw-\tau|)\frac{(bw-\tau)\cdot h_2(\tau)}{|bw-\tau|}d\tau\\
&\qquad\qquad\quad\left.\left.-\fint_{\T}K_0'(\lambda|bw-\tau|)\frac{(bw-\tau)\cdot h_1(\tau)}{|bw-\tau|}d\tau\right]\overline{w}\right\}.
\end{align*}
Let us check that $\mathcal{L}:X_{\mathbf{m}}^{1+\alpha}\rightarrow Y_{\mathbf{m}}^\alpha$ is an isomorphism and $\mathcal{K}:X_{\mathbf{m}}^{1+\alpha}\rightarrow Y_{\mathbf{m}}^\alpha$ is a compact operator. In that case, since compact perturbations of Fredholm operators remain Fredholm with same index, then we can conclude that the linearized operator $DG(\lambda,b,0,0):X_{\mathbf{m}}^{1+\alpha}\rightarrow Y_{\mathbf{m}}^\alpha$ is Fredholm of zero index, concluding the proof.\\
$\blacktriangleright$ \textbf{Isomorphism component :}  
    First, let us begin by checking that $\mathcal{L}$ is an isomorphism. By the definition of the function spaces we can write $h_1$ and $h_2$ in Fourier series as
    $$
    h_1(w)=\sum_{n=1}^\infty h_{1,n}\overline{w}^{n\mathbf{m}-1}, \qquad h_2(w)=\sum_{n=1}^\infty h_{2,n}\overline{w}^{n\mathbf{m}-1},
    $$
    implying
      $$
    \textnormal{Im}\left\{h_1'(w)\right\}=\sum_{n=1}^\infty h_{1,n}(n\mathbf{m}-1)e_{n\mathbf{m}}(w), \qquad \textnormal{Im}\left\{h_2'(w)\right\}=\sum_{n=1}^\infty h_{2,n}(n\mathbf{m}-1)e_{n\mathbf{m}}(w).
    $$
    In that way, we find
    \begin{align*}
        \mathcal{L}_1[h_1]=&\left[I_1(\lambda)K_1(\lambda)-bI_1(\lambda b)K_1(\lambda)\right]\sum_{n=1}^\infty h_{1,n}(n\mathbf{m}-1)e_{n\mathbf{m}},\\
         \mathcal{L}_2[h_2]=&\left[I_1(\lambda b)K_1(\lambda)-bI_1(\lambda b)K_1(\lambda b)\right]\sum_{n=1}^\infty h_{2,n}(n\mathbf{m}-1)e_{n\mathbf{m}}.
    \end{align*}    
    Note that in view of \eqref{comp I1}, \eqref{comp K1} and the fact that the modified Bessel function are strictly positive, we obtain
    $$I_1(\lambda)K_1(\lambda)-bI_1(\lambda b)K_1(\lambda)\neq 0\qquad\textnormal{and}\qquad I_1(\lambda b)K_1(\lambda)-bI_1(\lambda b)K_1(\lambda b)\neq 0.$$
    Take now $(d_1,d_2)\in Y_{\mathbf{m}}^\alpha$ and let us solve
    \begin{equation}\label{fred-iso}
    \mathcal{L}[h_1,h_2]=(d_1,d_2).
    \end{equation}
    Writing $d_j$ in Fourier as
    $$
    d_j(w)=\sum_{n=1}^\infty d_{j,n}e_{n\mathbf{m}}(w),
    $$
    we find that
    \begin{align*}
    h_{1,n}=&\frac{1}{I_1(\lambda)K_1(\lambda)-bI_1(\lambda b)K_1(\lambda)}\frac{d_{1,n}}{n\mathbf{m}-1},\\
     h_{2,n}=&\frac{1}{I_1(\lambda b)K_1(\lambda)-bI_1(\lambda b)K_1(\lambda b)}\frac{d_{2,n}}{n\mathbf{m}-1}\cdot
    \end{align*}
    That implies that the solution of \eqref{fred-iso} agrees with
    \begin{align*}
    h_1(w)=&\frac{1}{I_1(\lambda)K_1(\lambda)-bI_1(\lambda b)K_1(\lambda)} \sum_{n=1}^\infty \frac{d_{1,n}}{n\mathbf{m}-1} \overline{w}^{nm-1},\\
     h_2(w)=&\frac{1}{I_1(\lambda b)K_1(\lambda)-bI_1(\lambda b)K_1(\lambda b)} \sum_{n=1}^\infty \frac{d_{2,n}}{nm-1} \overline{w}^{n\mathbf{m}-1}.
    \end{align*}
    Let us work with the first equation for $h_1$. We have
    $$
    h_1'(w)=-\frac{1}{I_1(\lambda)K_1(\lambda)-bI_1(\lambda b)K_1(\lambda)} \sum_{n=1}^\infty d_{1,n} \overline{w}^{n\mathbf{m}}.
    $$
    Moreover, since $d_1\in C^\alpha(\mathbb{T})$, we know that its Szeg\" o projection 
    $$d_1^+(w)\triangleq\sum_{n=1}^\infty d_{1,n} w^{n\mathbf{m}}$$ belongs also to $C^\alpha(\mathbb{T})$, and thus
    $$
    h_1'(w)=-\frac{1}{I_1(\lambda)K_1(\lambda)-bI_1(\lambda b)K_1(\lambda)} \overline{d_1^+(w)} \in C^\alpha(\mathbb{T}).
    $$
    A similar analysis can be performed for $h_2$ obtaining that $h_1,h_2\in C^{1+\alpha}(\mathbb{T})$ and thus $\mathcal{L}$ is a isomorphism.\\
    $\blacktriangleright$ \textbf{Compact component :}
   In order to get that $\mathcal{K}$ is a compact operator we shall check that $\mathcal{K}[h_1,h_2]\in C^\delta(\mathbb{T})$, for any $\delta\in(\alpha,1)$. Choosing $\delta>\alpha$ and using that the embedding $C^\delta(\mathbb{T})\subset C^\alpha(\mathbb{T})$ is compact, hence we get that $\mathcal{K}$ is a compact operator. To check such regularity we make use of Lemma \ref{Lem-pottheory}. 
   Let us give the idea of two terms. Define
   $$
   \mathcal{T}_1[h_1](w)\triangleq\fint_{\T}K_0'(\lambda|w-b\tau|) \frac{(w-b\tau)\cdot h_1(w)}{|w-b\tau|}d\tau. 
   $$
   Note that this term is related to the interaction between the two boundaries. Moreover, since the boundaries are well-separated, meaning $b<1$, we get that the function inside the integral is never singular and thus it is clear that it belongs to $C^1(\mathbb{T}).$ In particular,
   $$\forall\,\delta\in(\alpha,1),\quad\left\|T_1[h_1]\right\|_{C^\delta(\mathbb{T})}\leqslant C \|h_1\|_{L^\infty(\mathbb{T})}\leqslant C\|h_1\|_{C^{1+\alpha}(\mathbb{T})}.$$
Let us now choose another term whose function inside the integral is singular, for instance
   $$\mathcal{T}_2[h_1](w)\triangleq\fint_{\T} K_0'(\lambda|w-\tau|)\frac{(w-\tau)\cdot (h_1(w)-h_1(\tau))}{|w-\tau|}d\tau.$$
   Here, we use Lemma \ref{Lem-pottheory} and one just have to check that the kernel
   $$K(w,\tau)\triangleq K_0'(\lambda|w-\tau|)\frac{(w-\tau)\cdot (h_1(w)-h_1(\tau))}{|w-\tau|},$$
   satisfies the hypothesis of the lemma. Indeed, using the logarithmic behaviour of $K_0$ at $0$ we find
   \begin{align*}
    |K(w,\tau)|&\leqslant C \|h_1'\|_{L^{\infty}(\mathbb{T})}\leqslant C\|h_1\|_{C^{1+\alpha}(\mathbb{T})},\\
     \left|\partial_ wK(w,\tau)\right|&\leqslant C \frac{\|h_1'\|_{L^\infty(\mathbb{T})}}{|w-\tau|}\leqslant C \frac{\|h_1\|_{C^{1+\alpha}(\mathbb{T})}}{|w-\tau|}\cdot
   \end{align*}
   Applying Lemma \ref{Lem-pottheory} we conclude that $\mathcal{T}_2[h_1]\in C^\delta(\mathbb{T})$, for any $\delta \in(\alpha,1)$. This completes the proof of Proposition \ref{prop-fredholmness}.
    \end{proof}

\subsection{Bifurcation with the inner radius} \label{sec:bif_b}
In this subsection, we assume that $\lambda\in(0,\infty)$ is fixed and we denote
$$\mathcal{D}_{n,\lambda}(b)\triangleq\mathcal{D}_{n}(\lambda,b),\qquad\mathcal{D}_{\infty,\lambda}(b)\triangleq\mathcal{D}_{\infty}(\lambda,b),\qquad G_{\lambda}(b,f_1,f_2)\triangleq G(\lambda,b,f_1,f_2).$$
In order to find a nontrivial kernel of the linearized operator, which is needed for the bifurcation argument, our main task here in the following question:
\begin{center}
 {\it   Do there exist zeros $b$ of $\mathcal{D}_{n,\lambda}$ in $(0,1)$ ?}
\end{center}
The following numerical simulations answer positively to this question and moreover the zero seems to be unique. 

 \begin{figure}[!ht]
 	\centering
 	\begin{subfigure}[l]{0.3\textwidth}
 		\includegraphics[width=\textwidth]{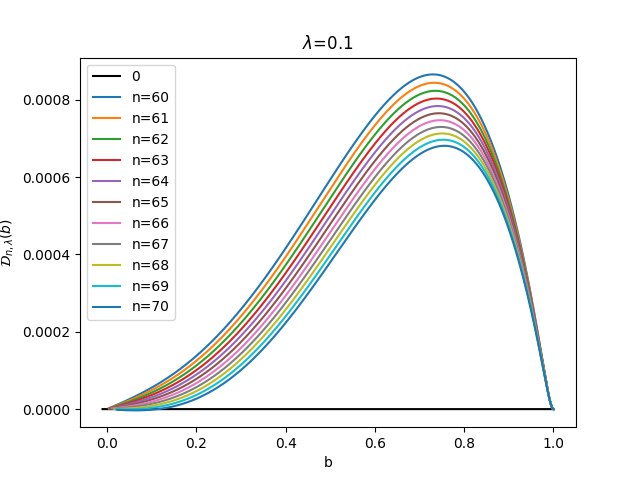}
	\end{subfigure}
	\begin{subfigure}[c]{0.3\textwidth}
 		\includegraphics[width=\textwidth]{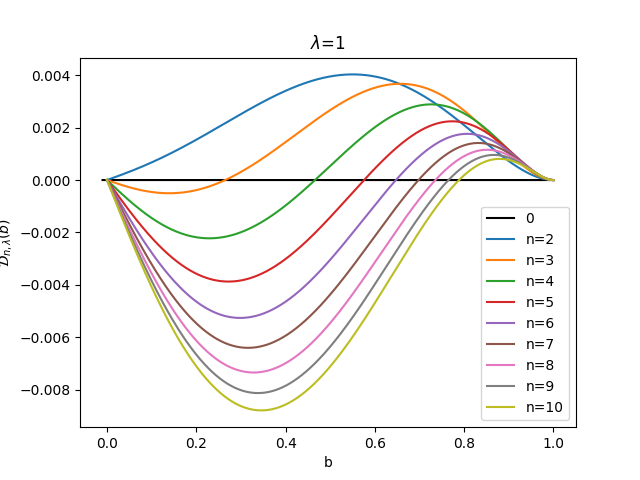}
 	\end{subfigure}
 	\begin{subfigure}[r]{0.3\textwidth}
 		\includegraphics[width=\textwidth]{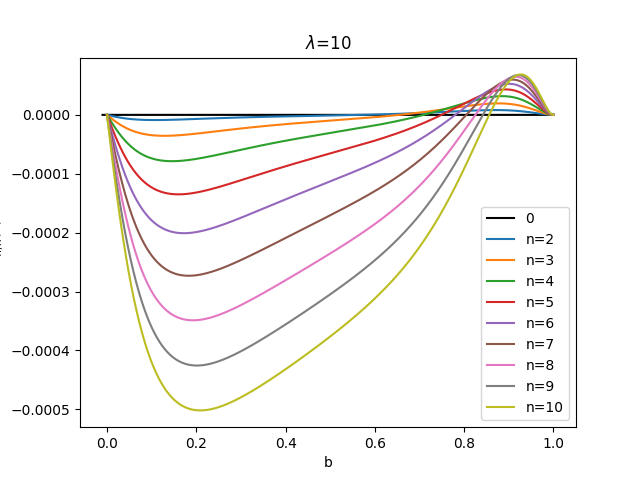}
 	\end{subfigure}
 	\caption{Graphs of $b\mapsto\mathcal{D}_{n,\lambda}(b)$ for $\lambda\in\{0.1,1,10\}$ and for different values of $n.$}
 \end{figure}
 \begin{figure}[!ht]
 	\centering
 	\begin{subfigure}[l]{0.3\textwidth}
 		\includegraphics[width=\textwidth]{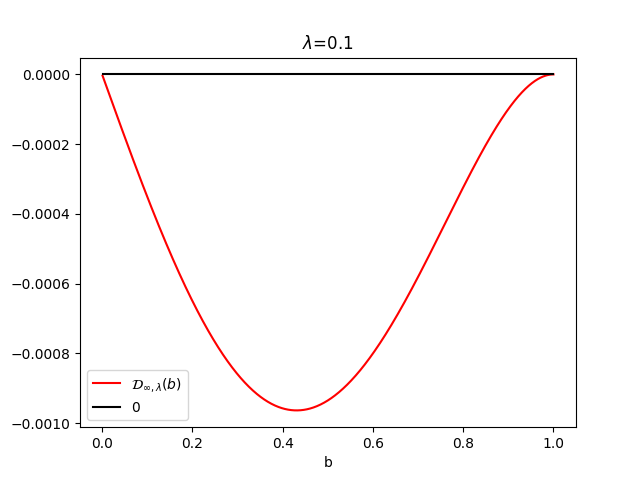}
 	\end{subfigure}
 	\begin{subfigure}[c]{0.3\textwidth}
 		\includegraphics[width=\textwidth]{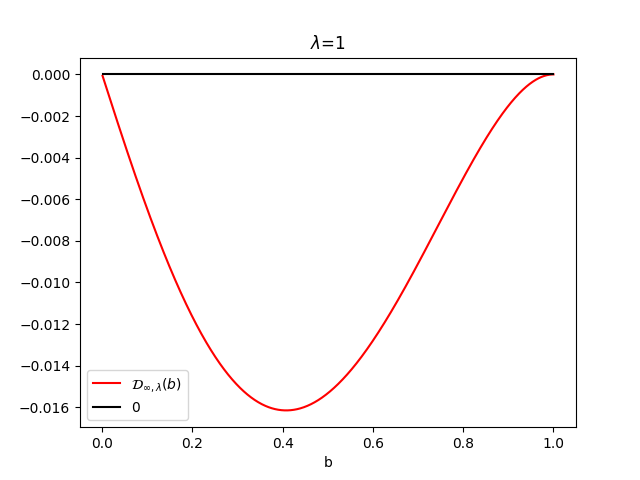}
 	\end{subfigure}
 	\begin{subfigure}[r]{0.3\textwidth}
 		\includegraphics[width=\textwidth]{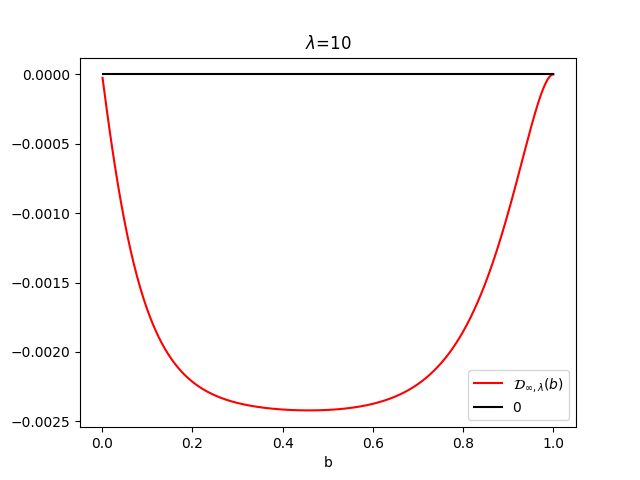}
 	\end{subfigure}
 	\caption{Graphs of the limiting profile $b\mapsto\mathcal{D}_{\infty,\lambda}(b)$ for $\lambda\in\{0.1,1,10\}.$}\label{detlimb}
 \end{figure}

Our first goal is to prove the following analytical result, validating the above mentioned numerical observations in the asymptotic $n\to \infty$.
\begin{prop}\label{prop seq b}
There exists $N(\lambda)\in\mathbb{N}^{*}$ such that, for all $n\in\mathbb{N}^{*}$, $n\geqslant N(\lambda)$, the equation 
$$\mathcal{D}_{n,\lambda}(b)=0,$$
admits a unique solution $b_{n,\lambda}\in(0,1).$ In addition, the zero $b_{n,\lambda}$ is simple. Moreover, the sequence $(b_{n,\lambda})_{n\geqslant N(\lambda)}$ is strictly increasing and converges to $1.$
\end{prop}

Proposition \ref{prop seq b} is a consequence of Lemmas \ref{lem:cvDnb}, \ref{lem:accu1}, \ref{lem:asybn} and \ref{lem:asyan} below. These are proved making extensive use of modified Bessel functions properties. Let us start by stating some results of the function $\mathcal{D}_{n,\lambda}$ and its first derivatives.

\begin{lem}\label{lem:cvDnb}
The determinant function $\mathcal{D}_{n,\lambda}$ enjoys the following properties.
\begin{enumerate}[label=(\roman*)]
    \item For any $n\in\mathbb{N}^*,$
    \begin{equation}\label{Dn01}
        \mathcal{D}_{n,\lambda}(0)=0=\mathcal{D}_{n,\lambda}(1).
    \end{equation}
    Moreover, we have the following uniform convergence
    $$\lim_{n\to\infty}\|\mathcal{D}_{n,\lambda}-\mathcal{D}_{\infty,\lambda}\|_{L^{\infty}([0,1])}=0.$$
    \item For any $n\in\mathbb{N}^*,$
    \begin{equation}\label{Dnprime01}
        \mathcal{D}_{n,\lambda}'(1)=0,\qquad\mathcal{D}_{n,\lambda}'(0)=\Omega_{n}(\lambda)\Omega_{n}^{-}(\lambda),\qquad\Omega_{n}^{-}(\lambda)\triangleq\frac{n(\lambda K_{1}(\lambda)-1)+1}{2n}\cdot
    \end{equation}
    Therefore,
    \begin{equation}\label{limDnprime0}
        \lim_{n\to\infty}\mathcal{D}_{n,\lambda}'(0)=\tfrac{1}{2}I_1(\lambda)K_1(\lambda)\big(\lambda K_{1}(\lambda)-1\big)<0.
    \end{equation}
    Moreover, we have the following uniform convergence
    \begin{equation}\label{unifCV:Dprime}
        \lim_{n\to\infty}\|\mathcal{D}_{n,\lambda}'-\mathcal{D}_{\infty,\lambda}'\|_{L^{\infty}([0,1])}=0.
    \end{equation}
    \item The sequence $\big(\mathcal{D}_{n,\lambda}''(1)\big)_{n\in\mathbb{N}^*}$ converges and the following limit holds
    \begin{equation}\label{limit:Dn2}
        \lim_{n\to\infty}\mathcal{D}_{n,\lambda}''(1)=1-2\lambda^2I_0(\lambda)K_0(\lambda)I_1(\lambda)K_1(\lambda)\in(0,1).
        \end{equation}
    \item For any $b\in(0,1],$ the following asymptotic holds
    \begin{equation}\label{asymp:Dn3}
        \mathcal{D}_{n,\lambda}'''(b)\underset{n\to\infty}{=}2nb^{2n-2}+O\left(b^{2n-2}\right).
    \end{equation}
\end{enumerate}
\end{lem}
\begin{proof}
$(i)$ This point is just a remainder of \eqref{lim det} and \eqref{unifCV:D-2var}.\\
$(ii)$ We start by recalling the decomposition \eqref{decomp:Dn}. Referring to \cite[pp. 27-28]{HR21}, we have that for any $n,k\in\mathbb{N}$ with $n>\frac{k}{2}$,
\begin{equation}\label{decay:InKnHR}
    |\partial_x^k(I_nK_n)(x)|\leqslant\frac{1}{2n-k}\cdot
\end{equation}
    Besides, it is immediate from \eqref{deriv:iterate} that, for any $k\in\mathbb{N}$,
    $$\partial_{b}^{k}\Lambda_{n}(\lambda,b)=\lambda^k I_n^{(k)}(\lambda b)K_n(\lambda)=\left(\frac{\lambda}{2}\right)^{k}K_{n}(\lambda)\sum_{p=0}^{k}\binom{k}{p}I_{n-k+2p}(\lambda b).$$
    In addition, from \eqref{asymp:large-order-In} and \eqref{asymp:large-order-Kn}, we find, for any $k\in\mathbb{Z}$,
    $$I_{n-k}(\lambda b)K_n(\lambda)\underset{n\to\infty}{=}\frac{b^{n-k}}{2}\left(\frac{\lambda}{2}\right)^{-k}n^{k-1}+O\left(b^{n-k}n^{k-2}\right).$$
    Combining the last computations gives, for any $k\in\mathbb{N}$,
    \begin{equation}\label{svi:diffkLbd}
        \partial_{b}^{k}\Lambda_{n}(\lambda,b)\underset{n\to\infty}{=}\frac{b^{n-k}}{2}n^{k-1}+O\left(b^{n-k}n^{k-2}\right).
    \end{equation}
    From \eqref{svi:diffkLbd}, we deduce that
    \begin{equation}\label{svi:diffbL}
        \partial_b\Big(b\Lambda_n^2(\lambda,b)\Big)=\Lambda_n^2(\lambda,b)+2b\Lambda_n(\lambda,b)\partial_b\Lambda_n(\lambda,b)\underset{n\to\infty}{=}\frac{b^{2n}}{2n}+O\left(\frac{b^{2n}}{n^2}\right).
    \end{equation}
    Note that only the last term contributes to the main order asymptotic. Combining \eqref{decomp:Dn}, \eqref{decay:InKnHR} and \eqref{svi:diffbL}, we obtain the uniform convergence \eqref{unifCV:Dprime}.\\
    $(ii)$ One readily has
\begin{align*}
	\mathcal{D}_{n,\lambda}'(b) & =  \big[\Lambda_{1}(\lambda,b)+b\partial_{b}\Lambda_{1}(\lambda,b)\big]\big[b\Omega_{n}(\lambda b)-\Lambda_{1}(\lambda,b)\big]\nonumber\\
	&  \quad+\big[\Omega_{n}(\lambda)-b\Lambda_{1}(\lambda,b)\big]\big[\partial_{b}\Lambda_{1}(\lambda,b)-\Omega_{n}(\lambda b)-b\partial_{b}(\Omega_{n}(\lambda b))\big]\nonumber\\
	&  \quad+\Lambda_{n}^{2}(\lambda,b)+2b\Lambda_{n}(\lambda,b)\partial_{b}\Lambda_{n}(\lambda,b).
\end{align*}
By using \eqref{Bessel derivatives}, we obtain 
\begin{equation*}
	\partial_{b}\Lambda_{n}(\lambda,b)=\lambda I_{n-1}(\lambda b)K_{n}(\lambda)-\frac{n\Lambda_{n}(\lambda,b)}{b}
\end{equation*}
and
\begin{align*}
	\partial_{b}\big(\Omega_{n}(\lambda b)\big)&=\lambda\big[I_{0}(\lambda b)K_{1}(\lambda b)-I_{1}(\lambda b)K_{0}(\lambda b)+I_{n}(\lambda b)K_{n-1}(\lambda b)-I_{n-1}(\lambda b)K_{n}(\lambda b)\big]\nonumber\\
	&\quad+\frac{2}{b}\big[nI_{n}(\lambda b)K_{n}(\lambda b)-I_{1}(\lambda b)K_{1}(\lambda b)\big].
\end{align*}
According to the asymptotic \eqref{asymptotic expansion of small argument} and the fact that $K_{0}$ behaves like a logarithm at $0$, we deduce that $\mathcal{D}_{n,\lambda}'$ admits a limit at $0$, which implies that $\mathcal{D}_{n,\lambda}$ admits an extension of class $C^1$ to $[0,1]$ and  
\begin{equation}\label{expression:Dnprime0}
    \mathcal{D}_{n,\lambda}'(0)=\Omega_{n}(\lambda)\frac{n(\lambda K_{1}(\lambda)-1)+1}{2n}=\Omega_{n}(\lambda)\Omega_{n}^{-}(\lambda).
\end{equation}
Consider the function $\varphi$ defined by 
$$\forall \lambda>0,\quad\varphi(\lambda)\triangleq\lambda K_{1}(\lambda).$$
From \eqref{Bessel derivatives}, we get 
$$\varphi'(\lambda)=K_{1}(\lambda)+\lambda K_{1}'(\lambda)=-\lambda K_{0}(\lambda)<0.$$
Hence $\varphi$ is strictly decreasing on $(0,\infty).$ Moreover, in view of the asymptotic \eqref{asymptotic expansion of small argument}, we infer 
$$\lim_{\lambda\rightarrow 0^{+}}\varphi(\lambda)=1.$$
Thus,
\begin{equation}\label{e-K1}
	\forall\lambda>0,\quad\varphi(\lambda)\in(0,1).
\end{equation}
In particular,
$$1-\lambda K_1(\lambda)>0.$$
Passing to the limit $n\to\infty$ in \eqref{expression:Dnprime0}, we find
$$\lim_{n\to\infty}\mathcal{D}_{n,\lambda}'(0)=\tfrac{1}{2}I_1(\lambda)K_1(\lambda)\big(\lambda K_1(\lambda)-1\big)<0.$$
On the other hand, by continuity of the modified Bessel functions on $(0,\infty),$ we obtain after straightforward simplifications
$$\mathcal{D}_{n,\lambda}'(1)=\lambda I_{n}(\lambda)K_{n}(\lambda)\Big[I_{n-1}(\lambda)K_{n}(\lambda)+I_{n}(\lambda)K_{n-1}(\lambda)-I_{0}(\lambda)K_{1}(\lambda)-I_{1}(\lambda)K_{0}(\lambda)\Big].$$
Combining the previous expression with \eqref{wronskian}, we finally get
$$\mathcal{D}_{n,\lambda}'(1)=I_{n}(\lambda)K_{n}(\lambda)(1-1)=0.$$
$(iii)$ Differentiating \eqref{svi:diffbL} and using \eqref{svi:diffkLbd}, we infer
\begin{equation}\label{svi:diff2bL}
    \begin{aligned}
        \partial_b^2\Big(b\Lambda_n^2(\lambda,b)\Big)&\,\,\,=4\Lambda_n(\lambda,b)\partial_b\Lambda_n(\lambda,b)+2b\big(\partial_b\Lambda_n(\lambda,b)\big)^2+2b\Lambda_n(\lambda,b)\partial_b^2\Lambda_n(\lambda,b)\\
        &\underset{n\to\infty}{=}b^{2n-1}+O\left(\frac{b^{2n-1}}{n}\right).
    \end{aligned}
\end{equation}
    Note that only the last two terms contribute to the main order asymptotic. In particular,
    \begin{equation}\label{diff2BL-at-1}
        \partial_b^2\Big(b\Lambda_n^2(\lambda,b)\Big)|_{b=1}\underset{n\to\infty}{\longrightarrow}1.
    \end{equation}
    Putting together \eqref{decomp:Dn}, \eqref{diff2Dinfty}, \eqref{decay:InKnHR} and \eqref{diff2BL-at-1}, we get the following pointwise convergence
    $$\mathcal{D}_{n,\lambda}''(1)\underset{n\to\infty}{\longrightarrow}\mathcal{D}_{\infty,\lambda}''(1)+1=1-2\lambda^2I_0(\lambda)K_0(\lambda)I_1(\lambda)K_1(\lambda).$$
    For simplicity, we denote
    $$P_n(x)\triangleq I_n(x)K_n(x)$$
and
    \begin{equation}\label{deff}
        f(\lambda)\triangleq1-2\lambda^2P_0(\lambda)P_1(\lambda).
    \end{equation}
    One easily has
    $$f'(\lambda)=-2\Big(2\lambda P_0(\lambda)P_1(\lambda)+\lambda^2\big(P_0'(\lambda)P_1(\lambda)+P_0(\lambda)P_1'(\lambda)\big)\Big).$$
    Using the relations \eqref{Bessel derivatives}, we find
    \begin{equation}\label{diffP0}
        \begin{aligned}
            P_0'&=I_0'K_0+I_0K_0'=I_1K_0-I_0K_1
        \end{aligned}
    \end{equation}
    and
    \begin{equation*}
        \begin{aligned}
            P_1'(\lambda)&=I_1'(\lambda)K_1(\lambda)+I_1(\lambda)K_1'(\lambda)\\
            &=\left(I_0(\lambda)-\frac{I_1(\lambda)}{\lambda}\right)K_1(\lambda)+I_1(\lambda)\left(-K_0(\lambda)-\frac{K_1(\lambda)}{\lambda}\right)\\
            &=-P_0'(\lambda)-\frac{2}{\lambda}P_1(\lambda).
        \end{aligned}
    \end{equation*}
    With this in hand we can rewrite $f'$ as
    $$f'(\lambda)=-2\lambda^2\big(P_1(\lambda)-P_0(\lambda)\big)P_0'(\lambda).$$
    Referring to Appendix \ref{appendix Bessel}, the function $(n,x)\mapsto P_n(x)$ is strictly decreasing in each argument. Therefore, $f'<0$ on $(0,\infty),$ which implies in turn that 
    \begin{equation}\label{decayf}
        f\textnormal{ is strictly decreasing on }(0,\infty).
    \end{equation} In addition, the asymptotics \eqref{asymptotic expansion of small argument} and \eqref{asymp large z} provide
    \begin{equation}\label{limitsf}
        \lim_{\lambda\to0}f(\lambda)=1\qquad\textnormal{and}\qquad\lim_{\lambda\to\infty}f(\lambda)=0.
    \end{equation}
    This proves that 
    $$\forall\lambda>0,\quad1-2\lambda^2I_0(\lambda)K_0(\lambda)I_1(\lambda)K_1(\lambda)=f(\lambda)\in(0,1).$$
    $(iv)$ Differentiating \eqref{svi:diff2bL} and using \eqref{svi:diffkLbd}, we infer
    \begin{align*}
        \partial_b^3\Big(b\Lambda_n^2(\lambda,b)\Big)&\,\,\,=6\big(\partial_b\Lambda_n(\lambda,b)\big)^2+6\Lambda_n(\lambda,b)\partial_b^2\Lambda_n(\lambda,b)+6b\partial_b\Lambda_n(\lambda,b)\partial_b^2\Lambda_n(\lambda,b)+2b\Lambda_n(\lambda,b)\partial_b^3\Lambda_n(\lambda,b)\\
        &\underset{n\to\infty}{=}2nb^{2n-2}+O(b^{2n-2}).
    \end{align*}
    Note that only the last two terms contribute to the main order asymptotic. This combined with \eqref{decomp:Dn} and \eqref{decay:InKnHR} gives \eqref{asymp:Dn3}.
    This concludes the proof of Lemma \ref{lem:cvDnb}.
\end{proof}

With this in hand, we can now study the existence of candidates bifurcation points. In what follows, we denote 
$$Z_{n,\lambda}\triangleq\Big\{b\in(0,1)\quad\textnormal{s.t.}\quad \mathcal{D}_{n,\lambda}(b)=0\Big\}.$$
We have the following result.
\begin{lem}\label{lem:accu1}
    There exists $N(\lambda)\in\mathbb{N}^*$ such that, for any $n\in\mathbb{N}^*$, $n\geqslant N(\lambda)$, we have
    \begin{equation}\label{nonvoid zeroset}
        Z_{n,\lambda}\neq\varnothing.
    \end{equation}
    Moreover, the set $Z_{n,\lambda}$ is discrete and
    $$\mathtt{diam}(Z_{n,\lambda}\cup\{1\})\underset{n\to\infty}{\longrightarrow}0,$$
    where $\mathtt{diam}$ is the diameter defined by
    $$\mathtt{diam}(A)\triangleq\sup_{a,b\in A}|a-b|.$$
    This means that the zeros of $\mathcal{D}_{n,\lambda}$ inside $(0,1)$ accumulate to $1$ as $n\to\infty.$
\end{lem}
\begin{proof}
$\blacktriangleright$ \textbf{Existence :} According to \eqref{Dn01} and \eqref{limDnprime0}, we have for $n$ large enough,
$$\mathcal{D}_{n,\lambda}(0)=0\qquad\textnormal{and}\qquad\mathcal{D}_{n,\lambda}'(0)<0.$$
Therefore, $\mathcal{D}_{n,\lambda}$ is strictly negative near $b=0.$ On the other hand, by virtue of \eqref{Dn01}, \eqref{Dnprime01} and \eqref{limit:Dn2}, we have that for $n$ large enough
$$\mathcal{D}_{n,\lambda}(1)=\mathcal{D}_{n,\lambda}'(1)=0\qquad\textnormal{and}\qquad\mathcal{D}_{n,\lambda}''(1)>0.$$
Therefore, $\mathcal{D}_{n,\lambda}$ is strictly positive near $b=1.$ The intermediate value theorem implies \eqref{nonvoid zeroset}. The fact that $Z_{n,\lambda}$ is discrete is a consequence of the real analyticity of the modified functions and the definition of $\mathcal{D}_{n,\lambda}$ in \eqref{def:Dn}.\\
$\blacktriangleright$ \textbf{Renormalized function :} We consider the function $F_{n,\lambda}$ defined on $[0,1]$ by 
    $$F_{n,\lambda}(b)\triangleq\begin{cases}
        \frac{\mathcal{D}_{n,\lambda}(b)}{b}, & \textnormal{if }b\in(0,1],\\
        \Omega_n(\lambda)\Omega_n^-(\lambda), & \textnormal{if }b=0.
    \end{cases}$$
    The function $F_{n,\lambda}$ is continuous on $(0,1]$ and prolongates by continuity in $0$ by virtue of \eqref{Dnprime01} and the l'H\^opital rule since
$$\lim_{b\to0^+}F_{n,\lambda}(b)=\mathcal{D}_n'(0)=\Omega_n(\lambda)\Omega_n^-(\lambda)=F_{n,\lambda}(0).$$
By construction, one has
    $$Z_{n,\lambda}=\Big\{b\in(0,1)\quad\textnormal{s.t.}\quad F_{n,\lambda}(b)=0\Big\}.$$
    Note that
$$\lim_{n\to\infty}F_{n,\lambda}(0)=\tfrac{1}{2}I_1(\lambda)K_1(\lambda)\big(\lambda K_1(\lambda)-1\big)<0.$$
$\blacktriangleright$ \textbf{Limiting profile :} We consider the function $F_{\infty,\lambda}$ defined on $[0,1]$ by
$$F_{\infty,\lambda}(b)\triangleq\begin{cases}
    \frac{\mathcal{D}_{\infty,\lambda}(b)}{b}, & \textnormal{if }b\in(0,1],\\
    \tfrac{1}{2}I_1(\lambda)K_1(\lambda)\big(\lambda K_1(\lambda)-1\big), & \textnormal{if }b=0.
\end{cases}$$
Using the l'H\^opital rule and \eqref{limDnprime0}-\eqref{unifCV:Dprime}, we find  
$$\lim_{b\to0^+}F_{\infty,\lambda}(b)=\mathcal{D}_{\infty,\lambda}'(0)=\tfrac{1}{2}I_1(\lambda)K_1(\lambda)\big(\lambda K_1(\lambda)-1\big)=F_{\infty,\lambda}(0).$$
Therefore, the function $F_{\infty,\lambda}$ is continuous on $[0,1].$ Besides, \eqref{e-K1} and \eqref{symmetry Bessel} imply that 
\begin{equation}\label{Finfty0}
    F_{\infty,\lambda}(0)<0.
\end{equation}
Moreover, making appeal to \eqref{sign Dinfty}, we have that
\begin{equation}\label{Finftyneg}
    \forall b\in(0,1),\quad F_{\infty,\lambda}(b)<0.
\end{equation}
$\blacktriangleright$ \textbf{Uniform convergence toward the limiting profile :}
    By Taylor's integral formula, one has 
$$F_{n,\lambda}(b)-F_{\infty,\lambda}(b)=\int_0^1\big(\mathcal{D}_{n,\lambda}'(\tau b)-\mathcal{D}_{\infty\lambda}'(\tau b)\big)d\tau.$$
Therefore, using the uniform convergence \eqref{unifCV:Dprime}, we find
$$\|F_{n,\lambda}-F_{\infty,\lambda}\|_{L^{\infty}([0,1])}\leqslant\|\mathcal{D}_{n,\lambda}'-\mathcal{D}_{\infty,\lambda}'\|_{L^{\infty}([0,1])}\underset{n\to\infty}{\longrightarrow}0.$$
This proves that $(F_{n,\lambda})_{n\in\mathbb{N}^*}$ converges uniformly towards $F_{\infty,\lambda}$ on $[0,1].$\\
$\blacktriangleright$ \textbf{Conclusion :}
 Fix $\varepsilon\in(0,1).$ By virtue of \eqref{Finfty0} and \eqref{Finftyneg}, one has
 $$\max_{b\in[0,1-\varepsilon]}F_{\infty,\lambda}(b)<0.$$
 Combining this with the uniform convergence, one obtains the existence of $N(\varepsilon,\lambda)\in\mathbb{N}^*$ such that
$$\forall n\in\mathbb{N}^*,\quad n\geqslant N(\varepsilon,\lambda)\quad\Rightarrow\quad\forall b\in[0,1-\varepsilon],\quad F_{n,\lambda}(b)<0.$$
This proves that
$$\forall n\geqslant N(\varepsilon,\lambda),\quad Z_{n,\lambda}\subset(1-\varepsilon,1).$$
The arbitrariness of $\varepsilon>0$ gives desired claim.
\end{proof}
In what follows, we consider $b_{n,\lambda}\in Z_{n,\lambda}$, for $n\geqslant N(\lambda)$, any possible zero of $\mathcal{D}_{n,\lambda}$.
In view of Lemma \ref{lem:accu1}, we can write
$$b_{n,\lambda}=1-\beta_{n,\lambda}\qquad\textnormal{with}\qquad\beta_{n,\lambda}>0\qquad\textnormal{and}\qquad\beta_{n,\lambda}\underset{n\to\infty}{\longrightarrow}0.$$
Our next goal is to give an asymptotic of $\beta_{n,\lambda}$ as $n\to\infty.$
\begin{lem}\label{lem:asybn}
    The following asymptotic holds
    \begin{equation}\label{asy:bn}
        \beta_{n,\lambda}\underset{n\to\infty}{=}\frac{\beta(\lambda)}{n}+O\left(\frac{1}{n^2}\right).
    \end{equation}
    The function $\beta:(0,\infty)\to(0,\infty)$ is strictly decreasing and satisfies the lower bound
    \begin{equation}\label{bndbeta}
        \beta(\lambda)>\frac{1}{2\lambda I_0(\lambda)K_{1}(\lambda)}>0.
    \end{equation}
\end{lem}
\begin{rem} \label{rem_beta_lambda}
    As we shall see in the proof, the number $\beta(\lambda)$ is explicitly given by the formula \eqref{def:betalbd}.
\end{rem}
\begin{proof}
    Recall from \eqref{Dn01} and \eqref{Dnprime01} that $\mathcal{D}_{n,\lambda}(1)=\mathcal{D}_{n,\lambda}'(1)=0.$ Therefore, Taylor's integral formula gives
    \begin{align*}
        0&=\mathcal{D}_{n,\lambda}(b_{n,\lambda})=\mathcal{D}_{n,\lambda}(1-\beta_{n,\lambda})=\frac{\beta_{n,\lambda}^2}{2}\mathcal{D}_{n,\lambda}''(1)-\frac{\beta_{n,\lambda}^3}{2}\int_{0}^{1}(1-t)^2\mathcal{D}_{n,\lambda}'''(1-t\beta_{n,\lambda}) \, dt.
    \end{align*}
    Since $\beta_{n,\lambda}\neq0$, we obtain
    $$\mathcal{D}_{n,\lambda}''(1)=\beta_{n,\lambda}\int_{0}^{1}(1-t)^2\mathcal{D}_{n,\lambda}'''(1-t\beta_{n,\lambda})\,dt.$$
    Now making appeal to \eqref{asymp:Dn3}, we infer 
    \begin{align*}
        \beta_{n,\lambda}\int_{0}^{1}(1-t)^2\mathcal{D}_{n,\lambda}'''(1-t\beta_{n,\lambda})\,dt\underset{n\to\infty}{=}2n\beta_{n,\lambda}\int_{0}^1(1-t)^2(1-t\beta_{n,\lambda})^{2n-2}\,dt+O(\beta_{n,\lambda}).
    \end{align*}
    In what follows, we denote
    $$c_{n,\lambda}\triangleq n\beta_{n,\lambda},$$
    which we expect to be of order $1$ as $n$ increases.
    Now, recalling \eqref{limit:Dn2}, we deduce
    \begin{equation}\label{lim:calIn}
        1-2\lambda^2I_0(\lambda)K_0(\lambda)I_1(\lambda)K_1(\lambda)=\lim_{n\to\infty}\mathcal{I}_{n,\lambda},\qquad\mathcal{I}_{n,\lambda}\triangleq 2c_{n,\lambda}\int_{0}^{1}(1-t)^2\left(1-t\frac{c_{n,\lambda}}{n}\right)^{2n-2}\,dt.
    \end{equation}
    $\blacktriangleright$ \textbf{Boundedness of $(c_{n,\lambda})_{n\in\mathbb{N}^*}$ :} Assume, in view of a contradiction, that the sequence $(c_{n,\lambda})_{n\in\mathbb{N}^*}$ is unbounded. Up to taking a subsequence, we can assume $c_{n,\lambda}\underset{n\to\infty}{\longrightarrow}\infty.$ Performing the linear change of variables $t=\frac{u}{2c_{n,\lambda}}$, we obtain
    $$\mathcal{I}_{n,\lambda}=\int_{0}^{2c_{n,\lambda}}\left(1-\frac{u}{2c_{n,\lambda}}\right)^2\left(1-\frac{u}{2n}\right)^{2n-2}\,du.$$
    By dominated convergence theorem, we deduce that
    $$\lim_{n\to\infty}\mathcal{I}_{n,\lambda}=\int_{0}^{\infty}e^{-u}\,du=1.$$
    This enters in contradiction with \eqref{lim:calIn} and \eqref{limit:Dn2}. Therefore, the sequence $(c_{n,\lambda})_{n\in\mathbb{N}^*}$ is bounded.\\
    $\blacktriangleright$ \textbf{First asymptotic :} Applying Bolzano-Weierstrass theorem, there exists $(c_{n_k,\lambda})_{k\in\mathbb{N}}$ a converging subsequence with limit $c_{\infty,\lambda}\geqslant0$. Still by dominated convergence, one gets
    $$1-2\lambda^2I_0(\lambda)K_0(\lambda)I_1(\lambda)K_1(\lambda)=\lim_{k\to\infty}\mathcal{I}_{n_k,\lambda}=2c_{\infty,\lambda}\int_{0}^{1}(1-t)^2e^{-2c_{\infty,\lambda}t}\,dt\triangleq\Phi(c_{\infty,\lambda}).$$
    By integrations by parts, we find
    \begin{equation}\label{def:Phi}
        \Phi(x)=\frac{2x^2-2x+1-e^{-2x}}{2x^2}\cdot
    \end{equation}
    One readily has
    $$\Phi'(x)=\frac{g(x)}{x^3},\qquad g(x)\triangleq(x-1)+(x+1)e^{-2x}.$$
    It is straightforward that $$g'(x)=1-(2x+1)e^{-2x}\qquad\textnormal{and}\qquad g''(x)=4xe^{-2x}.$$
    From this, we deduce that $g'$ is strictly increasing on $(0,\infty)$ and that $g'(0)=0.$ Therefore, $g'>0$ on $(0,\infty),$ which means that $g$ is strictly increasing on $(0,\infty).$ But $g(0)=0$, so $g$ is strictly positive on $(0,\infty).$ This implies in turn that $\Phi$ is strictly increasing on $(0,\infty).$ In addition, using the Taylor expansion of the exponential at $0$ and computing directly the limit at infinity, we find
    \begin{equation}\label{limitsPhi}
        \lim_{x\to0}\Phi(x)=0\qquad\textnormal{and}\qquad\lim_{x\to\infty}\Phi(x)=1.
    \end{equation} Invoking the bijection theorem, the application $\Phi:(0,\infty)\to(0,1)$ is a bijection and, recalling \eqref{limit:Dn2},
    $$c_{\infty,\lambda}=\Phi^{-1}\Big(1-2\lambda^2I_0(\lambda)K_0(\lambda)I_1(\lambda)K_1(\lambda)\Big)>0.$$
    The limit of the subsequence being uniquely determined, then the bounded sequence $(c_{n,\lambda})_{n\in\mathbb{N}^*}$ converges. Hence, we get the asymptotic 
    \begin{equation}\label{def:betalbd}
        \beta_{n,\lambda}\underset{n\to\infty}{\sim}\frac{\beta(\lambda)}{n},\qquad\beta(\lambda)\triangleq\lim_{n\to\infty}c_{n,\lambda}=\Phi^{-1}\Big(1-2\lambda^2I_0(\lambda)K_0(\lambda)I_1(\lambda)K_1(\lambda)\Big)>0.
    \end{equation}
    $\blacktriangleright$ \textbf{Properties of the function $\beta$ :} We can write $\beta=\Phi^{-1}\circ f$, where $f$ has been introduced in \eqref{deff}. By monotonicity of the function $\Phi$, \eqref{limitsPhi},  \eqref{decayf} and \eqref{limitsf}, the function $\beta:(0,\infty)\to(0,\infty)$ is strictly decreasing. Now, let us consider
    $$\varpi(x)\triangleq1-\Phi(x)=\frac{2x-1+e^{-2x}}{2x^2}\cdot$$
    Note that according to \eqref{def:betalbd},
    $$\varpi\big(\beta(\lambda)\big)=1-\Phi\big(\beta(\lambda)\big)=2 \lambda^2 I_0(\lambda) K_0(\lambda) I_1(\lambda) K_1(\lambda).$$
    Since $\Phi$ is strictly increasing, then $\varpi$ is strictly decreasing.  Furthermore, we notice that 
    \begin{align*}
        \varpi\left(\frac{1}{2x}\right) &= \frac{\frac{1}{x}-1+e^{-\frac{1}{x}}}{\frac{1}{2x^2}} = 2x^2 \left(\frac{1}{x}-1+e^{-\frac{1}{x}}\right) = 2x(1 - x) + 2x^2 e^{-\frac{1}{x}} > 2x(1-x).
    \end{align*}
    Recall from \eqref{wronskian}, the Wronskian identity
\begin{equation}\label{wron}
    \lambda I_0(\lambda)K_1(\lambda)+\lambda I_1(\lambda)K_0(\lambda)=1.
\end{equation}
    Therefore, thanks to \eqref{wron}, we get
    \begin{align*}
        \varpi\left(\frac{1}{2\lambda I_0(\lambda) K_1(\lambda)}\right) &> 2\lambda I_0(\lambda) K_1(\lambda)\big(1 - \lambda I_0(\lambda) K_1(\lambda)\big) \\
        & = 2 \lambda^2 I_0(\lambda) K_0(\lambda) I_1(\lambda) K_1(\lambda)\\
        &=\varpi\big(\beta(\lambda)\big).
    \end{align*}
    Since $\varpi$ is strictly decreasing, we have obtained \eqref{bndbeta}.\\
    $\blacktriangleright$ \textbf{Next order asymptotic :} From what precedes, we can write
    $$\mathcal{D}_n''(1)\underset{n\to\infty}{=}\Phi\big(\beta(\lambda)\big)+O\left(\frac{1}{n}\right)$$
    and
    $$\beta_{n,\lambda}\int_{0}^{1}(1-t)^2\mathcal{D}_{n,\lambda}'''(1-t\beta_{n,\lambda})\,dt\underset{n\to\infty}{=}\Phi(c_{n,\lambda})+O\left(\frac{1}{n}\right).$$
    We deduce that
    $$\Phi(c_{n,\lambda})-\Phi\big(\beta(\lambda)\big)\underset{n\to\infty}{=}O\left(\frac{1}{n}\right).$$
    By definition, for $N(\lambda)\in\mathbb{N}^*$ large enough, there exists $M>0$ such that, for any $n\in\mathbb{N}^*$, $n\geqslant N(\lambda)$,
    $$\left|\Phi(c_{n,\lambda})-\Phi\big(\beta(\lambda)\big)\right|\leqslant \frac{M}{n}\cdot$$
    Applying the mean value theorem, there exists $d_{n,\lambda}\in\big(c_{n,\lambda},\beta(\lambda)\big)$ such that
    $$\Phi'(d_{n,\lambda})=\frac{\Phi(c_{n,\lambda})-\Phi\big(\beta(\lambda)\big)}{c_{n,\lambda}-\beta(\lambda)}\cdot$$
    Since $\Phi'$ is strictly positive on $(0,\infty)$, we have in particular $\Phi'\big(\beta(\lambda)\big)\neq0.$ According to \eqref{def:betalbd}, we have $\displaystyle\lim_{n\to\infty}d_{n,\lambda}=\beta(\lambda).$ So, by continuity of $\Phi'$ at $\beta(\lambda)$, there exists $\mathtt{m}_0>0$ such that, for any $n\geqslant N(\lambda)$ with $N(\lambda)$ large enough, we get
    $$|\Phi'(d_{n,\lambda})|\geqslant \mathtt{m}_0>0.$$ We deduce that
    $$\left|c_{n,\lambda}-\beta(\lambda)\right|=\left|\frac{\Phi(c_{n,\lambda})-\Phi\big(\beta(\lambda)\big)}{\Phi'(d_{n,\lambda})}\right|\leqslant\frac{M}{\mathtt{m}_0n}\cdot$$
    This means
    $$c_{n,\lambda}\underset{n\to\infty}{=}\beta(\lambda)+O\left(\frac{1}{n}\right),\qquad\textnormal{i.e.}\qquad\beta_{n,\lambda}\underset{n\to\infty}{=}\frac{\beta(\lambda)}{n}+O\left(\frac{1}{n^2}\right).$$
    This concludes the proof of Lemma \ref{lem:asybn}.
\end{proof}

We shall now prove the uniqueness of the zero for $n$ large enough.
\begin{lem}\label{lem:asyan}
    There exists $\alpha(\lambda)>0$ such that the following asymptotic holds
    \begin{equation}\label{asy:Dnprimebn}
        \mathcal{D}_{n,\lambda}'(b_{n,\lambda})\underset{n\to\infty}{=}\frac{\alpha(\lambda)}{n}+O\left(\frac{1}{n^2}\right).
    \end{equation}
    In particular, asymptotically the set $Z_{n,\lambda}$ is a singleton, namely there is a unique zero
    $$Z_{n,\lambda}=\big\{b_{n,\lambda}\big\}.$$
    In addition, this zero is simple. Moreover, the sequence $(b_{n,\lambda})_{n}$ is asymptotically strictly increasing and converges to $1.$
\end{lem}
\begin{rem}
    As we shall see in the proof, the number $\alpha(\lambda)$ is explicitly given by the formula \eqref{def:alphalbd}.
\end{rem}
\begin{proof}
$\blacktriangleright$ \textbf{Asymptotic expansion :} Differentiating \eqref{decomp:Dn} with respect to $b$, we get
\begin{align*}
    \mathcal{D}_{n,\lambda}'(b)&=\mathcal{D}_{\infty,\lambda}'(b)+P_n(\lambda)\partial_b\big[\Lambda_1(\lambda,b)-b\Omega_n(\lambda b)\big]+P_n(\lambda b)\big[\Omega_n(\lambda)-b\Lambda_1(\lambda,b)\big]+\lambda bP_n'(\lambda b)\big[\Omega_n(\lambda)-b\Lambda_1(\lambda,b)\big]\\
    &\quad-bP_n(\lambda b)\partial_b\big[b\Lambda_1(\lambda,b)\big]+\partial_b\big[b\Lambda_n^2(\lambda,b)\big].
\end{align*}
First, from \eqref{asy:bn} and \eqref{diff2Dinfty}, we have
$$\mathcal{D}_{\infty,\lambda}'(b_{n,\lambda})\underset{n\to\infty}{=}\mathcal{D}_{\infty,\lambda}'(1)-\frac{\beta(\lambda)}{n}\mathcal{D}_{\infty,\lambda}''(1)+O\left(\frac{1}{n^2}\right)\underset{n\to\infty}{=}\frac{2\lambda^2\beta(\lambda)I_0(\lambda)K_0(\lambda)I_1(\lambda)K_1(\lambda)}{n}+O\left(\frac{1}{n^2}\right).$$
Now, recall the asymptotic
\begin{equation}\label{expand:Pn}
    P_n(x)\underset{n\to\infty}{=}\frac{1}{2n}-\frac{x^2}{4n^3}+O\left(\frac{1}{n^5}\right).
\end{equation}
Differentiating this asymptotic, we get, on any compact subset of $[0,\infty)$,
\begin{equation}\label{expand:Pnprima}
    P_n'(x)\underset{n\to\infty}{=}O\left(\frac{1}{n^3}\right).
\end{equation}
Then, using the asymptotic expansions \eqref{asy:bn} and \eqref{expand:Pn}, it results
\begin{align*}
    P_n(\lambda)\partial_b\big[\Lambda_1(\lambda,b)-b\Omega_n(\lambda b)\big]|_{b=b_{n,\lambda}}\underset{n\to\infty}{=}\frac{\lambda I_1(\lambda)K_0(\lambda)}{2n}+O\left(\frac{1}{n^2}\right).
\end{align*}
Besides, still using \eqref{asy:bn} and \eqref{expand:Pn}, we observe cancellation of the first order asymptotic and obtain
$$P_n(\lambda b_{n,\lambda})\big[\Omega_n(\lambda)-b_{n,\lambda}\Lambda_1(\lambda,b_{n,\lambda})\big]\underset{n\to\infty}{=}O\left(\frac{1}{n^2}\right).$$
Now, by virtue of \eqref{expand:Pnprima} and \eqref{asy:bn}, we infer
$$\lambda b_{n,\lambda}P_n'(\lambda b_{n,\lambda})[\Omega_{n}(\lambda)-b_{n,\lambda}\Lambda_{1}(\lambda,b_{n,\lambda})]\underset{n \to \infty}{=}O\left(\frac{1}{n^2}\right).$$
Next, we have, by \eqref{expand:Pn} and \eqref{asy:bn},
\begin{align*}
    b_{n,\lambda}P_n(\lambda b_{n,\lambda})\partial_b\big[b\Lambda_1(\lambda,b)\big]|_{b=b_{n,\lambda}}&\,\,\,=b_{n,\lambda}P_n(\lambda b_{n,\lambda})\lambda I_0(\lambda b_{n,\lambda})K_1(\lambda)\underset{n\to\infty}{=}\frac{\lambda I_0(\lambda)K_1(\lambda)}{2n}+O\left(\frac{1}{n^2}\right).
\end{align*}
According to \eqref{svi:diffbL} and \eqref{asy:bn}, we infer
$$\partial_b\big[b\Lambda_n^2(\lambda,b)\big]|_{b=b_{n,\lambda}}\underset{n\to\infty}{=}\frac{e^{-2\beta(\lambda)}}{2n}+O\left(\frac{1}{n^2}\right).$$
Combining the foregoing calculation yields
$$\mathcal{D}_{n,\lambda}'(b_{n,\lambda})\underset{n\to\infty}{=}\frac{\alpha(\lambda)}{n}+O\left(\frac{1}{n^2}\right),$$
with
$$\alpha(\lambda)\triangleq\frac{1}{2}\Big(4\lambda^2\beta(\lambda)I_0(\lambda)K_0(\lambda)I_1(\lambda)K_1(\lambda)+\lambda\big(I_1(\lambda)K_0(\lambda)-I_0(\lambda)K_1(\lambda)\big)+e^{-2\beta(\lambda)}\Big).$$
The $\alpha(\lambda)$ can be written
$$\alpha(\lambda)=\frac{1}{2}\Big(4\lambda^2\beta(\lambda)P_0(\lambda)P_1(\lambda)+\lambda P'_0(\lambda)+e^{-2\beta(\lambda)}\Big).$$
$\blacktriangleright$ \textbf{Positivity of $\alpha(\lambda)$ :} Combining \eqref{wron} with the definition \eqref{def:betalbd} and the relation \eqref{diffP0}, we infer
\begin{align*}
    2\Phi\big(\beta(\lambda)\big)-1&=1-4\lambda^2I_0(\lambda)K_0(\lambda)I_1(\lambda)K_1(\lambda)\\
    &=1-4\lambda I_1(\lambda)K_0(\lambda)\big(1-\lambda I_1(\lambda)K_0(\lambda)\big)\\
    &=\big(2\lambda I_1(\lambda)K_0(\lambda)-1\big)^2\\
    &=\lambda^2\big(I_1(\lambda)K_0(\lambda)-I_0(\lambda)K_1(\lambda)\big)^2\\
    &=\big(\lambda P_0'(\lambda)\big)^2.
\end{align*}
Since $P_0$ is strictly decreasing, then $P_0'<0$ and therefore
$$\lambda P_0'(\lambda)=-\sqrt{2\Phi\big(\beta(\lambda)\big)-1}.$$
Hence, still using the definition \eqref{def:betalbd}, we can write $\alpha(\lambda)$ as a closed formula in $\beta(\lambda)$ through
\begin{equation}\label{def:alphalbd}
    \alpha(\lambda)=\frac{1}{2}\Psi\big(\beta(\lambda)\big),\qquad\Psi(x)\triangleq2x\big(1-\Phi(x)\big)-\sqrt{2\Phi(x)-1}+e^{-2x}.
\end{equation}
Coming back to the definition of $\Phi$ in \eqref{def:Phi}, we have
\begin{align*}
    2x\big(1-\Phi(x)\big)=\frac{2x-1+e^{-2x}}{x}
\end{align*}
and 
\begin{align*}
    2\Phi(x)-1=\frac{x^2-2x+1-e^{-2x}}{x^2}=\frac{(x-1)^2-e^{-2x}}{x^2}\cdot
\end{align*}
Therefore,
$$x\Psi(x)=2x-1+(1+x)e^{-2x}-\sqrt{(x-1)^2-e^{-2x}}.$$
We have the equivalence
\begin{align}\label{sgn-cond-Psi}
    \Psi(x)>0\quad\Leftrightarrow\quad\big(2x-1+(1+x)e^{-2x}\big)^2>(x-1)^2-e^{-2x}.
\end{align}
Let us consider the following function
\begin{align*}
    h:\mathbb{R}\to\mathbb{R},\qquad h(x)\triangleq\big(2x-1+(1+x)e^{-2x}\big)^2+e^{-2x}-(x-1)^2.
\end{align*}
Notice that
$$h(x) = (x+1)^2 e^{-4x} + (4x^2 + 2x - 1) e^{-2x} + 3x^2-2x,$$
meaning that the expression of $h$ is quadratic in $e^{-2x}$. Hence, for any fixed $x > 0$, we define the function
$$Q_x : \mathbb{R} \to \mathbb{R}, \qquad Q_x(t) \triangleq (x+1)^2 t^2 + (4x^2+2x-1)t + 3x^2-2x.$$
As long as $(x+1)^2 > 0$, this is a parabola that is increasing for $t \geqslant t_*(x)$, with
$$t_*(x)\triangleq - \frac{4x^2+2x-1}{2(x+1)^2}\cdot$$
One has, for every $x>0$, 
$$e^{-2x}>\frac{1-x}{x+1}>t_*(x).$$
Indeed, on one side,
\begin{align*}
    \frac{1-x}{x+1} > t_*(x) & \iff \frac{1-x}{x+1} + \frac{4x^2+2x-1}{2(x+1)^2}> 0 \iff \frac{2(1-x)(1+x)+4x^2+2x-1}{2(x+1)^2} > 0 \\
    & \iff \frac{2-2x^2+4x^2+2x-1}{2(x+1)^2} > 0 \iff \frac{2x^2+2x+1}{2(x+1)^2} >0 
\end{align*}
and the last inequality is clearly true. On the other side, since for $x \geqslant 1$ it results
$$\frac{1-x}{x+1} \leqslant 0,$$
it is sufficient to show that, if $x \in (0, 1)$, then
$$e^{-2x}>\frac{1-x}{x+1}\cdot$$
One has, being both side positive,
\begin{align*}
    e^{-2x} > \frac{1-x}{x+1} & \iff \ln\left(e^{-2x}\right) > \ln\left(\frac{1-x}{x+1}\right) \iff -2x > \ln\left(\frac{1-x}{x+1}\right) \\
    & \iff \ln\left(\frac{1+x}{1-x}\right) - 2x > 0 \iff \ln(1+x) - \ln(1-x) -2x > 0.
\end{align*}
Let
$$k:[0,1)\to\mathbb{R},\qquad k(x)\triangleq\ln(1+x)-\ln(1-x)-2x.$$
It holds $k(0)=0$. Moreover,
$$k'(x)=\frac{1}{1+x}+\frac{1}{1-x}-2=\frac{2}{1-x^2}-2=\frac{2x^2}{1-x^2}>0.$$
Therefore, for every $x\in(0,1)$, the function $k$ is strictly increasing and $0=k(0)<k(x)$. This proves the inequality. Using the monotonicity of $Q_x$, we finally get, for every $x>0$,
\begin{align*}
    h(x)=Q_x\left(e^{-2x}\right)>Q_x\left(\frac{1-x}{x+1}\right)=\frac{2x^2}{1+x}>0.
\end{align*}
Thus, the condition \eqref{sgn-cond-Psi} is true for any $x>0.$ This allows to conclude that $\alpha(\lambda)>0.$\\
$\blacktriangleright$ \textbf{Uniqueness :} The expansion \eqref{asy:Dnprimebn} shows that asymptotically 
\begin{equation}\label{pos:Dnp}
    \mathcal{D}_{n,\lambda}'(b_{n,\lambda})>0.
\end{equation} Since this is true for any zero $b_{n,\lambda}$, this implies the asymptotic uniqueness. Indeed, take $n$ sufficiently large and assume in view of a contradiction that 
$$|Z_{n,\lambda}|\geqslant2.$$
Let us recall from Lemma \ref{lem:accu1} that the set $Z_{n,\lambda}$ is discrete and denote
$$b_{n,\lambda}^{[1]}\triangleq\min(Z_{n,\lambda})\qquad\textnormal{and}\qquad b_{n,\lambda}^{[2]}\triangleq\min\left(Z_{n,\lambda}\setminus\left\{b_{n,\lambda}^{[1]}\right\}\right).$$
By what precedes, we have 
\begin{equation}\label{sgnDpNb12}
    \mathcal{D}_{n,\lambda}'\left(b_{n,\lambda}^{[1]}\right)>0\qquad\textnormal{and}\qquad\mathcal{D}_{n,\lambda}'\left(b_{n,\lambda}^{[2]}\right)>0.
\end{equation}
By Taylor formula, the first condition in \eqref{sgnDpNb12} implies that $\mathcal{D}_{n,\lambda}>0$ locally near $b_{n,\lambda}^{[1]}$ inside $\left(b_{n,\lambda}^{[1]},b_{n,\lambda}^{[2]}\right).$ Since $b_{n,\lambda}^{[1]}$ and $b_{n,\lambda}^{[2]}$ are two consecutive zeros  and the function $\mathcal{D}_{n,\lambda}$ is continuous, by intermediate value theorem, we have
$$\forall\, b\in\left(b_{n,\lambda}^{[1]},b_{n,\lambda}^{[2]}\right),\quad\mathcal{D}_{n,\lambda}(b)>0.$$
From the mean value theorem, for any $b\in\left(b_{n,\lambda}^{[1]},b_{n,\lambda}^{[2]}\right)$ there exists $c\in\left(b,b_{n,\lambda}^{[2]}\right)$ such that
$$\mathcal{D}_{n,\lambda}'(c)=\frac{\mathcal{D}_{n,\lambda}(b)-\mathcal{D}_{n,\lambda}\left(b_{n,\lambda}^{[2]}\right)}{b-b_{n,\lambda}^{[2]}}<0.$$
Taking the limit $b\to b_{n,\lambda}^{[2]}$ in the previous expression, yields by continuity of $\mathcal{D}_{n,\lambda}'$,
$$\mathcal{D}_{n,\lambda}'\left(b_{n,\lambda}^{[2]}\right)\leqslant0.$$
This is a contradiction with the second condition in \eqref{sgnDpNb12}. Thus, $\left|Z_{n,\lambda}\right|=1.$\\
$\blacktriangleright$ \textbf{Last properties :} The simplicity of the zero is a consequence of \eqref{pos:Dnp}. Finally, the asymptotic \eqref{asy:bn} implies the desired monotonicity property. This ends the proof of Lemma \ref{lem:asyan}.
\end{proof}

With Proposition \ref{prop seq b} in hand, we can now state the bifurcation hypothesis allowing to prove Theorem \ref{thm stationary Vstates QGSW}-$(i)$ after application of the Crandall-Rabinowitz Theorem \ref{Crandall-Rabinowitz theorem}.

\begin{prop}\label{prop:CRb}
	Let $\lambda>0$, $\alpha\in(0,1)$ and $\mathbf{m}\in\mathbb{N}^*$ such that $\mathbf{m}\geqslant N(\lambda)$, with $N(\lambda)$ defined as in Proposition \ref{prop seq b}. Then the following assertions hold true.
	\begin{enumerate}[label=(\roman*)]
		\item There exists $r>0$ such that $G_{\lambda}:(0,1)\times B_{r,\mathbf{m}}^{1+\alpha}\times B_{r,\mathbf{m}}^{1+\alpha}\rightarrow Y_{\mathbf{m}}^{\alpha}$ is well-defined and of class $C^1.$
		\item The kernel $\ker\Big(DG_\lambda(b_{\mathbf{m},\lambda},0,0)\Big)$ is one-dimensional and generated by
		\begin{equation}\label{genker:b}
		    v_{\mathbf{m},\lambda}:\begin{array}[t]{rcl}
			\mathbb{T} & \rightarrow & \mathbb{C}^2\\
			w & \mapsto & \begin{pmatrix}
        -b_{\mathbf{m},\lambda}\Lambda_{\mathbf{m}}(\lambda,b_{\mathbf{m},\lambda})\\
			\Omega_{\mathbf{m}}(\lambda)-b_{\mathbf{m},\lambda}\Lambda_{1}(\lambda,b_{\mathbf{m},\lambda})
		\end{pmatrix}\overline{w}^{\mathbf{m}-1}.
		\end{array}
		\end{equation}
		\item The range $R\Big(DG_\lambda(b_{\mathbf{m},\lambda},0,0)\Big)$ is closed and of codimension one in $Y_{\mathbf{m}}^{\alpha}.$ It is given by the following orthogonality condition with respect to the scalar product in \eqref{scalar product},
		\begin{equation}\label{genim:b}
		    R\Big(DG_\lambda(b_{\mathbf{m},\lambda},0,0)\Big)=\mathtt{span}( y_{\mathbf{m},\lambda})^{\perp},\qquad y_{\mathbf{m},\lambda}\triangleq\begin{pmatrix}
        \Lambda_{\mathbf{m}}(\lambda,b_{\mathbf{m},\lambda})\\
			\Omega_{\mathbf{m}}(\lambda)-b_{\mathbf{m},\lambda}\Lambda_{1}(\lambda,b_{\mathbf{m},\lambda})
		\end{pmatrix}e_{\mathbf{m}}.
		\end{equation}
		\item Transversality condition :
		\begin{equation}\label{tras:b}
        \partial_{b}DG_\lambda(b_{\mathbf{m},\lambda},0,0)[v_{\mathbf{m},\lambda}]\not\in R\Big(DG_\lambda(b_{\mathbf{m},\lambda},0,0)\Big).
		\end{equation}
	\end{enumerate}
\end{prop}
\begin{proof}
	$(i)$ Follows from Proposition \ref{reg G and lin op}-1.\\
	$(ii)$ Let $(h_{1},h_{2})\in X_{\mathbf{m}}^{1+\alpha}.$  We write
	\begin{equation}\label{sym h1 h2 b}
		h_{1}(w)=\sum_{n=1}^{\infty}h_{1,n}\overline{w}^{n\mathbf{m}-1}\quad\mbox{ and }\quad h_{2}(w)=\sum_{n=1}^{\infty}h_{2,n}\overline{w}^{n\mathbf{m}-1}.
	\end{equation}
	Proposition \ref{reg G and lin op}-2 gives 
	\begin{equation}\label{DGlbd}
		DG_\lambda(b,0,0)[h_{1},h_{2}]=\sum_{n=1}^{\infty}n\mathbf{m}M_{n\mathbf{m}}(\lambda,b)\begin{pmatrix}
			h_{1,n}\\
			h_{2,n}
		\end{pmatrix}e_{n\mathbf{m}}.
	\end{equation}
	For $b=b_{\mathbf{m},\lambda},$ we have 
		$$\det\Big(M_{\mathbf{m}}\big(\lambda,b_{\mathbf{m},\lambda}\big)\Big)=0.$$
	Thus, the kernel of $DG_\lambda\big(b_{\mathbf{m},\lambda},0,0\big)$ is non trivial and it is one dimensional if and only if 
	$$\forall n\in\mathbb{N}^{*},\quad n\geqslant 2\quad\Rightarrow\quad\det\Big(M_{n\mathbf{m}}\big(\lambda,b_{\mathbf{m},\lambda}\big)\Big)\neq 0.$$
	The previous condition is satisfied in view the simplicity and monotonicity properties in Proposition \ref{prop seq b}. Hence, we have the equivalence 
	$$(h_{1},h_{2})\in\ker\Big(DG_\lambda\big(b_{\mathbf{m},\lambda},0,0\big)\Big)\Leftrightarrow\left\lbrace\begin{array}{l}
		\forall n\in\mathbb{N}^{*},\quad n\geqslant 2\,\Rightarrow \,h_{1,n}=0=h_{2,n},\\
		\begin{pmatrix}
			h_{1,1}\\
			h_{2,1}
		\end{pmatrix}\in\ker\Big(M_{\mathbf{m}}\big(\lambda,b_{\mathbf{m},\lambda}\big)\Big).
	\end{array}\right.$$
	Therefore, we can select as generator of $\ker\Big(DG_\lambda\big(b_{\mathbf{m},\lambda},0,0\big)\Big)$ the following function $$v_{\mathbf{m},\lambda}:\begin{array}[t]{rcl}
		\mathbb{T} & \rightarrow & \mathbb{C}^2\\
		w & \mapsto & \begin{pmatrix}
        -b_{\mathbf{m},\lambda}\Lambda_{\mathbf{m}}(\lambda,b_{\mathbf{m},\lambda})\\
			\Omega_{\mathbf{m}}(\lambda)-b_{\mathbf{m},\lambda}\Lambda_{1}(\lambda,b_{\mathbf{m},\lambda})
		\end{pmatrix}\overline{w}^{\mathbf{m}-1}.
	\end{array}$$
$(iii)$ First remark that the point $(ii)$ together with Proposition \ref{prop-fredholmness} imply that $R\Big(DG_\lambda(b_{\mathbf{m},\lambda},0,0)\Big)$ is closed and of codimension one in $Y_{\mathbf{m}}^{\alpha}.$ Now consider $f\in R\Big(DG_\lambda(b_{\mathbf{m},\lambda},0,0)\Big).$ There exists $h$ as in \eqref{sym h1 h2 b} such that
$$f=DG_\lambda(b_{\mathbf{m},\lambda},0,0)[h]=\sum_{n=1}^{\infty}n\mathbf{m}M_{n\mathbf{m}}(\lambda,b_{\mathbf{m},\lambda})\begin{pmatrix}
	h_{1, n}\\
	h_{2, n}
\end{pmatrix}e_{n\mathbf{m}}.$$
Now taking the scalar product \eqref{scalar product} of $f$ with $y_{\mathbf{m},\lambda}$ yields
\begin{align*}
	\big\langle f,y_{\mathbf{m},\lambda}\big\rangle&=\left\langle M_{\mathbf{m}}(\lambda,b_{\mathbf{m},\lambda})\begin{pmatrix}
		h_{1,1}\\
		h_{2,1}
	\end{pmatrix},\begin{pmatrix}
        \Lambda_{\mathbf{m}}(\lambda,b_{\mathbf{m},\lambda})\\
			\Omega_{\mathbf{m}}(\lambda)-b_{\mathbf{m},\lambda}\Lambda_{1}(\lambda,b_{\mathbf{m},\lambda})
		\end{pmatrix}\right\rangle_{\mathbb{R}^2}\\
	&=\left\langle\begin{pmatrix}
		h_{1,1}\\
		h_{2,1}
	\end{pmatrix},M_{\mathbf{m}}^{\top}(\lambda,b_{\mathbf{m},\lambda})\begin{pmatrix}
        \Lambda_{\mathbf{m}}(\lambda,b_{\mathbf{m},\lambda})\\
			\Omega_{\mathbf{m}}(\lambda)-b_{\mathbf{m},\lambda}\Lambda_{1}(\lambda,b_{\mathbf{m},\lambda})
		\end{pmatrix}\right\rangle_{\mathbb{R}^2}\\
	&=0,
\end{align*}
because by construction 
$$\begin{pmatrix}
        \Lambda_{\mathbf{m}}(\lambda,b_{\mathbf{m},\lambda})\\
			\Omega_{\mathbf{m}}(\lambda)-b_{\mathbf{m},\lambda}\Lambda_{1}(\lambda,b_{\mathbf{m},\lambda})
		\end{pmatrix}\in\ker\big(M_{\mathbf{m}}^{\top}(\lambda,b_{\mathbf{m},\lambda})\big).$$
This proves 
$$R\Big(DG_\lambda(b_{\mathbf{m},\lambda},0,0)\Big)\subset\mathtt{span}\big( y_{\mathbf{m},\lambda}\big)^{\perp}.$$
The equality is obtained by the codimension condition, by virtue of \cite[Lem. B.1]{R25}.\\
	$(iv)$ From \eqref{DGlbd}, one has
	$$\partial_{b}DG_\lambda(b,0,0)[h_{1},h_{2}]=\sum_{n=1}^{\infty}n\mathbf{m}\partial_{b}M_{n\mathbf{m}}(\lambda,b)\begin{pmatrix}
		h_{1,n}\\
		h_{2,n}
	\end{pmatrix}e_{n\mathbf{m}}.$$
In view of \eqref{scalar product}, \eqref{genker:b} and \eqref{genim:b}, the transversality condition \eqref{tras:b} is equivalent to 
\begin{equation}\label{def:Tmlbd}
    T_{\mathbf{m},\lambda}\triangleq\mathbf{m}\left\langle\partial_{b}M_{\mathbf{m}}(\lambda,b_{\mathbf{m},\lambda})\begin{pmatrix}
        -b_{\mathbf{m},\lambda}\Lambda_{\mathbf{m}}(\lambda,b_{\mathbf{m},\lambda})\\
			\Omega_{\mathbf{m}}(\lambda)-b_{\mathbf{m},\lambda}\Lambda_{1}(\lambda,b_{\mathbf{m},\lambda})
		\end{pmatrix},\begin{pmatrix}
        \Lambda_{\mathbf{m}}(\lambda,b_{\mathbf{m},\lambda})\\
			\Omega_{\mathbf{m}}(\lambda)-b_{\mathbf{m},\lambda}\Lambda_{1}(\lambda,b_{\mathbf{m},\lambda})
		\end{pmatrix}\right\rangle_{\mathbb{R}^2}\neq0.
\end{equation}
Differentiating \eqref{matrix Mn}, we infer
$$\partial_bM_n(\lambda,b)=\begin{pmatrix}
    -\partial_b\big(b\Lambda_1(\lambda,b)\big) & \partial_b\big(b\Lambda_n(\lambda,b)\big)\\
    -\partial_b\big(\Lambda_n(\lambda,b)\big) & \partial_b\Big(\Lambda_1(\lambda,b)-b\Omega_n(\lambda b)\Big)
\end{pmatrix}.$$
We study each term separately. First, from \eqref{Bessel derivatives}, we have
\begin{align*}
    \partial_b\Big(b\Lambda_1(\lambda,b)\Big)&=K_1(\lambda)\Big(I_1(\lambda b)+\lambda bI_1'(\lambda b)\Big)=\lambda bI_{0}(\lambda b)K_0(\lambda).
\end{align*}
Together with the asymptotic \eqref{asy:bn} and continuity of the modified Bessel functions, we obtain
\begin{align*}
    \partial_b\Big(b\Lambda_1(\lambda,b)\Big)|_{b=b_{\mathbf{m},\lambda}}\underset{\mathbf{m}\to\infty}{=}\lambda I_0(\lambda)K_1(\lambda)+O\left(\frac{1}{\mathbf{m}}\right).
\end{align*}
Next, one readily has from \eqref{svi:diffkLbd} that
$$\partial_b\big(\Lambda_n(\lambda,b)\big)\underset{n\to\infty}{=}\frac{b^{n-1}}{2}+O\left(\frac{b^{n-1}}{2n}\right).$$
Hence, together with \eqref{asy:bn}, we get
\begin{align*}
    \partial_b\big(\Lambda_{\mathbf{m}}(\lambda,b)\big)|_{b=b_{\mathbf{m},\lambda}}\underset{\mathbf{m}\to\infty}{=}\frac{e^{-\beta(\lambda)}}{2}+O\left(\frac{1}{\mathbf{m}}\right).
\end{align*}
Then, still from \eqref{svi:diffkLbd}, we have
$$\partial_b\big(b\Lambda_n(\lambda,b)\big)=\Lambda_n(\lambda,b)+b\partial_b\big(\Lambda(\lambda,b)\big)\underset{n\to\infty}{=}\frac{b^{n}}{2}+O\left(\frac{b^{n}}{2n}\right).$$
Thus, with \eqref{asy:bn}, we infer
\begin{align*}
    \partial_b\big(b\Lambda_n(\lambda,b)\big)|_{b=b_{\mathbf{m},\lambda}}\underset{\mathbf{m}\to\infty}{=}\frac{e^{-\beta(\lambda)}}{2}+O\left(\frac{1}{\mathbf{m}}\right).
\end{align*}
Finally,
\begin{align*}
    \partial_b\Big(\Lambda_1(\lambda,b)-b\Omega_n(\lambda b)\Big)=\lambda I_1'(\lambda b)K_1(\lambda)-\big(P_1(\lambda b)-P_n(\lambda b)\big)-\lambda b\big(P_1'(\lambda b)-P_n'(\lambda b)\big).
\end{align*}
Making appeal to \eqref{decay:InKnHR}, we find
\begin{align*}
    \partial_b\Big(\Lambda_1(\lambda,b)-b\Omega_n(\lambda b)\Big)\underset{n\to\infty}{=}\lambda I_1'(\lambda b)K_1(\lambda)-P_1(\lambda b)-\lambda b P_1'(\lambda b)+O\left(\frac{1}{n}\right).
\end{align*}
Using the asymptotic \eqref{asy:bn} and the relation \eqref{useful:diff}, we deduce that
\begin{align*}
    \partial_b\Big(\Lambda_1(\lambda,b)-b\Omega_{\mathbf{m}}(\lambda b)\Big)|_{b=b_{\mathbf{m},\lambda}}&\underset{\mathbf{m}\to\infty}{=}\lambda I_1'(\lambda)K_1(\lambda)-P_1(\lambda)-\lambda P_1'(\lambda)+O\left(\frac{1}{\mathbf{m}}\right)\\
    &\underset{\mathbf{m}\to\infty}{=}-I_1(\lambda)\Big(K_1(\lambda)+\lambda K_1'(\lambda)\Big)+O\left(\frac{1}{\mathbf{m}}\right)\\
    &\underset{\mathbf{m}\to\infty}{=}\lambda I_1(\lambda)K_0(\lambda)+O\left(\frac{1}{\mathbf{m}}\right).
\end{align*}
Consequently, we can write
$$\partial_{b}M_{\mathbf{m}}(\lambda,b_{\mathbf{m},\lambda})\underset{\mathbf{m}\to\infty}{=}\mathbb{M}(\lambda)+O\left(\frac{1}{\mathbf{m}}\right),$$
where
$$\mathbb{M}(\lambda)\triangleq\frac{1}{2}\begin{pmatrix}
    -2\lambda I_0(\lambda)K_1(\lambda) & e^{-\beta(\lambda)}\\
    -e^{-\beta(\lambda)} & 2\lambda I_1(\lambda)K_0(\lambda)
\end{pmatrix}.$$
Let us pass to find the asymptotic of $v_{\mathbf{m}, \lambda}$ and $y_{\mathbf{m}, \lambda}$, whose expressions are given in \eqref{genker:b} and \eqref{genim:b}, respectively. Recalling \eqref{svi:diffkLbd} and \eqref{asy:bn}, it results
    \begin{align*}
        -b_{\mathbf{m},\lambda}\Lambda_{\mathbf{m}}(\lambda,b_{\mathbf{m},\lambda})&\underset{\mathbf{m}\to\infty}{=}-\frac{\left(1-\frac{\beta(\lambda)}{\mathbf{m}}+O\left(\frac{1}{\mathbf{m}^2}\right)\right)^{\mathbf{m}+1}}{2\mathbf{m}}+O\left(\frac{1}{\mathbf{m}^2}\right)\\
        &\underset{\mathbf{m}\to\infty}{=}-\frac{e^{-\beta(\lambda)}}{2\mathbf{m}}+O\left(\frac{1}{\mathbf{m}^2}\right)
    \end{align*}
    and
    $$\Lambda_{\mathbf{m}}(\lambda,b_{\mathbf{m},\lambda})\underset{\mathbf{m}\to\infty}{=}\frac{e^{-\beta(\lambda)}}{2\mathbf{m}}+O\left(\frac{1}{\mathbf{m}^2}\right).$$
    Besides, by virtue of \eqref{def:Lbd and Omg}, \eqref{expand:Pn} and \eqref{asy:bn},
    \begin{align*}
        \Omega_{\mathbf{m}}(\lambda)-b_{\mathbf{m},\lambda}\Lambda_{1}(\lambda,b_{\mathbf{m},\lambda})&\underset{\mathbf{m}\to\infty}{=}I_1(\lambda)K_1(\lambda)-\frac{1}{2\mathbf{m}}-\left(1-\frac{\beta(\lambda)}{\mathbf{m}}\right)I_1(\lambda)K_1(\lambda)+\frac{\beta(\lambda)}{\mathbf{m}}\lambda I_1'(\lambda)K_1(\lambda)+O\left(\frac{1}{\mathbf{m}^2}\right)\\
        &\underset{\mathbf{m}\to\infty}{=}\frac{1}{\mathbf{m}}\left(\lambda\beta(\lambda)I_0(\lambda)K_1(\lambda)-\frac{1}{2}\right)+O\left(\frac{1}{\mathbf{m}^2}\right).
    \end{align*}
Thus, the transversality function admits the following asymptotic
$$T_{\mathbf{m},\lambda}\underset{\mathbf{m}\to\infty}{=}\frac{-1}{4\mathbf{m}}\left\langle\mathbb{M}(\lambda)\begin{pmatrix}
    e^{-\beta(\lambda)}\\
    1-2\lambda\beta(\lambda)I_0(\lambda)K_1(\lambda)
\end{pmatrix},\begin{pmatrix}
    e^{-\beta(\lambda)}\\
    2\lambda\beta(\lambda) I_0(\lambda)K_1(\lambda)-1
\end{pmatrix}\right\rangle+O\left(\frac{1}{\mathbf{m}^3}\right).$$
In what follows, we study the sign of
\begin{align*}
    \tau(\lambda)&\triangleq\left\langle\mathbb{M}(\lambda)\begin{pmatrix}
    e^{-\beta(\lambda)}\\
    1-2\lambda\beta(\lambda)I_0(\lambda)K_1(\lambda)
\end{pmatrix},\begin{pmatrix}
    e^{-\beta(\lambda)}\\
    2\lambda\beta(\lambda) I_0(\lambda)K_1(\lambda)-1
\end{pmatrix}\right\rangle\\
&=e^{-2\beta(\lambda)}\big(1-\lambda I_0(\lambda)K_1(\lambda)-2\lambda\beta(\lambda)I_0(\lambda)K_1(\lambda)\big)-\lambda I_1(\lambda)K_0(\lambda)\big(1-2\lambda\beta(\lambda)I_0(\lambda)K_1(\lambda)\big)^2.
\end{align*}
Using the wronskian identity \eqref{wron}, we can write
\begin{equation*}
    \tau(\lambda)=e^{-2\beta(\lambda)} \big(\lambda I_1(\lambda)K_0(\lambda)-2\lambda\beta(\lambda)I_0(\lambda)K_1(\lambda)\big)-\lambda I_1(\lambda)K_0(\lambda)\big(1-2\lambda\beta(\lambda)I_0(\lambda)K_1(\lambda)\big)^2.
\end{equation*}
Coming back to \eqref{diffP0} and using the fact that $P_0$ is strictly decreasing, we find 
$$I_1K_0<I_0K_1.$$
Inserting this into \eqref{wron}, we find
\begin{equation}\label{B12}
    0<\lambda I_1(\lambda)K_0(\lambda)<\frac{1}{2}<\lambda I_0(\lambda)K_1(\lambda)<1.
\end{equation}
Combining \eqref{B12} with \eqref{bndbeta}, we deduce that
$$\lambda I_1(\lambda) K_0(\lambda) - 2 \lambda \beta(\lambda) I_0(\lambda) K_1(\lambda) < \frac{1}{2} - 1 = - \frac{1}{2}<0.$$
Thus, 
$$\tau(\lambda)\leqslant e^{-2\beta(\lambda)} \big(\lambda I_1(\lambda)K_0(\lambda)-2\lambda\beta(\lambda)I_0(\lambda)K_1(\lambda)\big)<0.$$
This implies that for $\mathbf{m}$ large enough,
$$T_{\mathbf{m},\lambda}>0.$$
In particular, \eqref{def:Tmlbd} holds. This achieves the proof of Proposition \ref{prop:CRb}.
\end{proof}

\subsection{Bifurcation with the inverse Rossby radius} \label{sec:bif_lambda}
In this subsection, we assume that $b\in(0,1)$ is fixed and we denote
$$\mathcal{D}_{n,b}(\lambda)\triangleq\mathcal{D}_{n}(\lambda,b),\qquad\mathcal{D}_{\infty,b}(\lambda)\triangleq\mathcal{D}_{\infty}(\lambda,b),\qquad G_{b}(\lambda,f_1,f_2)\triangleq G(\lambda,b,f_1,f_2).$$
As before, in order to find a nontrivial kernel and then possible bifurcation points, we have to answer the following question:
\begin{center}
    {\it Do there exist zeros $\lambda$ of $\mathcal{D}_{n,b}$ in $(0,\infty)$ ?}
\end{center}
The following numerical simulations answer positively to this question and moreover the zero seems to be unique.

\begin{figure}[!ht]
\centering
	\begin{subfigure}[l]{0.3\textwidth}
		\includegraphics[width=\textwidth]{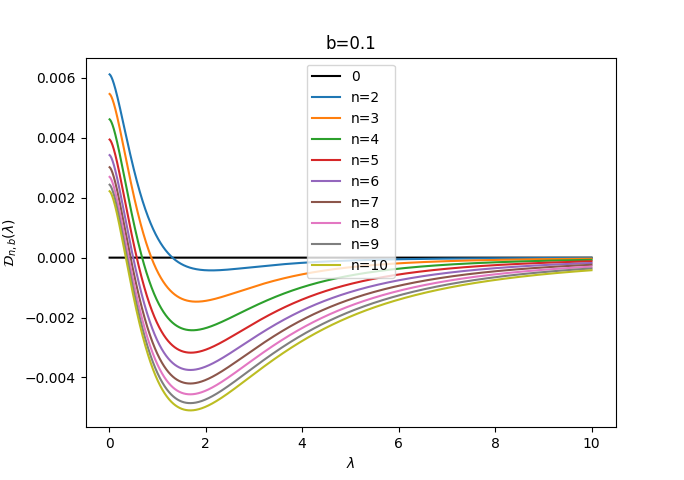}
	\end{subfigure}
	\begin{subfigure}[c]{0.3\textwidth}
		\includegraphics[width=\textwidth]{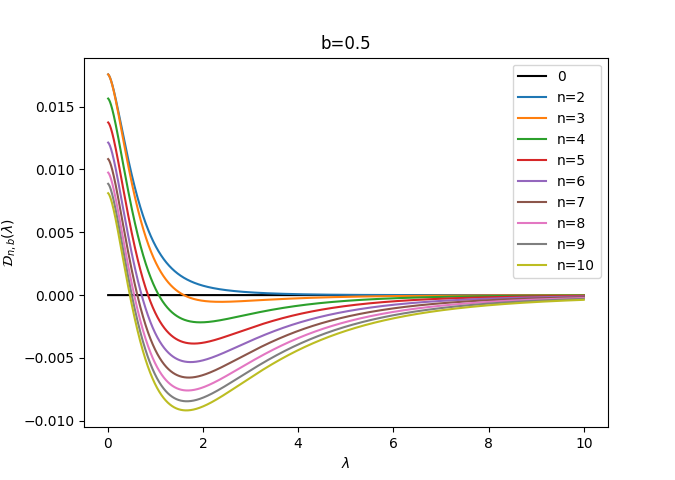}
	\end{subfigure}
	\begin{subfigure}[r]{0.3\textwidth}
		\includegraphics[width=\textwidth]{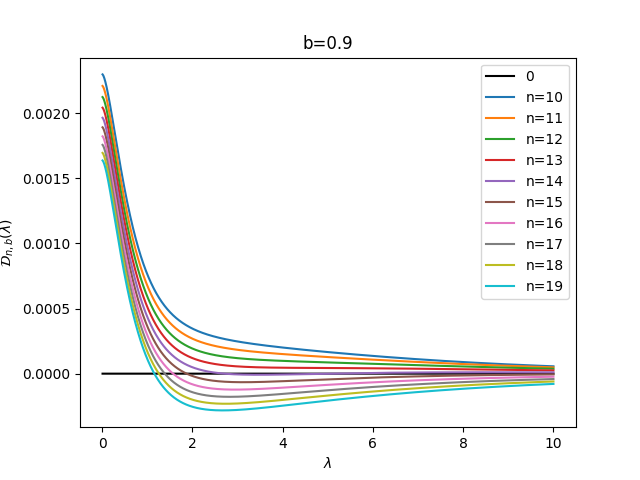}
	\end{subfigure}
	\caption{Graphs of $\lambda\mapsto\mathcal{D}_{n,b}(\lambda)$ for $b\in\{0.1,0.5,0.9\}$ and for different values of $n.$}\label{detlbd}
\end{figure}
\begin{figure}[!ht]
	\centering
	\begin{subfigure}[l]{0.3\textwidth}
		\includegraphics[width=\textwidth]{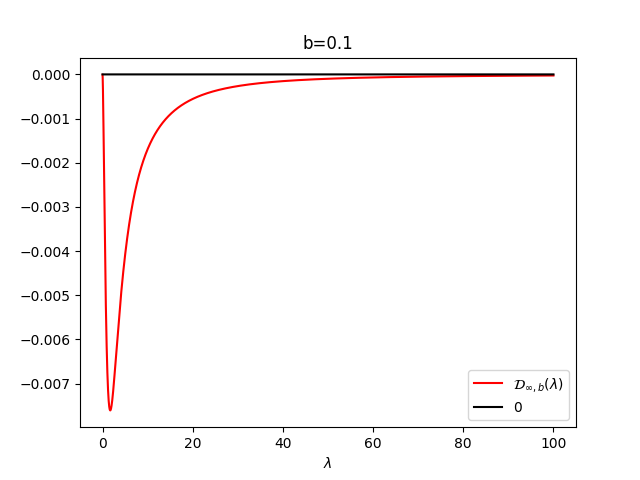}
	\end{subfigure}
	\begin{subfigure}[c]{0.3\textwidth}
		\includegraphics[width=\textwidth]{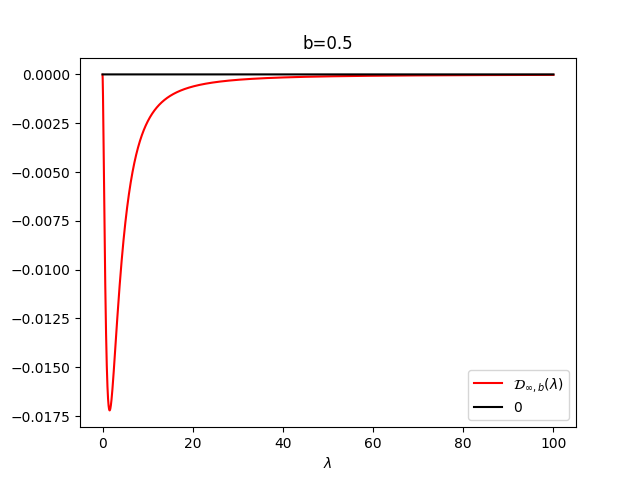}
	\end{subfigure}
	\begin{subfigure}[r]{0.3\textwidth}
		\includegraphics[width=\textwidth]{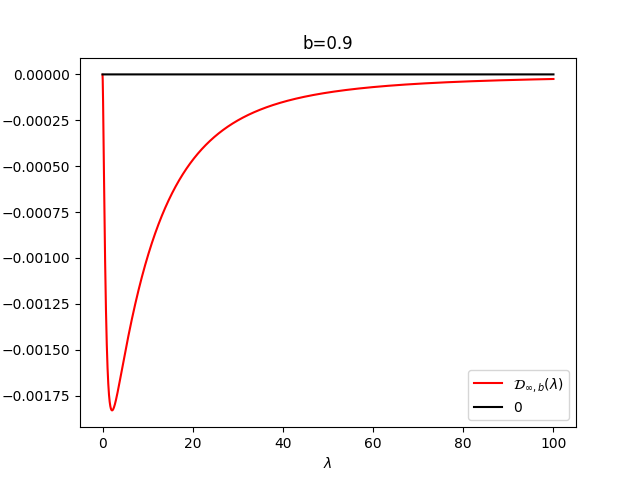}
	\end{subfigure}
	\caption{Graphs of the limiting profile $\lambda\mapsto\mathcal{D}_{\infty,b}(\lambda)$ for $b\in\{0.1,0.5,0.9\}.$}\label{detlimlbd}
\end{figure}

Our first goal is to prove the following analytical result, validating the above mentioned numerical observations in the asymptotic $n\to \infty$. 
\begin{prop}\label{prop seq lbd}
Let $b\in(0,1)$ and $\lambda_{\textnormal{max}}>0.$ There exists $N(b)\triangleq N(b,\lambda_{\textnormal{max}})\in\mathbb{N}^{*}$ such that, for all $n\in\mathbb{N}^{*}$, $n\geqslant N(b)$, the equation 
$$\mathcal{D}_{n,b}(\lambda)=0,$$
admits a unique solution $\lambda_{n,b}\in(0,\lambda_{\textnormal{max}}).$ Moreover, the sequence $(\lambda_{n,b})_{n\geqslant N(b)}$ is strictly decreasing and converges to $0$.
\end{prop}

\begin{rem}
    The threshold $\lambda_{\textnormal{max}}$ is introduced only to simplify the proof by avoiding the possible accumulation of zeros at infinity. However, the numerical simulations reported in Figure \ref{detlbd} suggest that such a phenomenon does not occur and therefore we strongly believe that the result remains valid on the entire half-line $(0,\infty)$. The main difficulty is to prove, uniformly with respect to $\lambda$ in an interval of the form $\left[\underline{\lambda},\infty\right)$, that $\mathcal{D}_{n,b}$ does not vanish for all $n$ large enough. Since this last property is satisfied by the limiting profile $\mathcal{D}_{\infty,b}$, one expects that the desired conclusion can be achieved by combining the uniform convergence \eqref{unifCV:D-2var} with asymptotic expansions of large arguments for $\mathcal{D}_{n,b},$ using \eqref{asymp large z}. In addition, as can be observed in Figure \ref{detlbd}, one may also try to show and exploit that the sequence of functions $(\mathcal{D}_{n,b})_n$ is asymptotically decreasing. This latter fact can easily be obtained pointwise in $\lambda$, but is hard to get uniformly for $\lambda$ large.
\end{rem}

\begin{proof}
	$\blacktriangleright$ \textbf{Existence :} Recall from \eqref{lim det} that
	$$\mathcal{D}_{n,b}(0)=\frac{b}{4n^2}\Big[n(1-b^2)-1+b^{2n}\Big].$$
	Therefore, we deduce the existence of $N_1(b)\in\mathbb{N}^*$ such that
	$$\forall n\in\mathbb{N}^*,\quad n\geqslant N_1(b)\quad\Rightarrow\quad\mathcal{D}_{n,b}(0)>0.$$
	Since \eqref{sign Dinfty} implies in particular that $\mathcal{D}_{\infty,b}(\lambda_{\textnormal{max}})<0,$ then by pointwise convergence we have that there exists $N_2(b)\triangleq N_2(b,\lambda_{\textnormal{max}})\in\mathbb{N}^*$ such that
	$$\forall n\in\mathbb{N}^*,\quad n\geqslant N_2(b)\quad\Rightarrow\quad\mathcal{D}_{n,b}(\lambda_{\textnormal{max}})<0.$$
	Setting $N_3(b)\triangleq\max(N_1(b),N_2(b))$, by virtue of the intermediate value theorem, we obtain for all integer $n\geqslant N_3(b)$ the existence of (at least) one solution $\lambda_{n,b}\in(0,\lambda_{\textnormal{max}})$ of the equation
	\begin{equation*}
		\mathcal{D}_{n,b}(\lambda_{n,b})=0.
	\end{equation*}
    Now, proceeding similarly to the proof of Lemma \ref{lem:accu1} (actually here is simpler), we get
    $$\lambda_{n,b}\underset{n\to\infty}{\longrightarrow}0.$$
    $\blacktriangleright$ \textbf{First order asymptotic :} Using the expression \eqref{definition of modified Bessel function of first kind}, we have the following asymptotics
$$I_n(x)\underset{x\to0}{=}\frac{\left(\frac{1}{2}x\right)^n}{\Gamma(n+1)}+\frac{\left(\frac{1}{2}x\right)^{n+2}}{\Gamma(n+2)}+O\left(x^{n+4}\right).$$
As for the second kind modified Bessel functions, we use the expression \eqref{power:Kn} to obtain
$$K_1(x)\underset{x\to0}{=}\frac{1}{x}+\frac{x}{2}\log(x)+O(x)$$
and, for $n\geqslant2,$
$$K_n(x)\underset{x\to0}{=}\frac{\Gamma(n)}{2\left(\frac{1}{2}x\right)^n}+O\left(\frac{1}{x^{n-2}}\right).$$
Combining the foregoing expansions, we obtain, for $b\in(0,1],$
\begin{equation*}
    \Lambda_1(\lambda,b)=I_1(\lambda b)K_1(\lambda)\underset{\lambda\to0}{=}\frac{b}{2}+\frac{b\lambda^2}{4}\log(\lambda)+O\left(\lambda^2\right)
\end{equation*}
and, for $n\geqslant2,$
\begin{equation}\label{svi0:Lbdn}
    \Lambda_n(\lambda,b)\underset{\lambda\to0}{=}b^n\left(\frac{1}{2n}+O\left(\lambda^2\right)\right).
\end{equation}
In what follows, we denote
$$A_{n,b}(\lambda)\triangleq\Omega_{n}(\lambda)-b\Lambda_{1}(\lambda,b)\qquad\textnormal{and}\qquad B_{n,b}(\lambda)\triangleq\Lambda_{1}(\lambda,b)-b\Omega_{n}(\lambda b).$$
Putting together the previous asymptotics, we find
\begin{equation}\label{svi0:An}
    A_{n,b}(\lambda)\underset{\lambda\to0}{=}\frac{1-b^2}{2}-\frac{1}{2n}+\frac{1-b^2}{4}\lambda^2\log(\lambda)+O(\lambda^2)
\end{equation}
and
\begin{equation}\label{svi0:Bn}
    B_{n,b}(\lambda)\underset{\lambda\to0}{=}\frac{b}{2n}+\frac{b(1-b^2)}{4}\lambda^2\log(\lambda)+O(\lambda^2).
\end{equation}
Gathering \eqref{def:Dn}, \eqref{svi0:Lbdn}, \eqref{svi0:An} and \eqref{svi0:Bn}, we infer
\begin{equation}\label{asympto:Dnl0}
    \mathcal{D}_{n,b}(\lambda)\underset{\lambda\to0}{=}\mathcal{D}_{n,b}(0)+\frac{b(1-b^2)^2}{8}\lambda^2\log(\lambda)+O(\lambda^2).
\end{equation}
Note that the second term of the previous expansion is independent of $n$ and the $O(\lambda^2)$ contains terms bounded in the variable $n.$ Hence, using in particular \eqref{lim det}, we have the equivalence
$$\lambda_{n,b}^2\log\left(\lambda_{n,b}^2\right)\underset{n\to\infty}{\sim}-\frac{16\mathcal{D}_{n,b}(0)}{b(1-b^2)^2}\underset{n\to\infty}{\sim}-\frac{4}{(1-b^2)n}\cdot$$
Let us consider
\begin{equation}\label{def:yn}
    y_{n,b}\triangleq\frac{1}{\lambda_{n,b}^2}\cdot
\end{equation}
Then,
\begin{equation}\label{equiv:fyn}
    \frac{\log(y_{n,b})}{y_{n,b}}\underset{n\to\infty}{\sim}\frac{4}{(1-b^2)n}\cdot
\end{equation}
Therefore,
\begin{equation}\label{svi:logy}
    \log(y_{n,b})\underset{n\to\infty}{=}\log(n)+\log\big(\log(y_{n,b})\big)+\log\left(\frac{1-b^2}{4}\right)+o(1).
\end{equation}
From \eqref{equiv:fyn}, we also get
$$y_{n,b}\underset{n\to\infty}{=}O\big(n\log(y_{n,b})\big),$$
from which we deduce
\begin{equation}\label{y:bnd:n2}
    y_{n,b}\underset{n\to\infty}{=}O(n^2).
\end{equation}
Since the $\log$ is increasing, we obtain from \eqref{y:bnd:n2} that
$$\log(y_{n,b})\underset{n\to\infty}{=}O\big(\log(n)\big)\qquad\textnormal{and}\qquad\log\big(\log(y_{n,b})\big)\underset{n\to\infty}{=}O\Big(\log\big(\log(n)\big)\Big)\underset{n\to\infty}{=}o\big(\log(n)\big).$$
Inserting this last information into \eqref{svi:logy}, we infer
\begin{equation}\label{equiv:logyn}
    \log(y_{n,b})\underset{n\to\infty}{=}\log(n)+o\big(\log(n)\big),\qquad\textnormal{i.e.}\qquad\log(y_{n,b})\underset{n\to\infty}{\sim}\log(n).
\end{equation}
Putting together \eqref{def:yn}, \eqref{equiv:fyn} and \eqref{equiv:logyn}, we finally get
\begin{equation}\label{svi:lbdn}
    \lambda_{n,b}^2\underset{n\to\infty}{\sim}\frac{4}{(1-b^2)n\log(n)},\qquad\textnormal{i.e.}\qquad\lambda_{n,b}\underset{n\to\infty}{\sim}\frac{2}{\sqrt{(1-b^2)n\log(n)}}\cdot
\end{equation}
The asymptotic \eqref{svi:lbdn} gives the desired asymptotic strict monotonicity of the sequence $(\lambda_{n,b})_n.$\\
$\blacktriangleright$ \textbf{Next order asymptotic :} Observe, by virtue of \eqref{lim det} and \eqref{svi:lbdn}, that
$$\mathcal{D}_{n,b}(0)\underset{n\to\infty}{=}\frac{b(1-b^2)}{4n}+O\left(\frac{1}{n^2}\right)\underset{n\to\infty}{=}\frac{b(1-b^2)}{4n}+O\left(\lambda_{n,b}^2\right).$$
Therefore, coming back to \eqref{asympto:Dnl0}, we infer
$$\lambda_{n,b}^2\log\left(\lambda_{n,b}^2\right)\underset{n\to\infty}{=}-\frac{4}{(1-b^2)n}+O(\lambda_{n,b}^2).$$
Dividing by $\lambda_{n,b}^2$, we find
$$\log\left(\lambda_{n,b}^2\right)\underset{n\to\infty}{=}-\frac{4}{(1-b^2)n\lambda_{n,b}^2}+O(1).$$
But, from \eqref{def:yn}, \eqref{svi:logy} and \eqref{equiv:logyn}, we have
\begin{equation}\label{loglbd2}
    \log\left(\lambda_{n,b}^2\right)\underset{n\to\infty}{=}-\log(n)-\log\big(\log(n)\big)+O(1).
\end{equation}
Combining the last two asymptotics and rearranging terms, we get
\begin{align*}
    \lambda_{n,b}^2&\underset{n\to\infty}{=}\frac{4}{(1-b^2)n}\left(\frac{1}{\log(n)+\log\big(\log(n)\big)+O(1)}\right)\\
    &\underset{n\to\infty}{=}\frac{4}{(1-b^2)n\log(n)}\left(1-\frac{\log\big(\log(n)\big)}{\log(n)}+O\left(\frac{1}{\log(n)}\right)\right).
\end{align*}
From this, we deduce that
\begin{equation} \label{asy_lambda}
    \lambda_{n,b}\underset{n\to\infty}{=}\frac{2}{\sqrt{(1-b^2)n\log(n)}}\left(1-\frac{\log\big(\log(n)\big)}{2\log(n)}+O\left(\frac{1}{\log(n)}\right)\right).
\end{equation}
	$\blacktriangleright$ \textbf{Uniqueness :} Differentiating the relation \eqref{asympto:Dnl0}, we find
    $$\mathcal{D}_{n,b}'(\lambda)\underset{\lambda\to0}{=}\frac{b(1-b^2)^2}{4}\lambda\log(\lambda) +O(\lambda).$$
    According to \eqref{loglbd2}, we have
    $$\log(\lambda_{n,b})\underset{n\to\infty}{=}-\frac{1}{2}\log(n)+O\Big(\log\big(\log(n)\big)\Big).$$
    Combined with \eqref{asy_lambda}, we get
    $$\lambda_{n,b}\log\left(\lambda_{n,b}\right)\underset{n\to\infty}{=}-\frac{1}{\sqrt{1-b^2}} \sqrt{\frac{\log(n)}{n}}+O\left(\frac{\log\big(\log(n)\big)}{\sqrt{n\log(n)}}\right).$$
    From this, it follows that
    $$\mathcal{D}_{n,b}'\left(\lambda_{n,b}\right)\underset{n\to\infty}{=}-\frac{b\left(1-b^2\right)^{\frac{3}{2}}}{4} \sqrt{\frac{\log(n)}{n}}+O\left(\frac{\log\big(\log(n)\big)}{\sqrt{n\log(n)}}\right).$$
    In particular, the following asymptotic holds
    $$\mathcal{D}_{n,b}'(\lambda_{n,b})\underset{n\to\infty}{\sim}-\frac{b(1-b^2)^{\frac{3}{2}}}{4}\sqrt{\frac{\log(n)}{n}}\cdot$$
    From this and the fact that $b\in(0,1)$, we have asymptotically
    $$\mathcal{D}_{n,b}'(\lambda_{n,b})<0.$$
    This being true for any zero $\lambda_{n,b}$, reasoning as in the proof of Lemma \ref{lem:asyan}, we conclude the uniqueness. This achieves the proof of Proposition \ref{prop seq lbd}.
\end{proof}
With Proposition \ref{prop seq lbd} in hand, we can now state the bifurcation hypothesis allowing to prove Theorem \ref{thm stationary Vstates QGSW}-$(ii)$ after application of the Crandall-Rabinowitz Theorem \ref{Crandall-Rabinowitz theorem}.
\begin{prop}\label{prop:CRlbd}
	Let $b\in(0,1)$, $\alpha\in(0,1)$, $\lambda_{\textnormal{max}}>0$ and $\mathbf{m}\in\mathbb{N}^*$ such that $\mathbf{m}\geqslant N(b),$ with $N(b)$ defined as in Proposition \ref{prop seq lbd}. Then the following assertions hold true.
	\begin{enumerate}[label=(\roman*)]
		\item There exists $r>0$ such that $G_{b}:(0,\lambda_{\textnormal{max}})\times B_{r,\mathbf{m}}^{1+\alpha}\times B_{r,\mathbf{m}}^{1+\alpha}\rightarrow Y_{\mathbf{m}}^{\alpha}$ is well-defined and of class $C^1.$
		\item The kernel $\ker\Big(DG_b(\lambda_{\mathbf{m},b},0,0)\Big)$ is one-dimensional and generated by
		\begin{equation}\label{genker:lbd}
		    v_{\mathbf{m},b}:\begin{array}[t]{rcl}
			\mathbb{T} & \rightarrow & \mathbb{C}^2\\
			w & \mapsto & \begin{pmatrix}
        -b\Lambda_{\mathbf{m}}(\lambda_{\mathbf{m},b},b)\\
			\Omega_{\mathbf{m}}(\lambda_{\mathbf{m},b})-b\Lambda_{1}(\lambda_{\mathbf{m},b},b)
		\end{pmatrix}\overline{w}^{\mathbf{m}-1}.
		\end{array}
		\end{equation}
	\item The range $R\Big(DG_b(\lambda_{\mathbf{m},b},0,0)\Big)$ is closed and of codimension one in $Y_{\mathbf{m}}^{\alpha}.$ It is given by the following orthogonality condition with respect to the scalar product in \eqref{scalar product},
	\begin{equation}\label{genim:lbd}
	    R\Big(DG_b(\lambda_{\mathbf{m},b},0,0)\Big)=\mathtt{span}(y_{\mathbf{m},b})^{\perp},\qquad y_{\mathbf{m},b}\triangleq\begin{pmatrix}
        \Lambda_{\mathbf{m}}(\lambda_{\mathbf{m},b},b)\\
			\Omega_{\mathbf{m}}(\lambda_{\mathbf{m},b})-b\Lambda_{1}(\lambda_{\mathbf{m},b},b)
		\end{pmatrix}e_{\mathbf{m}}.
	\end{equation}
	\item Transversality condition :
	\begin{equation}\label{tras:lbd}
	    \partial_{\lambda}DG_b(\lambda_{\mathbf{m},b},0,0)[v_{\mathbf{m},b}]\not\in R\Big(DG_b(\lambda_{\mathbf{m},b},0,0)\Big).
	\end{equation}
	\end{enumerate}
\end{prop}
\begin{proof}
	$(i)$ Follows from Proposition \ref{reg G and lin op}-1.\\
	$(ii)$ Let $(h_{1},h_{2})\in X_{\mathbf{m}}^{1+\alpha}.$  We write
	\begin{equation*}
		h_{1}(w)=\sum_{n=1}^{\infty}h_{1,n}\overline{w}^{n\mathbf{m}-1}\quad\mbox{ and }\quad h_{2}(w)=\sum_{n=1}^{\infty}h_{2,n}\overline{w}^{n\mathbf{m}-1}.
	\end{equation*}
	Proposition \ref{reg G and lin op}-2 gives 
	\begin{equation}\label{DGb}
		DG_b(\lambda,0,0)[h_{1},h_{2}]=\sum_{n=1}^{\infty}n\mathbf{m}M_{n\mathbf{m}}(\lambda,b)\begin{pmatrix}
			h_{1,n}\\
			h_{2,n}
		\end{pmatrix}e_{n\mathbf{m}}.
	\end{equation}
	For $\lambda=\lambda_{\mathbf{m},b},$ we have 
	$$\det\Big(M_{\mathbf{m}}\big(\lambda_{\mathbf{m},b},b\big)\Big)=0.$$
	Thus, the kernel of $DG_b\big(\lambda_{\mathbf{m},b},0,0\big)$ is non trivial and it is one dimensional if and only if 
	\begin{equation*}
		\forall\, n\in\mathbb{N}^{*},\quad n\geqslant 2\quad\Rightarrow\quad\det\Big(M_{n\mathbf{m}}\big(\lambda_{\mathbf{m},b},b\big)\Big)\neq 0.
	\end{equation*}
	The previous condition is satisfied in view of Proposition \ref{prop seq lbd}.
	Hence, we have the equivalence 
	$$(h_{1},h_{2})\in\ker\Big(DG_b\big(\lambda_{\mathbf{m},b},0,0\big)\Big)\Leftrightarrow\left\lbrace\begin{array}{l}
		\forall n\in\mathbb{N}^{*},\quad n\geqslant 2\,\Rightarrow \,h_{1,n}=0=h_{2,n}\\
		\begin{pmatrix}
			h_{1,1}\\
			h_{2,1}
		\end{pmatrix}\in\ker\Big(M_{\mathbf{m}}\big(\lambda_{\mathbf{m},b},b\big)\Big).
	\end{array}\right.$$
	Therefore, we can select as generator of $\ker\Big(DG_b\big(\lambda_{\mathbf{m},b},0,0\big)\Big)$ the following function
	$$v_{\mathbf{m},b}:\begin{array}[t]{rcl}
		\mathbb{T} & \rightarrow & \mathbb{C}^{2}\\
		w & \mapsto & \begin{pmatrix}
        -b\Lambda_{\mathbf{m}}(\lambda_{\mathbf{m},b},b)\\
			\Omega_{\mathbf{m}}(\lambda_{\mathbf{m},b})-b\Lambda_{1}(\lambda_{\mathbf{m},b},b)
		\end{pmatrix}\overline{w}^{\mathbf{m}-1}.
	\end{array}$$
\textbf{(iii)} First remark that the point $(ii)$ together with Proposition \ref{prop-fredholmness} imply that $R\Big(DG_b(\lambda_{\mathbf{m},b},0,0)\Big)$ is closed and of codimension one in $Y_{\mathbf{m}}^{\alpha}.$ Now consider $f\in R\Big(DG_b(\lambda_{\mathbf{m},b},0,0)\Big).$ There exists $h$ as in \eqref{sym h1 h2 b} such that
$$f=DG_b(\lambda_{\mathbf{m},b},0,0)[h]=\sum_{n=1}^{\infty}n\mathbf{m}M_{n\mathbf{m}}(\lambda_{\mathbf{m},b},b)\begin{pmatrix}
	a_n\\
	b_n
\end{pmatrix}e_{n\mathbf{m}}.$$
Now taking the scalar product \eqref{scalar product} of $f$ with $y_{\mathbf{m},b}$ yields
\begin{align*}
	\big\langle f,y_{\mathbf{m},b}\big\rangle&=\left\langle M_{\mathbf{m}}(\lambda_{\mathbf{m},b},b)\begin{pmatrix}
		h_{1,1}\\
		h_{2,1}
	\end{pmatrix},\begin{pmatrix}
        \Lambda_{\mathbf{m}}(\lambda_{\mathbf{m},b},b)\\
			\Omega_{\mathbf{m}}(\lambda_{\mathbf{m},b})-b\Lambda_{1}(\lambda_{\mathbf{m},b},b)
		\end{pmatrix}\right\rangle_{\mathbb{R}^2}\\
	&=\left\langle\begin{pmatrix}
		h_{1,1}\\
		h_{2,1}
	\end{pmatrix},M_{\mathbf{m}}^{\top}(\lambda_{\mathbf{m},b},b)\begin{pmatrix}
        \Lambda_{\mathbf{m}}(\lambda_{\mathbf{m},b},b)\\
			\Omega_{\mathbf{m}}(\lambda_{\mathbf{m},b})-b\Lambda_{1}(\lambda_{\mathbf{m},b},b)
		\end{pmatrix}\right\rangle_{\mathbb{R}^2}\\
	&=0,
\end{align*}
because by construction 
$$\begin{pmatrix}
        \Lambda_{\mathbf{m}}(\lambda_{\mathbf{m},b},b)\\
			\Omega_{\mathbf{m}}(\lambda_{\mathbf{m},b})-b\Lambda_{1}(\lambda_{\mathbf{m},b},b)
		\end{pmatrix}\in\ker\big(M_{\mathbf{m}}^{\top}(\lambda_{\mathbf{m},b},b)\big).$$
This proves
$$R\Big(DG_b(\lambda_{\mathbf{m},b},0,0)\Big)\subset\mathtt{span}(y_{\mathbf{m},b})^{\perp}.$$
The converse inclusion is obtained by the codimension condition, making appeal to \cite[Lem. B.1]{R25}.\\
$(iv)$ From \eqref{DGb}, one has
$$\partial_{\lambda}DG_b(\lambda,0,0)[h_{1},h_{2}]=\sum_{n=1}^{\infty}n\mathbf{m}\partial_{\lambda}M_{n\mathbf{m}}(\lambda,b)\begin{pmatrix}
	h_{1,n}\\
	h_{2,n}
\end{pmatrix}e_{n\mathbf{m}}.$$
In view of \eqref{scalar product}, \eqref{genker:lbd} and \eqref{genim:lbd}, the transversality condition \eqref{tras:lbd} is equivalent to 
\begin{equation}\label{def:Tmb}
    T_{\mathbf{m},b}\triangleq\mathbf{m}\left\langle\partial_{\lambda}M_{\mathbf{m}}(\lambda_{\mathbf{m},b},b)\begin{pmatrix}
        -b\Lambda_{\mathbf{m}}(\lambda_{\mathbf{m},b},b)\\
			\Omega_{\mathbf{m}}(\lambda_{\mathbf{m},b})-b\Lambda_{1}(\lambda_{\mathbf{m},b},b)
		\end{pmatrix},\begin{pmatrix}
        \Lambda_{\mathbf{m}}(\lambda_{\mathbf{m},b},b)\\
			\Omega_{\mathbf{m}}(\lambda_{\mathbf{m},b})-b\Lambda_{1}(\lambda_{\mathbf{m},b},b)
		\end{pmatrix}\right\rangle_{\mathbb{R}^2}\neq0.
\end{equation}
Differentiating \eqref{matrix Mn}, we infer
$$\partial_{\lambda}M_n(\lambda,b)=\begin{pmatrix}
    \partial_{\lambda}\big(\Omega_{n}(\lambda)-b\Lambda_1(\lambda,b)\big) & \partial_{\lambda}\big(b\Lambda_n(\lambda,b)\big)\\
    -\partial_{\lambda}\big(\Lambda_n(\lambda,b)\big) & \partial_{\lambda}\Big(\Lambda_1(\lambda,b)-b\Omega_n(\lambda b)\Big)
\end{pmatrix}.$$
We study each term separately. Differentiating \eqref{svi0:An}, we find
\begin{equation*}
    \partial_{\lambda}\big(\Omega_{n}(\lambda)-b\Lambda_1(\lambda,b)\big)\underset{\lambda\to0}{=}\frac{1-b^2}{2}\lambda\log(\lambda)+O(\lambda).
\end{equation*}
Combining this with \eqref{asy_lambda}, we obtain
\begin{equation*}
    \partial_{\lambda}\big(\Omega_{\mathbf{m}}(\lambda)-b\Lambda_1(\lambda,b)\big)|_{\lambda=\lambda_{\mathbf{m},b}}\underset{\mathbf{m}\to\infty}{=}-\sqrt{1-b^2}\sqrt{\frac{\log(\mathbf{m})}{\mathbf{m}}}+O\left(\frac{\log\big(\log(\mathbf{m})\big)}{\sqrt{\mathbf{m}\log(\mathbf{m})}}\right).
\end{equation*}
Differentiating \eqref{svi0:Bn}, we find
\begin{equation*}
    \partial_{\lambda}\big(\Lambda_1(\lambda,b)-b\Omega_{n}(\lambda b)\big)\underset{\lambda\to0}{=}\frac{b(1-b^2)}{2}\lambda\log(\lambda)+O(\lambda).
\end{equation*}
Combining this with \eqref{asy_lambda}, we get
\begin{equation*}
    \partial_{\lambda}\big(\Lambda_1(\lambda,b)-b\Omega_{\mathbf{m}}(\lambda b)\big)|_{\lambda=\lambda_{\mathbf{m},b}}\underset{\mathbf{m}\to\infty}{=}-b\sqrt{1-b^2}\sqrt{\frac{\log(\mathbf{m})}{\mathbf{m}}}+O\left(\frac{\log\big(\log(\mathbf{m})\big)}{\sqrt{\mathbf{m}\log(\mathbf{m})}}\right).
\end{equation*}
Besides, from \eqref{svi0:Lbdn}, we infer
\begin{align*}
    \partial_{\lambda}\big(\Lambda_{n}(\lambda,b)\big)\underset{\lambda\to0}{=}O\left(b^n\right).
\end{align*}
Therefore, in view of \eqref{asy_lambda}, we get
$$\partial_{\lambda}\big(\Lambda_{\mathbf{m}}(\lambda,b)\big)|_{\lambda=\lambda_{\mathbf{m},b}}\underset{\mathbf{m}\to\infty}{=}O\left(\frac{\log\big(\log(\mathbf{m})\big)}{\sqrt{\mathbf{m}\log(\mathbf{m})}}\right).$$
Putting together the foregoing calculations yields
$$\partial_{\lambda}M_{\mathbf{m}}(\lambda_{\mathbf{m},b},b)\underset{\mathbf{m}\to\infty}{=}-\sqrt{1-b^2}\sqrt{\frac{\log(\mathbf{m})}{\mathbf{m}}}\begin{pmatrix}
     1 & 0\\
     0 & b
\end{pmatrix}+O\left(\frac{\log\big(\log(\mathbf{m})\big)}{\sqrt{\mathbf{m}\log(\mathbf{m})}}\right).$$
Therefore, the transversality admits the asymptotic
\begin{align*}
    T_{\mathbf{m},b}&\underset{\mathbf{m}\to\infty}{=}-\frac{(1-b^2)^{\frac{5}{2}}}{4}\sqrt{\mathbf{m}\log(\mathbf{m})}\left\langle\begin{pmatrix}
     1 & 0\\
     0 & b
\end{pmatrix}\begin{pmatrix}
    0\\
    1
\end{pmatrix},\begin{pmatrix}
    0\\
    1
\end{pmatrix}\right\rangle_{\mathbb{R}^2}+O\left(\frac{\sqrt{\mathbf{m}}\log\big(\log(\mathbf{m})\big)}{\sqrt{\log(\mathbf{m})}}\right).\\
&\underset{\mathbf{m}\to\infty}{=}-b\frac{(1-b^2)^{\frac{5}{2}}}{4}\sqrt{\mathbf{m}\log(\mathbf{m})}+O\left(\frac{\sqrt{\mathbf{m}}\log\big(\log(\mathbf{m})\big)}{\sqrt{\log(\mathbf{m})}}\right).
\end{align*}
Since $b\in(0,1)$, the previous asymptotic implies
$$T_{\mathbf{m},b}\underset{\mathbf{m}\to\infty}{\longrightarrow}-\infty.$$
In particular, \eqref{def:Tmb} holds. This concludes the proof of Proposition \ref{prop:CRlbd}.
\end{proof}

\section{Radial symmetry for simply-connected V-states} \label{section_rigidity}
In this section, we study necessary radial symmetry properties of simply-connected V-states, thus proving Theorem \ref{thm sym Vstates QGSW}. We set 
$$\mathcal{K}_{\lambda}(x)\triangleq-\frac{1}{2\pi}K_0(\lambda|x|).$$
We start with the following observation : A simply-connected bounded domain $D_0\subset\mathbb{R}^2$ generates a V-state with angular velocity $\Omega\in\mathbb{R}$ for $(QGSW)_{\lambda}$ equations if and only if there exists $C\in\mathbb{R}$ such that
\begin{equation}\label{sta eq}
	\forall x\in\partial D_0,\quad\mathbf{1}_{D_0}\ast \mathcal{K}_{\lambda}(x)-\frac{\Omega}{2}|x|^2=C.
\end{equation}
The equation \eqref{sta eq} is obtained by integrating the contour dynamics equation \eqref{CDE} over the arclength. For more details, we refer to \cite{HMV13}.

\begin{proof}[Proof of Theorem \ref{thm sym Vstates QGSW}]
	$(i)$ Follows directly from \cite[Thm. 4.2]{GPSY20}. Indeed,
	if we denote
	$$g(r)\triangleq-\frac{1}{2\pi}K_0(\lambda r)$$
	then $g(|x|)=\mathcal{K}_{\lambda}(x)$ and the formulae \eqref{Bessel derivatives}, \eqref{symmetry Bessel} and \eqref{e-K1} imply 
	$$0<g'(r)=\frac{\lambda}{2\pi}K_1(\lambda r)\leqslant\frac{1}{2\pi r}\cdot$$
	$(ii)$ The proof is inspired from \cite[Thm. 5.1]{GPSY20}. There exists $x_0\in\partial D_{0}$ such that $R=|x_0|.$ Assume by contradiction that $D_0\neq R\cdot\mathbb{D}.$ We have $D_0\subset R\cdot\mathbb{D}$ and we set $U\triangleq (R\cdot\mathbb{D})\setminus D_0.$ We have
	\begin{equation}\label{1UKlbd}
		\forall x\in\partial D_0,\quad\mathbf{1}_{U}\ast\mathcal{K}_{\lambda}(x)=\mathbf{1}_{R\cdot\mathbb{D}}\ast\mathcal{K}_{\lambda}(x)-\frac{\Omega}{2}|x|^2-C.
	\end{equation}
	$\blacktriangleright$ On one hand, for $x\in D_0$, we have
	\begin{align*}
		\nabla\big(\mathbf{1}_U\ast\mathcal{K}_{\lambda}\big)(x_0)\cdot x_0&=\frac{\lambda}{2\pi}\int_{U}\frac{K_{1}(\lambda|x_0-y|)}{|x_0-y|}(x_0-y)\cdot x_0\, dy.
	\end{align*}
For any $y\in U\subset R\cdot\mathbb{D}$, Cauchy-Schwarz inequality implies
$$(x_0-y)\cdot x_0=|x_0|^2-y\cdot x_0\geqslant R(R-|y|)>0.$$
Using \eqref{symmetry Bessel} one deduces that
\begin{equation}\label{plus nbl1K}
	\nabla\big(\mathbf{1}_U\ast\mathcal{K}_{\lambda}\big)(x_0)\cdot x_0>0.
\end{equation}
$\blacktriangleright$ On the other hand, straightforward computations based on \eqref{Bessel derivatives} yield, for all $i\in\{1,2\}$,
$$\partial_{x_i}^2\big(\mathbf{1}_{U}\ast\mathcal{K}_{\lambda}\big)(x)=\frac{\lambda}{2\pi}\left[\int_{U}\frac{K_1(\lambda|x-y|)}{|x-y|}\,dy+\lambda\int_{U}\frac{K_1'(\lambda|x-y|)(x_i-y_i)^2}{|x-y|^2}\,dy-\int_{U}\frac{K_1(\lambda|x-y|)(x_i-y_i)^2}{|x-y|^3}\,dy\right].$$
Therefore,
\begin{align*}
	\Delta\big(\mathbf{1}_{U}\ast\mathcal{K}_{\lambda}\big)(x)&=\left(\partial_{x_1}^2+\partial_{x_2}^2\right)\big(\mathbf{1}_{U}\ast\mathcal{K}_{\lambda}\big)(x)\\
	&=\frac{\lambda}{2\pi}\left[\int_{U}\frac{K_1(\lambda|x-y|)}{|x-y|} \, dy+\lambda\int_{U}K_1'(\lambda|x-y|)\, dy\right].
\end{align*}
Using \eqref{useful:diff} and the positivity of $K_0$, we get
\begin{equation}\label{Lap neg}
	\forall x\in D_0,\quad\Delta\big(\mathbf{1}_{U}\ast\mathcal{K}_{\lambda}\big)(x)=-\frac{\lambda^2}{2\pi}\int_{U}|x-y|K_0(\lambda|x-y|)\,dy<0.
\end{equation} 
Fix $x=\widetilde{r}e^{\ii\eta}\in R\cdot\mathbb{D}.$ Then, by polar change of variable, we infer
\begin{align*}
	\big(\mathbf{1}_{R\cdot\mathbb{D}}\ast\mathcal{K}_{\lambda}\big)(x)&=-\frac{1}{2\pi}\int_{R\cdot\mathbb{D}}K_0(\lambda|x-y|)\,dy\\
	&=-\frac{1}{2\pi}\int_{\widetilde{r}\cdot\mathbb{D}}K_0(\lambda|x-y|)\,dy-\frac{1}{2\pi}\int_{(R\cdot\mathbb{D})\setminus (\widetilde{r}\cdot\mathbb{D})}K_0(\lambda|x-y|)\,dy\\
	&=-\frac{1}{2\pi}\int_{0}^{\widetilde{r}}\int_{0}^{2\pi}K_0\big(\lambda\big|\widetilde{r}e^{\ii\eta}-re^{\ii\theta}\big|\big)\,r\,dr\,d\theta-\frac{1}{2\pi}\int_{\widetilde{r}}^{R}\int_{0}^{2\pi}K_0\big(\lambda\big|\widetilde{r}e^{\ii\eta}-re^{\ii\theta}\big|\big)\,r\,dr\,d\theta.
\end{align*}
Now, by using \eqref{Beltrami's summation formula} and \eqref{Bessel and anti-derivatives}, we obtain
\begin{align*}
	\big(\mathbf{1}_{R\cdot\mathbb{D}}\ast\mathcal{K}_{\lambda}\big)(x)&=-K_0(\lambda\widetilde{r})\int_{0}^{\widetilde{r}}rI_0(\lambda r)\,dr-I_0(\lambda\widetilde{r})\int_{\widetilde{r}}^{R}rK_0(\lambda r)\,dr\\
	&=-\frac{K_0(\lambda\widetilde{r})}{\lambda^2}\int_{0}^{\lambda\widetilde{r}}uI_0(u)\,du-\frac{I_0(\lambda\widetilde{r})}{\lambda^2}\int_{\lambda\widetilde{r}}^{\lambda R}uK_0(u)\,du\\
	&=-\frac{K_0(\lambda\widetilde{r})}{\lambda^2}\big[uI_1(u)\big]_{0}^{\lambda\widetilde{r}}-\frac{I_0(\lambda\widetilde{r})}{\lambda^2}\big[-uK_1(u)\big]_{\lambda\widetilde{r}}^{\lambda R}\\
	&=\frac{1}{\lambda}\Big[RK_1(\lambda R)I_0(\lambda\widetilde{r})-\widetilde{r}\Big(K_0(\lambda\widetilde{r})I_1(\lambda\widetilde{r})+I_0(\lambda\widetilde{r})K_1(\lambda\widetilde{r})\Big)\Big].
\end{align*}
The identity \eqref{wronskian} finally gives
$$\big(\mathbf{1}_{R\cdot\mathbb{D}}\ast\mathcal{K}_{\lambda}\big)(x)=\frac{R}{\lambda}K_1(\lambda R)I_0(\lambda|x|)-\frac{1}{\lambda^2}\cdot$$
For $\Omega\geqslant I_{1}(\lambda R)K_1(\lambda R)$, we get
\begin{align*}
	\frac{d}{d|x|}\Big(\mathbf{1}_{R\cdot\mathbb{D}}\ast\mathcal{K}_{\lambda}-\frac{\Omega}{2}|x|^2\Big)&=RK_{1}(\lambda R)I_{1}(\lambda|x|)-\Omega|x|\\
	&\leqslant R\lambda|x|K_1(\lambda R)\Big(\frac{I_{1}(\lambda|x|)}{\lambda|x|}-\frac{I_{1}(\lambda R)}{\lambda R}\Big).
\end{align*}
Now, we easily deduce from \eqref{definition of modified Bessel function of first kind} that the function $x\mapsto\frac{I_1(x)}{x}$ is increasing on $(0,\infty).$ Therefore, the function $x\mapsto \mathbf{1}_{R\cdot\mathbb{D}}\ast\mathcal{K}_{\lambda}(x)-\frac{\Omega}{2}|x|^2$ is non-increasing with respect to $|x|$ on $(0,R]$ provided that $\Omega\geqslant I_{1}(\lambda R)K_1(\lambda R).$ Consequently, coming back to \eqref{1UKlbd}, we see that the minimum of $\mathbf{1}_U\ast\mathcal{K}_{\lambda}$ on $\partial D_0$ is achieved at $x_0.$ Putting together this remark with \eqref{Lap neg} leads to apply the maximum principle to $\mathbf{1}_U\ast\mathcal{K}_{\lambda}.$  This implies that the minimum of $\mathbf{1}_U\ast\mathcal{K}_{\lambda}$ in $\overline{D_0}$ is also achieved at $x_0$ and therefore
\begin{equation}\label{minus nbl1K}
	\nabla\big(\mathbf{1}_U\ast\mathcal{K}_{\lambda}\big)(x_0)\cdot x_0\leqslant 0.
\end{equation}
Combining \eqref{plus nbl1K} and \eqref{minus nbl1K} leads to a contradiction. This achieves the proof of Theorem \ref{thm sym Vstates QGSW}.
\end{proof}

The next result provides a quantification on the V-states deformation with respect to the unit disc.
\begin{cor}\label{cor mea diffsym}
	Let $\lambda>0.$ Let $D_0$ be a simply-connected bounded domain of class $C^1$ in $\mathbb{R}^2$ with area $\pi$ and generating a V-state with angular velocity $\Omega\in\big(0,I_1(\lambda)K_{1}(\lambda)\big).$ Then, the symmetric difference $D_0\,\Delta\, \mathbb{D}$ satisfies the following estimate
	$$\big|D_0\,\Delta\, \mathbb{D}\big|\leqslant 2\pi\Big(\big(f_\lambda^{-1}(\Omega)\big)^2-1\Big),$$
	where $f_\lambda^{-1}$ denotes the inverse function of $f_{\lambda} : x\in(0,\infty)\mapsto I_1(\lambda x)K_1(\lambda x)\in(0,\tfrac{1}{2}).$ Note that the right hand-side tends to $0$ when $\Omega\nearrow I_1(\lambda)K_1(\lambda).$
\end{cor}
\begin{proof}
	Recall from Appendix \ref{appendix Bessel} that the function $f_{\lambda}$ is strictly decreasing on $(0,\infty)$ with values in $(0,\tfrac{1}{2}).$ Therefore, it is invertible and its inverse $f_{\lambda}^{-1}$ is also strictly decreasing. We denote $R\triangleq \displaystyle\max_{x\in D_0}|x|.$ If $D_0$ generates a non-trivial V-state with angular velocity $\Omega,$ then Theorem \ref{thm sym Vstates QGSW} implies 
	$$\Omega\leqslant I_{1}(\lambda R)K_1(\lambda R),\qquad\textnormal{i.e.}\qquad R\leqslant f_{\lambda}^{-1}(\Omega).$$
	Consequently,
	$$D_0\subset f_\lambda^{-1}(\Omega)\cdot\mathbb{D}.$$
	Hence,
	$$\big|D_0\,\Delta\, \mathbb{D}\big|= 2\big|D_0\setminus \mathbb{D}\big|\leqslant 2\Big|\Big(f_\lambda^{-1}(\Omega)\cdot\mathbb{D}\Big)\setminus \mathbb{D}\Big|\leqslant 2\pi\Big(\big(f_\lambda^{-1}(\Omega)\big)^2-1\Big).$$
	This ends the proof of Corollary \ref{cor mea diffsym}.
\end{proof}
Finally, we conclude this section by mentioning a rescaling result for simply-connected QGSW V-states.
\begin{lem}\label{lem:dilation}
    Let $\alpha,\lambda>0$. Let $D_0\subset\mathbb{R}^2$ be a simply-connected bounded set. Then, the following statements are equivalent
    \begin{enumerate}[label=(\roman*)]
        \item $D_0$ generates a V-state with angular velocity $\Omega$ for $(QGSW)_{\lambda};$
        \item $\alpha\cdot D_0$ generates a V-state with angular velocity $\Omega$ for $(QGSW)_{\frac{\lambda}{\alpha}}$.
    \end{enumerate}
\end{lem}
\begin{proof} It is sufficient to prove one implication, the converse one being obtained by the converse dilation. Fix $x=\alpha x'\in\partial(\alpha\cdot D_0)$ with $x'\in \partial D_0.$ By change of variable, one gets
\begin{align*}
	\mathbf{1}_{\alpha\cdot D_0}\ast \mathcal{K}_{\frac{\lambda}{\alpha}}(x)&=-\frac{1}{2\pi}\int_{\alpha\cdot D_0}K_{0}\big(\tfrac{\lambda}{\alpha}|x-y|\big)\,dy\\
	&=-\frac{\alpha^2}{2\pi}\int_{D_0}K_{0}\big(\lambda|x'-y'|\big)\,dy'\\
	&=\alpha^2\mathbf{1}_{D_0}\ast \mathcal{K}_{\lambda}(x').
\end{align*}
Therefore, one deduces from \eqref{sta eq} that
$$\mathbf{1}_{\alpha\cdot D_0}\ast \mathcal{K}_{\frac{\lambda}{\alpha}}(x)-\frac{\Omega}{2}|x|^2=\alpha^2\Big(\mathbf{1}_{D_0}\ast \mathcal{K}_{\lambda}(x')-\frac{\Omega}{2}|x'|^2\Big)=\alpha^2C.$$
This ends the proof of Lemma \ref{lem:dilation}.
\end{proof}
\appendix
\section{Formular on modified Bessel functions}\label{appendix Bessel}
We put here all the information needed about modified Bessel functions. The starting references are \cite{AS64,W95}. We first define the Bessel functions of order $\nu\in\mathbb{C}$ by 
$$J_{\nu}(z)=\sum_{m=0}^{+\infty}\frac{(-1)^{m}\left(\frac{z}{2}\right)^{\nu +2m}}{m!\Gamma(\nu+m+1)},\mbox{ }\quad|\mbox{arg}(z)|<\pi.$$
Notice that when $\nu\in\mathbb{N}$ we have the integral representation (see \cite[p. 115]{Leb65})
\begin{equation*}
	J_\nu(x)=\frac{1}{\pi}\int_0^\pi\cos\big(x \sin \theta-\nu \theta\big) d\theta.
\end{equation*}
Then, we define Bessel functions of imaginary argument by
\begin{equation}\label{definition of modified Bessel function of first kind}
	I_{\nu}(z)=\sum_{m=0}^{+\infty}\frac{\left(\frac{z}{2}\right)^{\nu+2m}}{m!\Gamma(\nu+m+1)},\mbox{ }\quad|\mbox{arg}(z)|<\pi
\end{equation}
and $$K_{\nu}(z)=\frac{\pi}{2}\cdot\frac{I_{-\nu}(z)-I_{\nu}(z)}{\sin(\nu\pi)},\mbox{ }\quad\nu\in\mathbb{C}\backslash\mathbb{Z},\mbox{ }\quad|\mbox{arg}(z)|<\pi.$$
For $n\in\mathbb{Z},$ we define $K_{n}(z)=\displaystyle\lim_{\nu\rightarrow n}K_{\nu}(z).$
We give now useful properties of modified Bessel functions.\\ 
\textbf{Symmetry and positivity properties} (see \cite[p. 375]{AS64}) \textbf{:}
\begin{equation}\label{symmetry Bessel}
	\forall n\in\mathbb{N},\quad\forall\lambda\in(0,\infty),\quad I_{-n}(\lambda)=I_{n}(\lambda)\in(0,\infty)\quad\mbox{ and }\quad K_{-n}(\lambda)=K_{n}(\lambda)\in(0,\infty).
\end{equation}
\textbf{Derivatives and Anti-derivatives} (see \cite[p. 376]{AS64}) \textbf{:}\\
If we set $\mathcal{Z}_{\nu}(z)=I_{\nu}(z)$ or $e^{i\nu\pi}K_{\nu}(z)$, then, for all $\nu\in\mathbb{R}$, we have 
\begin{equation}\label{Bessel derivatives}
	\mathcal{Z}_{\nu}'(z)=\mathcal{Z}_{\nu-1}(z)-\frac{\nu}{z}\mathcal{Z}_{\nu}(z)=\mathcal{Z}_{\nu+1}(z)+\frac{\nu}{z}\mathcal{Z}_{\nu}(z),
\end{equation}
for any $k\in\mathbb{N},$
\begin{equation}\label{deriv:iterate}
    \mathcal{Z}_{\nu}^{(k)}=\frac{1}{2^k}\sum_{p=0}^{k}\binom{k}{p}\mathcal{Z}_{\nu-k+2p}
\end{equation}
and
\begin{equation}\label{Bessel and anti-derivatives}
	\int z^{\nu+1}\mathcal{Z}_{\nu}(z)\,dz=z^{\nu+1}\mathcal{Z}_{\nu+1}(z).
\end{equation}
\textbf{Power series extension for $K_{n}$} (see \cite[p. 375]{AS64}) \textbf{:}\\
\begin{equation}\label{power:Kn}
    \begin{aligned}
	K_{n}(z)=&\frac{1}{2}\left(\frac{z}{2}\right)^{-n}\sum_{k=0}^{n-1}\frac{(n-k-1)!}{k!}\left(\frac{-z^2}{4}\right)^{k}+(-1)^{n+1}\ln\left(\frac{z}{2}\right)I_{n}(z)\\
	&+\frac{1}{2}\left(\frac{-z}{2}\right)^{n}\sum_{k=0}^{+\infty}\left(\psi(k+1)+\psi(n+k+1)\right)\frac{\left(\frac{z^{2}}{4}\right)^{k}}{k!(n+k)!},
\end{aligned}
\end{equation}
where 
$$\psi(1)=-\boldsymbol{\gamma} \textnormal{ (Euler's constant)}\quad \textnormal{and}\quad  \forall m\in\mathbb{N}^{*},\,\,\psi(m+1)=\displaystyle\sum_{k=1}^{m}\frac{1}{k}-\boldsymbol{\gamma}.$$
In particular, 
\begin{equation*}
	K_{0}(z)=-\log\left(\frac{z}{2}\right)I_{0}(z)+\sum_{m=0}^{+\infty}\frac{\left(\frac{z}{2}\right)^{2m}}{(m!)^{2}}\psi(m+1).
\end{equation*}
So $K_{0}$ behaves like a logarithm at $0.$\\

\noindent\textbf{Decay property for the product $I_{\nu}K_{\nu}$} (see \cite{B09} and \cite{DHR19}) \textbf{:}\\ The application $(\lambda,\nu)\mapsto I_{\nu}(\lambda)K_{\nu}(\lambda)$ is strictly decreasing in each variable $(\lambda,\nu)\in(0,\infty) \times (0, \infty).$\\

\noindent\textbf{Wronskian :}
\begin{equation}\label{wronskian}
	I_{\nu}'(z)K_{\nu}(z)-I_{\nu}(z)K_{\nu}'(z)=I_{\nu}(z)K_{\nu+1}(z)+I_{\nu+1}(z)K_{\nu}(z)=\frac{1}{z}\cdot
\end{equation}

\noindent\textbf{Beltrami's summation formula} (see \cite[p. 361]{W95}) \textbf{:}
Let $0<b<a.$ Then
\begin{equation}\label{Beltrami's summation formula}
	\forall \theta\in\mathbb{R},\quad K_{0}\left(\sqrt{a^{2}+b^{2}-2ab\cos(\theta)}\right)=\sum_{m=-\infty}^{\infty}I_{m}(b)K_{m}(a)\cos(m\theta).
\end{equation}
\textbf{Asymptotic expansion of small argument} (see \cite[p. 375]{AS64}) \textbf{:}
\begin{equation}\label{asymptotic expansion of small argument}
	\forall n\in\mathbb{N}^{*},\quad I_{n}(\lambda)\underset{\lambda\rightarrow 0}{\sim}\frac{\left(\frac{1}{2}\lambda\right)^{n}}{\Gamma(n+1)}\quad\mbox{ and }\quad K_{n}(\lambda)\underset{\lambda\rightarrow 0}{\sim}\frac{\Gamma(n)}{2\left(\frac{1}{2}\lambda\right)^{n}}\cdot
\end{equation} 
\textbf{Asymptotic expansion of large argument} (see \cite[p. 377-378]{AS64}) \textbf{:} 
\begin{equation}\label{asymp large z}
	\forall n\in\mathbb{N}^{*},\quad I_{n}(\lambda)\underset{\lambda\rightarrow +\infty}{\sim}\frac{e^{\lambda}}{\sqrt{2\pi \lambda}}\quad\mbox{ and }\quad K_{n}(\lambda)\underset{\lambda\rightarrow +\infty}{\sim}\sqrt{\frac{\pi}{2\lambda}}e^{-\lambda}.
\end{equation} 
\textbf{Asymptotic expansion of high order} (see \cite{HS11}) \textbf{:} For every $n\in\mathbb{N}^*$,
    \begin{equation}\label{asymp:large-order-In}
        I_n(\lambda) \underset{n \to \infty}{\sim} \frac{\left(\frac{1}{2} \lambda\right)^n}{\Gamma(n+1)} \sum_{m = 0}^{+ \infty} \frac{b_m(\lambda)}{n^m}, 
    \end{equation}
    and
    \begin{equation}\label{asymp:large-order-Kn}
        K_n(\lambda) \underset{n \to \infty}{\sim} \frac{1}{2} \frac{\Gamma(n)}{\left(\frac{1}{2} \lambda \right)^n} \sum_{m = 0}^{+ \infty} (-1)^m \frac{b_m(\lambda)}{n^m},
    \end{equation}
where for each $m\in\mathbb{N}$, $b_{m}(\lambda)$ is a polynomial of degree $m$ in $\lambda^{2}$ defined by 
$$b_{0}(\lambda)=1\quad\mbox{ and }\quad\forall m\in\mathbb{N}^{*},\,\,b_{m}(\lambda)=\sum_{k=1}^{m}(-1)^{m-k}\frac{S(m,k)}{k!}\left(\frac{\lambda^{2}}{4}\right)^{k}$$
and the $S(m,k)$ are Stirling numbers of second kind defined recursively by 
$$\forall (m,k)\in(\mathbb{N}^{*})^{2},\quad S(m,k)=S(m-1,k-1)+kS(m-1,k),$$
with $$S(0,0)=1,\quad\forall m\in\mathbb{N}^{*},\,\,S(m,1)=1\quad\mbox{ and }\quad S(m,0)=0\quad\mbox{ and if }m<k\mbox{ then }S(m,k)=0.$$
In particular,
\begin{equation*}
	\forall(\lambda,b)\in(0,\infty)\times(0,1],\quad I_{n}(\lambda b)K_{n}(\lambda)\underset{n\rightarrow+\infty}{\sim}\frac{b^{n}}{2n}\left(\sum_{m=0}^{+\infty}\frac{b_{m}(\lambda b)}{n^{m}}\right)\left(\sum_{m=0}^{+\infty}(-1)^{m}\frac{b_{m}(\lambda)}{n^{m}}\right).
\end{equation*}

\section{Technical theoretical results}
In this appendix, we gather general theoretical results used along this work.
\subsection{Bifurcation from simple eigenvalues}
We recall here the classical local bifurcation theorem due to Crandall and Rabinowitz \cite{CR71}. It gives sufficient condition to construct a one parameter family (curve) of non-trivial solutions emerging from a line of trivial solutions under some non-degeneracy conditions. The theorem reads as follows.
\begin{theo}[Crandall-Rabinowitz]\label{Crandall-Rabinowitz theorem}
	Let $X$ and $Y$ be two Banach spaces. Let $V$ be a neighborhood of $0$ in $X$ and $I\subset\mathbb{R}$ be an interval containing a number $\mathfrak{p}_0$. Consider
	$$F:\begin{array}[t]{rcl}
		I\times V & \rightarrow  & Y\\
		(\mathfrak{p},x) & \mapsto & F(\mathfrak{p},x)
	\end{array}$$
	 a function of class $C^{1}$ with the following properties 
	\begin{enumerate}[label=(\roman*)]
		\item (Trivial solutions) $\forall\,\mathfrak{p}\in I,\,\,F(\mathfrak{p},0)=0.$
		\item (Regularity) $\partial_{\mathfrak{p}}F$, $\partial_{x}F$ and $\partial_{\mathfrak{p}}\partial_{x}F$ exist and are continuous.
		\item (Fredholm property) $\partial_{x}F(\mathfrak{p}_0,0)$ is a Fredholm operator with index $0$ and $\ker\big(\partial_{x}F(\mathfrak{p}_0,0)\big)=\langle x_{0}\rangle.$
		\item (Transversality assumption) $\partial_{\mathfrak{p}}\partial_{x}F(\mathfrak{p}_0,0)[1,x_{0}]\not\in R\big(\partial_{x}F(\mathfrak{p}_0,0)\big).$
	\end{enumerate}
	If $\chi$ denotes any complement of $\ker\big(\partial_{x}F(\mathfrak{p}_0,0)\big)$ in $X$, then there exist 
	\begin{enumerate}[label=\textbullet]
		\item a neighborhood $U$ of $(\mathfrak{p}_0,0)$ in $I\times V,$
		\item an interval $(-a,a)$, for some $a>0$,
		\item continuous functions
		$\psi:(-a,a)\rightarrow \mathbb{R}$ and $\phi:(-a,a)\rightarrow\chi$ satisfying $\psi(0)=\mathfrak{p}_0$ and $\phi(0)=0$
	\end{enumerate} 
	such that the set of the zeros of $F$ in $U$ can be described as the following two curves intersecting at $(\mathfrak{p}_0,0)$
	$$\Big\{(\mathfrak{p},x)\in U\quad\textnormal{s.t.}\quad F(\mathfrak{p},x)=0\Big\}=\Big\{\big(\psi(s),sx_{0}+s\phi(s)\big)\quad\textnormal{s.t.}\quad|s|<a\Big\}\cup\Big\{(\mathfrak{p},0)\in U\Big\}.$$
\end{theo}
\subsection{Potential theory}
Here, we deal with singular integral operators of the type
\begin{equation}\label{operator-T}
\forall\, w\in\T,\quad\mathcal{T}(f)(w)=\fint_{\T} K(w,\xi)f(\xi)\, d\xi,
\end{equation}
where $K:\T\times \T\rightarrow\C$ is smooth off the diagonal $\{w=\xi\}$. The next result focuses on the smoothness of the last operator, whose proof can be found in \cite{HH15}. 

\begin{lem}\label{Lem-pottheory}
Let $0\leqslant \alpha<1$ and consider $K:\T\times\T\rightarrow \C$ with the following properties. There exists $C_0>0$ such that
\begin{itemize}
\item[(i)] $K$ is measurable on $\T\times\T\setminus\{(w,w), w\in\T\}$ and 
$$
\forall\, w\neq \xi\in\T,\quad|K(w,\xi)|\leqslant \frac{C_0}{|w-\xi|^\alpha}\cdot
$$
\item[(ii)] For each $\xi\in\T$, $w\mapsto K(w,\xi)$ is differentiable in $\T\setminus\{\xi\}$ and 
$$
\forall\, w\neq \xi\in\T,\quad|\partial_w K(w,\xi)|\leqslant \frac{C_0}{|w-\xi|^{1+\alpha}}\cdot
$$
\end{itemize}
Then,
\begin{enumerate}
\item The operator $\mathcal{T}$ defined by \eqref{operator-T} is continuous from $L^{\infty}(\T)$ to $C^{1-\alpha}(\T)$. More precisely, there exists a constant $C_{\alpha} > 0$, depending only on $\alpha$, such that
$$
\|\mathcal{T}(f)\|_{1-\alpha}\leqslant C_{\alpha}C_0\|f\|_{L^{\infty}}.
$$
\item For $\alpha=0$, the operator $\mathcal{T}$ is continuous from $L^{\infty}(\T)$ to $C^{\beta}(\T)$, for any $0<\beta<1$. That is, there exists a constant $C_{\beta} > 0$, depending only on $\beta$, such that
$$
\|\mathcal{T}(f)\|_{\beta}\leqslant C_{\beta}C_0\|f\|_{L^{\infty}}.
$$
\end{enumerate}
\end{lem}

\vspace{1cm}
\noindent\textbf{Vittorio Baroncini}\\ 
Departamento de Análisis Matemático and IMUS, Universidad de Sevilla, 41012 Seville, Spain.\\
E-mail address: vbaroncini@us.es\\

\noindent\textbf{Claudia Garc\'ia}\\
Departamento de Matematica Aplicada and Research Unit ``Modeling Nature" (MNat), Facultad de Ciencias, Universidad de Granada, 18071 Granada, Spain.\\
E-mail address:  claudiagarcia@ugr.es\\

\noindent \textbf{Emeric Roulley}\\
Dipartimento di Matematica ``Federigo Enriques'', Università degli Studi di Milano, Via Cesare Saldini 50, 20133 Milano,
Italy.\\
E-mail address : emeric.roulley@unimi.it
\end{document}